\documentclass[11pt]{amsart}
\usepackage[english]{babel}
\usepackage{amsfonts,latexsym,amsthm,amssymb,graphicx}
\usepackage{mathabx}
\usepackage[all]{xy}
\usepackage[usenames]{color}
\usepackage{tikz}
\usetikzlibrary{shapes}
\usepackage{xypic}
\usetikzlibrary{decorations.pathreplacing}
\usepackage{amsmath}
\usepackage{graphicx}
\usepackage[enableskew]{youngtab}
\usepackage[utf8]{inputenc}
\usepackage[english]{babel}
\usepackage{tikz}
\usetikzlibrary{arrows}
\usepackage{float}
\usepackage{amscd}
\usepackage{pstricks}
\usepackage{tableau}
\usepackage{geometry}
\usepackage[alphabetic]{amsrefs}
\usepackage{enumerate}
\usepackage{pdfsync}
\usepackage[colorlinks=true, linkcolor=blue, citecolor=blue, urlcolor=blue]{hyperref}

\newcommand{\C}{{\mathbb C}}
\newcommand{\Z}{{\mathbb Z}}
\newcommand{\bP}{{\mathbb P}}

\newcommand{\IG}{\mathrm{IG}}
\DeclareMathOperator{\Sp}{Sp}
\DeclareMathOperator{\GL}{GL}

\newcommand{\cO}{{\mathcal O}}
\newcommand{\cX}{{\mathcal X}}
\renewcommand{\Xc}{Z} 
\newcommand{\Xo}{X^\circ} 
\newcommand{\Y}{Y} 
\newcommand{\U}{U} 
\newcommand{\Tsmall}{T_{2n}}
\newcommand{\Wodd}{W^{odd}}
\newcommand{\Xev}{X^{ev}}
\newcommand{\Xt}{\widetilde{X}}

\newcommand{\Wsmall}{W_{2n}}
\newcommand{\Wlarge}{W}
\newcommand{\Wo}{W^\circ} 
\newcommand{\Oo}{O^\circ}
\newcommand{\OY}{O_\Y}
\newcommand{\OZ}{O_\Xc}

\newcommand{\wtE}{\widetilde{E}}
\newcommand{\wto}{\widetilde{\omega}}
\DeclareMathOperator{\Span}{Span}
\DeclareMathOperator{\e}{\mathbf{e}}
\DeclareMathOperator{\diag}{diag}

\newcommand{\Mb}{\overline{\mathcal{M}}}

\DeclareMathOperator{\ev}{ev}

\newcommand{\Part}{\operatorname{Part}}
\newcommand{\BC}{\operatorname{BC}}
\newcommand{\BCodd}{\operatorname{BC}^{odd}}
\newcommand{\BKT}{\operatorname{BKT}}
\newcommand{\BKTodd}{\operatorname{BKT}^{odd}}
\newcommand{\Comp}{\operatorname{Comp}}

\newcommand{\QH}{\mathrm{QH}}

\DeclareFontFamily{OT1}{rsfs}{}
\DeclareFontShape{OT1}{rsfs}{n}{it}{<-> rsfs10}{}
\DeclareMathAlphabet{\mathscr}{OT1}{rsfs}{n}{it}

\CompileMatrices

\newtheorem{thm}{Theorem}[section]
\newtheorem{lemma}[thm]{Lemma}
\newtheorem{cor}[thm]{Corollary}
\newtheorem{prop}[thm]{Proposition}

{\theoremstyle{definition} \newtheorem{defn}[thm]{Definition}}
{\theoremstyle{remark} \newtheorem{remark}[thm]{Remark}
\newtheorem{example}[thm]{Example}}

\begin{document}

\title[Curve Neighborhoods]{Curve neighborhoods of Schubert Varieties in the odd symplectic Grassmannian}

\author{Clelia Pech}

\address{
School of Mathematics, Statistics, and Actuarial Science,
Sibson Building,
University of Kent,
Canterbury,
CT2 7FS, UK
}
\email{c.m.a.pech@kent.ac.uk}

\author{Ryan M. Shifler}

\address{
Department of Mathematical Sciences,
Henson Science Hall, 
Salisbury University,
Salisbury, MD 21801
}
\email{rmshifler@salisbury.edu}


\begin{abstract}
Let $\IG(k,2n+1)$ be the odd symplectic Grassmannian. It is a quasi-ho\-mo\-ge\-neous space with homogeneous-like behavior. A very limited description of curve neighborhoods of Schubert varieties  in $\IG(k,2n+1)$ was used by Mihalcea and the second named author to prove an (equivariant) quantum Chevalley rule. In this paper we give a full description of the irreducible components of curve neighborhoods in terms of the Hecke product of (appropriate) Weyl group elements, $k$-strict partitions, and BC-partitions. The latter set of partitions respect the Bruhat order with inclusions. Our approach follows the philosophy of Buch and Mihalcea's curve neighborhood calculations of Schubert varieties in the homogeneous cases.

\end{abstract}

\maketitle


%
%

\section{Introduction}
The degree $d$ curve neighborhood of a subvariety $V \subset X$ is the closure of the union of all degree $d$ rational curves through $V$. Curve neighborhoods were introduced in \cite{BCMP:qkfin} to prove finiteness of quantum $K$-theory for $X$ a cominuscule homogeneous space. Let $\IG(k,2n+1)$ be the odd symplectic Grassmannian. In ~\cite{mihalcea.shifler:qhodd} an explicit description of curve neighborhoods of Schubert varieties in $\IG(k,2n+1)$ is limited to the $d=1$ case for irreducible components of ``expected dimension". In this paper we give a full description of the irreducible components of the degree $d$ curve neighborhood of any Schubert variety of $\IG(k,2n+1)$. A key argument is the description of the odd symplectic Grassmannian as a \emph{horospherical variety}, see \cite{pasquier:horo}. Namely, $\IG(k,2n+1)$ can be endowed with an action of the symplectic group $\Sp_{2n}$, which decomposes it into three orbits. Out of these three orbits two are closed and isomorphic to the ordinary (even) symplectic Grassmannians $\IG(k,2n)$ and $\IG(k-1,2n)$, whose curve neighborhoods play a crucial role in our results.

Let $E:= \C^{2n+1}$ be an odd-dimensional complex vector space and $1 \le k \le n+1$. An \emph{odd symplectic form} $\omega$ on $E$ is a skew-symmetric bilinear form with kernel of dimension $1$. The \emph{odd symplectic Grassmannian} $\IG:= \IG(k, E)$ parametrizes $k$-dimensional linear subspaces of $E$ which are isotropic with respect to $\omega$. One can find vector spaces $F \subset E \subset \widetilde{E}$ such that $\dim F = 2n$, $\dim \widetilde{E} = 2n+2$, the restriction of $\omega$ to $F$ is non-degenerate, and $\omega$ extends to a symplectic form (hence non-degenerate) on $\widetilde{E}$. Then the odd symplectic Grassmannian is an intermediate space
\begin{equation}\label{E:evenodd} 
    \IG(k-1, F) \subset \IG(k, E) \subset \IG(k, \widetilde{E}), 
\end{equation} 
sandwiched between two symplectic Grassmannians. This and the more general \emph{odd symplectic partial flag varieties} have been studied in \cite{mihai:odd,pech:quantum,GPPS,pasquier:horo,mihalcea.shifler:qhodd,LMS}. In particular, Mihai showed that $\IG(k, E)$ is a smooth Schubert variety in $\IG(k, \widetilde{E})$, and that it admits an action of Proctor's \emph{odd symplectic group} $\Sp_{2n+1}$ (see \cite{proctor:oddsymgrps}). If $k \neq n+1$ then the odd symplectic group acts on $\IG(k, E)$ with $2$ orbits, and the closed orbit can be identified with $\IG(k-1, F)$. If $k=1$ then $\IG(1, E) = \mathbb{P}(E)$ and if $k=n+1$ then $\IG(n+1, E)$ is isomorphic to the Lagrangian Grassmannian $\IG(n, F)$. Since $\IG$ is a Schubert variety in the symplectic Grassmannian $\IG(k, \wtE )$ it follows that the (equivariant) fundamental classes of those Schubert varieties $X(u) \subset \IG(k, \wtE )$ included in $\IG$ form a basis for the cohomology ring $H^*(\IG)$; we call this the \emph{Schubert basis}.

The odd symplectic group $\Sp_{2n+1}$ acts on $\IG$ with two orbits. There is a closed orbit $Z$ and an open orbit $\Xo$. The symplectic group $\Sp_{2n}$ acts on $\IG$ with three orbits. There are two closed orbits $Y$ and $Z$ and an open orbit $U$. The orbit $Z$ is isomorphic to $\IG(k-1,2n)$ and the orbit $Y$ is isomorphic to $\IG(k,2n)$. Also, $\Xo=Y \cup U$.

We now discuss curve neighborhoods. Let $X$ be a smooth variety. Let $d \in H_2(X,\Z)$ be an effective degree. Recall that the moduli space of genus $0$, degree $d$ stable maps with two marked points $\Mb_{0,2}(X,d)$ is endowed with two evaluation maps $\ev_i \colon \Mb_{0,2}(X,d) \to X$, $i=1,2$ which evaluate stable maps at the $i$-th marked point.

Let $\Omega \subset X$ be a closed subvariety. The \emph{curve neighborhood} of $\Omega$ is the subscheme 
\[ 
    \Gamma_d(\Omega) := \ev_2( \ev_1^{-1} \Omega) \subset X
\] 
endowed with the reduced scheme structure.

It is natural to first work out curve neighborhood calculations of Schubert varieties for the symplectic Grassmannian $\IG(k-1,2n)$ since it is isomorphic to the closed orbit $Z$. These calculations are worked out in \cite{ShiflerWithrow} in the context of calculating minimum quantum degree that appears in the quantum product. The results from the paper that are relevant to this manuscript are stated in Lemma \ref{lem:Lambdaline} and Lemma \ref{lem:Partline}.

We review the philosophy that Buch and Mihalcea used in \cite{buch.m:nbhds} to calculate curve neighborhoods in the homogeneous case. It was shown in \cite{BCMP:qkfin} that curve neighborhoods of Schubert varieties in the homogeneous space $G/P$ are irreducible. Thus, $\Gamma_d(X(id))=X(z_d)$ where $z_d$ is given by a greedy recursive formula in terms of Weyl group elements, roots, and the Hecke product. To compute any curve neighborhoods of a Schubert variety in $G/P$ it suffices to consider the Schubert point $id$ by considering the action of $G$. That is, $\Gamma_d(X(u))=X(u \cdot z_d)$ where $u$ is a Weyl group element and $\cdot$ is the Hecke product defined in Subsection \ref{s:hecke}. 

A similar situation occurs in $\IG$ by considering the horospherical action of $\Sp_{2n}$. There is a point in the orbit $Y$ corresponding to the Schubert point of $\IG(k,2n)$, which we call $id_Y$. The Schubert point $id$ of $\IG$ corresponds to the Schubert point of the orbit $Z$. For the calculation of curve neighborhoods in $\IG$ it suffices to consider the two points $id_Y$ and $id$. We explain this next.

First we discuss the case when a Schubert variety intersects the open orbit $\Xo$. In Proposition \ref{prop:irredopen} we prove that the curve neighborhood $\Gamma_d(X(w))$ of any Schubert variety $X(w) \subset \IG$ where $X(w) \cap \Xo \neq \emptyset$ (equivalently, $X(w) \cap Y \neq \emptyset$) is irreducible. Then 
\[ 
    \Gamma_d(X(id_Y))=X(id_Y \cdot_k O^\circ(d))
\] 
where the element $\Oo(d)$ is recursively defined. Furthermore, $\Oo(d)$ corresponds to the element $z_d$ in $\IG(k,2n)$. Using the horospherical action of $\Sp_{2n}$ it follows that 
\[
    \Gamma_d(X(w))=X(w \cdot_k \Oo(d))
\] 
where $w$ is an appropriate Weyl group element and $X(w) \cap Y \neq \emptyset$ and $\cdot_k$ is the modified Hecke product defined in Subsection \ref{s:hecke}.

Now we will discuss the case when a Schubert variety is contained in the closed orbit $Z$. There are two possibilities for the behavior of the degree $d$ rational curves that contain the Schubert point $id$. The first case is to consider the closure of the curves that are in $Z$. Since $Z$ is isomorphic to $\IG(k-1,2n)$ the union of those curves is a Schubert variety. This Schubert variety is given by $X(\OZ(d))$ where the element $\OZ(d)$ is recursively defined. Furthermore, $\OZ(d)$ corresponds to the element $z_d$ in $\IG(k-1,2n)$. The second case is to consider the closure of the curves that are intersect $Y \cup U$. In the manuscript we show that this is also a Schubert variety. This Schubert variety is given by $X(\OZ(d_1) \cdot \OY(1) \cdot_k \Oo(d_2))$ where $d_1+d_2+1=d$ and $X(\OY(1))$ is the irreducible component of $\Gamma_1(X(id))$ that intersects the open orbit $\Xo$. So we have that 
\[
    \Gamma_{d_1+d_2+1}(X(id)) = \Gamma_{d_2+1}(X(\OZ(d_1)))=X(\OZ(d_1) \cdot \OY(1) \cdot_k \Oo(d_2)) \cup X(\OZ(d_1+d_2+1)).
\] 
Once again using the horospherical action of $\Sp(2n)$, if $X(w) \subset Z$ then it follows that 
\begin{align*}
    \Gamma_{d_1+d_2+1}(X(w)) &= \Gamma_{d_2+1}(X(w \cdot \OZ(d_1))) \\
        &=X(w \cdot \OZ(d_1) \cdot \OY(1) \cdot_k \Oo(d_2)) \cup X(w \cdot \OZ(d_1+d_2+1)).
\end{align*}
\begin{remark}
There are cases where $X(w \cdot \OZ(d_1+d_2+1)) \subset  X(w \cdot \OZ(d_1) \cdot \OY(1) \cdot_k \Oo(d_2))$. 
\end{remark}

\subsection{Statement of results in terms of partitions}
In this paper we give a full description of the irreducible components of curve neighborhoods in terms of Hecke product of (appropriate) Weyl group elements, $k$-strict partitions, and BC-partitions. The latter set of partitions respect the Bruhat order with inclusions. Next, we will discuss the main results of this manuscript in terms of the two sets of partitions.

The Schubert varieties of $\IG$ are indexed by two sets of partitions. The first is set is the set $\BKT(k,2n+1)$ of $(n-k)$-strict partitions given as follows.
\begin{align*}
    &\BKT(k,2n+1) = \\
    &\{(2n+1-k\geq \lambda_1 \geq \cdots \geq \lambda_k \geq -1) \mid \text{$\lambda$ is $(n-k)$-strict, } \lambda_k=-1 \implies \lambda_1=2n+1-k \}. 
\end{align*}

The number of parts of partitions in the set $\BKT(k,2n+1)$ is the codimension of the corresponding Schubert varieties, however, in that set partition inclusion is not compatible with the Bruhat order.

Next we define BC-partitions, introduced in~\cite{ShiflerWithrow} as an alternative to $k$-strict partitions. The advantage of this set of partitions is that the Bruhat order corresponds to inclusion of the Young diagrams; the drawback is that codimension can no longer be readily computed by summing the parts of the partition. We call these collection of partitions BC-partitions.\footnote{ The ``BC" comes from the fact that this set of partitions is ``{\bf B}ruhat {\bf C}ompatible" for isotropic Grassmannians in Types {\bf B} and {\bf C}.}

To define these new partitions we first recall the encoding of partitions using $01$-words. Let $\lambda \in \Part(k,N)$ be a partition. The boundary of this partition consists of $N$ steps, either horizontal of vertical, going from the northeast corner of the $k \times (N-k)$ rectangle to its southwest corner. The  total number of vertical steps is $k$. We associate to $\lambda$ a $01$-word, denoted by $D(\lambda)$, as follows: if the $i$-th step is horizontal we set $D(\lambda)(i)=0$, otherwise $D(\lambda)=1$. 

Let $\BC(k,2n+2)$ denote the set of partitions $\lambda \in \Part(k,2n+2)$ such that if $D(\lambda)(i)=D(\lambda)(2n+3-i)$ for some $1 \leq i \leq n+1$, then $D(\lambda)(i)=0$. The Schubert varieties in $\IG(k,2n+1)$ are indexed by the following set of partitions $ \BC(k,2n+1):=\{\lambda: \lambda+1^k \in \BC(k,2n+2) \}. $

Let $I \in \{\BC(k,2n+1),\BKT(k,2n+1) \}$. For any $\lambda \in I$ there is a corresponding Weyl group element $w \in \Wodd$ that that indexes the same Schubert variety. We use $\lambda^{\OZ(d)}, \lambda^{\OY(d)},$ and  $\lambda^{\Oo(d)}$ to denote the partition in $I$ that corresponds to $w \cdot \OZ(1)$, $w \cdot \OY(d)$, and $w \cdot_k \Oo(d)$, respectively. These are defined in Section \ref{s:Partitions}.

Next we will set up notation to describe the irreducible components of the curve neighborhoods. First, Let $\ell^d_i(\lambda)=\# \{j \mid \lambda_j>i \mbox{ and } j>d \}$. The next definition gives the sets of partitions that correspond to curve neighborhoods with two irreducible components. 
\begin{defn}
    We will define two sets.
    \begin{enumerate}
        \item The first is for BC-partitions. 
            \[
                \Comp_{\BC(k,2n+1)}(d):=\{ \lambda \in \BC(k,2n+1) \mid \lambda_1=2n+1-k, \lambda_{d+1}-\ell_{d-1}^{d+1}(\lambda)-d=2(n+1-k) \};
            \]
        \item  The second is for $(n-k)$-strict partitions. 
            \begin{align*}
                &\Comp_{\BKT(k,2n+1)}(d):= \\
                &\{ \lambda \in \BKT(k,2n+1): \lambda_1=2n+1-k, \lambda_2^{\OZ(d-1)}-\ell_{-1}^2(\lambda^{\OZ(d-1)})=2(n+1-k)\}.
            \end{align*}
    \end{enumerate}
\end{defn}

We will now state the main theorem of the article. A more precise version of the theorem is given as Theorem \ref{thm:mainthmsec}.

\begin{thm}
Let $I \in \{\BC(k,2n+1),\BKT(k,2n+1) \}$. If $\lambda \in I$ then
\[ 
    \Gamma_d(X(\lambda)) =  \begin{cases}
                                X\left(\lambda^{\OY(d)}\right) \cup X\left(\lambda^{\OZ(d)}\right) & \text{if $\lambda \in \Comp_{I}(d)$}\\
                                X\left(\lambda^{\OY(d)}\right) & \text{if $X(\lambda) \subset Z, \lambda \notin \Comp_{I}(d)$} \\
                                X\left( \lambda^{\Oo(d)}\right) & \text{if $X(\lambda) \cap \Xo \neq \emptyset$.}
                            \end{cases}
\]
Moreover, in the first case $(\lambda \in \Comp_I(d))$, the Schubert varieties $X\left(\lambda^{\OY(d)}\right)$ and  $X\left(\lambda^{\OZ(d)}\right)$ form two irreducible components.
\end{thm}

\subsection{Broader context of results} Here is a what is known about irreducible curve neighborhood components. Curve neighborhoods in homogeneous $G/P$ case are irreducible. In \cite{Songul:thesis}, Aslan shows that the irreducible components of curve neighborhoods in the Affine Flag in Type A have equal dimension. The next example is an example of a curve neighborhood where the (co)dimensions of the irreducible components differ.

\begin{example}
Let $k=3$ and $n=4$ and consider the set of partitions $\BKT(k,2n+1)$ for this example. Then we have that \[\Gamma_1(X(6,5,-1))=X(5,0,0) \cup X(6,-1,-1).\]
Moreover, only one component, $X(6,-1,-1)$, is of ``expected" dimension in the sense that its codimension respects the grading of the quantum Chevalley formula $[X(1)] \star [X(\lambda)] \in \QH^*(\IG(k,2n+1))$. That is,  \[ \text{codim} X(1)+\text{codim} X(6,5,-1)=\text{codim}X(6,-1,-1)+1 \cdot c_1\] where $c_1=2n+2-k$ is the coefficient of the first Chern class of $\IG(k,2n+1)$. See \cite{mihalcea.shifler:qhodd}*{Proposition 11.3, Theorem 11.7} and \cite{Shifler:thesis}*{Example 1.2.4} for further details and examples.
\end{example}

{\em Acknowledgements.} Shifler would like to thank Leonardo Mihalcea for introducing him to curve neighborhoods and the odd symplectic Grassmannian. Shifler is partially supported by the Building Research Excellence (BRE) Program at Salisbury University.

\section{Notations and definitions}

\subsection{Odd symplectic Grassmannians}

Let $E:= \C^{2n+1}$ be an odd-dimensional complex vector space. An \emph{odd symplectic form} $\omega$ on $E$ is a skew-symmetric bilinear form of maximal rank, i.e. with kernel of dimension $1$. It will be convenient to extend the form $\omega$ to a (nondegenerate) symplectic form $\wto$ on an even-dimensional space $\wtE \supset E$, and to identify $E \subset \wtE$ with a coordinate hyperplane $\C^{2n+1} \subset \C^{2n+2}$. 

For that, let $\{ \e_1,\ldots , \e_{n+1}, \e_{\overline{n+1}}, \ldots , \e_{\bar{2}}, \e_{\bar{1}} \}$ be the standard basis of $\wtE:=\C^{2n+2}$, where $\bar{i}=2n+3-i$. Set $|i|=\min\{i,\bar{i} \}$, and consider $\wto$ to be the symplectic form on $\wtE$ defined by 
\[ 
    \wto(\e_i,\e_j)=\delta_{i,\bar{j}} \text{ for all $1 \leq i\leq j \leq \bar{1}$}.
\]
The form $\wto$ restricts to the degenerate skew-symmetric form $\omega$ on \[ E=\C^{2n+1}=\left<\e_1, \e_2,\cdots, \e_{2n+1} \right>\] such that the kernel $\ker \omega$ is generated by $\e_1$. Then 
\[
    \omega(\e_i,\e_j)=\delta_{i,\bar{j}} \text{ for all  $1 \leq i\leq j \leq \bar{2}$}.
\] 
Let $F \subset E$ denote the $2n$-dimensional vector space with basis $\{ \e_2, \e_3,\cdots,\e_{2n+1} \}$.

If $1 \leq k \leq n+1$, the \emph{odd symplectic Grassmannian} $X:= \IG(k, E)$ parametrizes $k$-dimensional linear subspaces of $E$ which are isotropic with respect to $\omega$. The restriction of $\omega$ to $F$ is non-degenerate, hence we can see the odd symplectic Grassmannian as intermediate space
\begin{equation}\label{E:evenodd} 
    \IG(k-1,F) \subset \IG(k,E) \subset \IG(k,\wtE), 
\end{equation} 
sandwiched between two symplectic Grassmannians. This and the more general ``odd symplectic partial flag varieties'' have been studied by Mihai \cite{mihai:odd} and Pech \cite{pech:quantum}. In particular, Mihai showed that $\IG(k,E)$ is a smooth Schubert variety in $\IG(k,\wtE)$. 

There are two degenerate cases, namely, if $k=1$ then $\IG(1,E) = \bP(E)$ and if $k=n+1$ then $\IG(n+1,E)$ is isomorphic to the Lagrangian Grassmannian $\IG(n,F)$.

\subsection{The odd symplectic group}

Proctor's \emph{odd symplectic group} (see \cite{proctor:oddsymgrps}) is the subgroup of $\GL(E)$ which preserves the odd symplectic form $\omega$: 
\[ 
    \Sp_{2n+1}(E) := \{ g \in \GL(E) \mid \omega(g \cdot u, g \cdot v) = \omega(u,v), \forall u,v \in E \}. 
\] 

Let $\Sp_{2n}(F)$ and $\Sp_{2n+2}(\widetilde{E})$ denote the symplectic groups which respectively preserve the symplectic forms $\omega_{\mid F}$ and $\wto$. Then with respect to the decomposition $E = F \oplus \ker \omega$ the elements of the odd symplectic group $\Sp_{2n+1}(E)$ are matrices of the form 
\begin{equation*}
    \Sp_{2n+1}(E) = \left\{  \begin{pmatrix}
        \lambda & a  \\
        0 & S  \\
    \end{pmatrix} \mid \lambda \in \C^*, a \in \C^{2n}, S \in \Sp_{2n}(F) \right\}.
\end{equation*}

The symplectic group $\Sp_{2n}(F)$ embeds naturally into $\Sp_{2n+1}(E)$ by $\lambda = 1$ and $a = 0$, but $\Sp_{2n+1}(E)$ is \emph{not} a subgroup of $\Sp_{2n+2}(\wtE)$.\begin{footnote}{However, Gelfand and Zelevinsky \cite{gelfand.zelevinsky} defined another group $\widetilde{\Sp}_{2n+1}$ closely related to $\Sp_{2n+1}$ such that $\Sp_{2n} \subset \widetilde{\Sp}_{2n+1} \subset \Sp_{2n+2}$.}\end{footnote}~Mihai showed in \cite{mihai:odd}*{Prop 3.3} that there is a surjection $P \to \Sp_{2n+1}(E)$ where $P \subset \Sp_{2n+2}(\wtE)$ is the parabolic subgroup which preserves $\ker \omega$, and the map is given by restricting $g \mapsto g_{|E}$. Then the Borel subgroup $B_{2n+2} \subset \Sp_{2n+2}(\wtE)$ of upper triangular matrices restricts to the (Borel) subgroup $B \subset \Sp_{2n+1}(E)$. Similarly, the maximal torus
\[
    T_{2n+2} := \{ \diag(t_1,\cdots,t_{n+1},t_{n+1}^{-1},\cdots, t_1^{-1}): t_1,\cdots,t_{n+1} \in \C^* \} \subset B_{2n+2}
\]
restricts to the maximal torus 
\[ 
    T=\{ \diag(t_1,\cdots,t_{n+1},t_{n+1}^{-1},\cdots, t_{2}^{-1}): t_1,\cdots,t_{n+1} \in \C^* \} \subset B.
\]
Later on we will also require notation for subgroups of $\Sp_{2n}(F)$, viewed as a subgroup of $\Sp_{2n+1}(E)$. We denote by $B_{2n} \subset B$ the Borel subgroup of upper-triangular matrices in $\Sp_{2n}(F)$ and by $\Tsmall$ the maximal torus. We have
\[
    \Tsmall= \{ \diag(1,t_2,\cdots,t_{n+1},t_{n+1}^{-1},\cdots, t_{2}^{-1}): t_2,\cdots,t_{n+1} \in \C^* \} \subset B_{2n}.
\]
We also denote by $U_{2n}$, $U$, $U_{2n+2}$ the maximal unipotent subgroups.

For $1 \leq k \leq n$, Mihai showed that the odd symplectic group $\Sp_{2n+1}(E)$ acts on $X=\IG(k,E)$ with two orbits:
\begin{align*}
 \Xo &= \{ V \in \IG(k,E) \mid \e_1 \notin V\} \text{ the open orbit} \\
\Xc &= \{ V \in \IG(k,E) \mid \e_1 \in V\} \text{ the closed orbit}.
\end{align*}
The closed orbit $\Xc$ is isomorphic to $\IG(k-1,F)$. Indeed, since $F \cap \ker \omega = \{0\}$ it follows that $\omega$ restricts to a nondegenerate form on $F$. If $k=n+1$ then $\IG(n+1,E)= \Xc$ may be identified to the Lagrangian Grassmannian $\IG(n,F)$. 

From now on we will identify $F \subset E \subset \widetilde{E}$ to $\C^{2n} \subset \C^{2n+1} \subset \C^{2n+2}$ with the bases $\langle \e_2, \ldots , \e_{2n+1} \rangle \subset \langle \e_1 , \ldots , \e_{2n+1} \rangle \subset \langle \e_1, \ldots , \e_{2n+2} \rangle$ introduced in the previous paragraph. The corresponding isotropic Grassmannians will be denoted by $\IG(k-1, 2n) \subset \IG(k,2n+1) \subset \IG(k, 2n+2) $. Similarly $\Sp_{2n+1}(E)$ will be denoted by $\Sp_{2n+1}$ etc.

\subsection{The action of \texorpdfstring{$\Sp_{2n}$}{Sp2n} on \texorpdfstring{$\IG(k,2n+1)$}{the odd symplectic Grassmannian}}\label{s:Sp2naction}

The odd symplectic Grassmannian also has an action of the (semisimple) group $\Sp_{2n}$. This action is \emph{horospherical}, see~\cite{pasquier:horo}. This means that the odd symplectic Grassmannian has an open $\Sp_{2n}$-orbit, denoted by $U$ below, which is a torus bundle over a generalized flag variety. Indeed, under the $\Sp_{2n}$-action $X=\IG(k,2n+1)$ possesses three orbits
\begin{itemize}
    \item an open orbit $\U = \{V \in X \mid V \not\subset F, \e_1 \notin V \}$,isomorphic to a $(\C^*)^k$-bundle over $\IG(k,2n)$ ;
    \item a closed orbit, $\Y= \{ V \in X \mid V \subset F \}$, isomorphic to the symplectic Grassmannian $\IG(k,2n)$;
    \item the closed $\Sp_{2n+1}$-orbit $\Xc$.
\end{itemize}
Note that the reunion $\U \cup \Y$ is the open $\Sp_{2n+1}$-orbit $\Xo$.

Denote by $P_\Y$, $P_\Xc$ the parabolic subgroups of $\Sp_{2n}$ such that $\Y=\Sp_{2n}/P_\Y$, $\Xc=\Sp_{2n}/P_\Xc$, and by $W_\Y$, $W_\Xc$ the associated Weyl groups, which are subgroups of the Weyl group $\Wsmall$ of $\Sp_{2n}$. We also let $W^\Y$ and $W^\Xc$ be minimal coset representatives of $\Wsmall/W_\Y$ and $\Wsmall/W_\Xc$, respectively.

For later use we recall the construction of $\IG(k,2n+1)$ as a horospherical variety. Denote by $V_\Y$ and $V_\Xc$ the irreducible $\Sp_{2n}$-representations corresponding to the minimal projective embedding of $\Y$ and $\Xc$, respectively, and let $v_\Y$ and $v_\Xc$ be the corresponding highest weight vectors. Then 
\[
    \IG(k,2n+1)=\overline{\Sp_{2n} \cdot (v_\Y+v_\Xc)} \subset \bP(V_\Y \oplus V_\Xc)
\]
Let $\pi_\Xc : \Xt_\Xc \to X$ be the blow-up of $\Xc$ in $X$. It is obtained via base change from the blow-up of $\bP(V_\Xc)$ in $\bP(V_\Y\oplus V_\Xc)$. In particular there is a natural projection $p_\Y : \Xt_\Xc \to \Y$, obtained as the restriction of the projection of the former blow-up to $\bP(V_\Y)$. The projection $p_\Y : \Xt_\Xc \to \Y$ restricts to a projection $p_\Y : \Xo \to \Y$, which realises $\Xo$ as a vector bundle over $\Y=\IG(k,2n)$. This bundle is $\Sp_{2n}$-equivariant and thus it is obtained from a $P_\Y$-representation. We write $F_\Y$ for the corresponding locally free sheaf; it is the normal bundle to $\Y$ in $X$. The projection $p_\Y : \Xt_\Xc \to \Y$ realises $\Xt_\Xc$ as the projective bundle $\bP_\Y(F_\Y \oplus \cO_\Y)$. Write $E$ for the exceptional divisor, which is isomorphic to the incidence variety $\Sp_{2n}/(P_\Y \cap P_\Xc)$.  The maps $q_\Y$ and $q_\Xc$ from $E$ to $\Y$ and $\Xc$ are the natural projections from $E$ to $\Sp_{2n}/P_\Y$ and $\Sp_{2n}/P_\Xc$.


We get the commutative diagram
\[
	\xymatrix{E \ar@{^(->}[r]^-{j_\Xc}  \ar@/^1.5pc/[rr]^-{q_\Y} \ar[d]_-{q_\Xc} & \Xt_\Xc \ar[r]^-{p_\Y} \ar[d]^-{\pi_\Xc} & \Y \\
  	\Xc \ar@{^(->}[r]^-{i_\Xc} & X & \\} 
\]

\subsection{The Weyl group of \texorpdfstring{$\Sp_{2n+2}$}{the symplectic group} and odd symplectic minimal representatives}\label{s:weyl}

There are many possible ways to index the Schubert varieties of isotropic Grassmannians. Here we recall an indexation using
signed permutations. 

Consider the root system of type $C_{n+1}$ with positive roots
\[
    R^+ = \{ t_i \pm t_j \mid 1 \leq  i < j \leq n+1\} \cup \{ 2 t_i \mid 1 \leq  i \leq n+1 \}
\]
and the subset of simple roots 
\[
    \Delta = \{ \alpha_i := t_i-t_{i+1} \mid 1 \leq i \leq n \} \cup \{ \alpha_{n+1}:= 2t_{n+1} \}.
\]
The associated Weyl group $\Wlarge$ is the hyperoctahedral group consisting of \emph{signed permutations}, i.e. permutations $w$ of the elements $\{1, \cdots, n+1,\overline{n+1},\cdots,\overline{1}\}$ satisfying $w(\overline{i})=\overline{w(i)}$ for all $w \in W$. For $1 \leq i \leq n$ denote by $s_i$ the simple reflection corresponding to the root $t_i - t_{i+1}$ and $s_{n+1}$ the simple reflection of $2 t_{n+1}$. Each subset $ I:=\{ i_1 < \ldots < i_r \} \subset \{ 1, \ldots , n+1 \}$ determines a parabolic subgroup $P:=P_I \leq \Sp_{2n+2}(\wtE)$ with Weyl group $W_{P} = \langle s_i \mid i \neq i_j \rangle$ generated by reflections with indices \emph{not} in $I$. Let $\Delta_P:= \{ \alpha_{i_s} \mid i_s \notin \{ i_1, \ldots , i_r \} \}$ and $R_P^+ := \Span_\Z \Delta_P \cap R^+$; these are the positive roots of $P$. Let $\ell \colon W \to \mathbb{N}$ be the length function and denote by $W^{P}$ the set of minimal length representatives of the cosets in $W/W_{P}$. The length function descends to $W/W_P$ by $\ell(u W_P) = \ell(u')$ where $u' \in W^P$ is the minimal length representative for the coset $u W_P$. We have a natural ordering 
\[ 
    1 < 2 < \ldots < n+1 < \overline{n+1} < \ldots < \overline{1},
\]
which is consistent with our earlier notation $\overline{i} := 2n+3 - i$. Let $P=P_k$ to be the maximal parabolic obtained by excluding the reflection $s_k$. Then the minimal length representatives $\Wlarge^P$ have the form $(w(1)<w(2)<\cdots<w(k)| w(k+1)< \ldots < w(n+1) \le n+1)$ if $k < n+1$ and $(w(1)<w(2)<\cdots<w(n+1))$ if $k= n+1$. Since the last $n+1-k$ labels are determined from the first $k$ ones, we will identify an element in $\Wlarge^P$ with the sequence $(w(1)<w(2)<\cdots<w(k))$. 

\begin{example}\label{Ex:minlengthrep}
The reflection $s_{t_1+t_2}$ is given by the signed permutation \[s_{t_1+t_2}(1)=\bar{2},  s_{t_1+t_2}(2)=\bar{1}, \mbox{ and }s_{t_1+t_2}(i)=i \mbox{ for all }3 \leq i \leq n+1.\] The minimal length representative of $s_{t_1+t_2}\Wlarge^P$ is $(3<4<\cdots<k<\bar{2}<\bar{1})$.
\end{example}

\subsection{Schubert Varieties in even and odd symplectic Grassmannians} 

Recall that the even symplectic Grassmannian $\Xev=\IG(k,2n+2)$ is a homogeneous space $\Sp_{2n+2}/P$, where $P=P_k$ is the parabolic subgroup generated by the simple reflections $s_i$ with $i \neq k$. For each $w \in \Wlarge^P$ let $\Xev(w)^\circ:=B_{2n+2} w B_{2n+2}/P$ be the \emph{Schubert cell}. This is isomorphic to the space $\C^{\ell(w)}$. Its closure $\Xev(w):=\overline{\Xev(w)^\circ}$ is the \emph{Schubert variety}. We might occasionally use the notation $\Xev(w \Wlarge_P)$ if we want to emphasize the corresponding coset, or if $w$ is not necessarily a minimal length representative. Recall that the Bruhat ordering can be equivalently described by $v \leq w$ if and only if $\Xev(v) \subset \Xev(w)$. Set 
\begin{equation}\label{E:w0} 
    w_0 = \begin{cases} 
            (\overline{2}, \overline{3}, \ldots , \overline{n+1}, 1) & \text{if $i_r < n+1$}; \\ 
            ( 1, \overline{2}, \overline{3}, \ldots , \overline{n+1}) & \text{ if $i_r = n+1$}; 
            \end{cases} 
\end{equation} 
this is an element in $\Wlarge$.
Recall that the odd symplectic Borel subgroup is $B= B_{2n+2} \cap \Sp_{2n+1}$. The following results were proved by Mihai \cite{mihai:odd}*{\S 4}. 

\begin{prop}\label{prop:schubert}
\begin{enumerate}[(a)]
    \item The natural embedding $\iota: X=\IG(k,2n+1) \hookrightarrow \Xev=\IG(k,2n+2)$ identifies $\IG(k,2n+1)$ with the (smooth) Schubert subvariety 
    \[ 
        \Xev(w_0 \Wlarge_P) \subset \IG(k,2n+2).
    \] 
    \item The Schubert cells (i.e. the $B_{2n+2}$-orbits) in $\Xev(w_0)$ coincide with the $B$-orbits in $\IG(k,2n+1)$. In particular, the $B$-orbits in $\IG(k,2n+1)$ are given by the Schubert cells $\Xev(w)^\circ \subset \IG(k,2n+2)$ such that $w \leq w_0$.
\end{enumerate}
\end{prop} 

To emphasize that we discuss Schubert cells or varieties in the odd symplectic case, for each $w \leq w_0$ such that $w \in \Wlarge^P$, we denote by $X(w)^\circ$, and $X(w)$, the Schubert cell, respectively the Schubert variety in $\IG(k,2n+1)$. The same Schubert variety $X(w)$, but regarded in the even symplectic Grassmannian is denoted by $\Xev(w)$. For further use we note that $\IG(k, 2n+1)$ has complex codimension $k$ in $\IG(k, 2n+2)$. Further, a Schubert variety $X(w)$ in $\IG(k, 2n+1)$ is included in the closed $\Sp_{2n+1}$-orbit $\Xc$ of if and only if it has a minimal length representative $w \leq w_0$ such that $w(1)=1$. 

Define the set $\Wodd:= \{ w \in \Wlarge \mid w \leq w_0 \}$ and call its elements \emph{odd symplectic permutations}. The set $\Wodd$ consists of permutations $w \in W$ such that $w(j) \neq \bar{1}$ for any $1 \leq j \leq n+1$ \cite{mihai:odd}*{Prop. 4.16}. We also introduce the subset $\Wo$ of odd symplectic permutations $w \in \Wodd$ such that $w(j) \neq 1$ for any $1 \leq j \leq k$. Permutations in $\Wo \cap W^P$ will index Schubert varieties of the open $\Sp_{2n+1}$-orbit $\Xo$.

Later on we will also require notation for the Schubert varieties of the closed $\Sp_{2n}$-orbits $\Y$ and $\Xc$ in $X=\IG(k,2n+1)$ introduced in Section~\ref{s:Sp2naction}. Namely, if $u \in W^\Y$ and $v \in W^\Xc$ we denote by $\Y(u)$ and $\Xc(v)$ the corresponding Schubert varieties of $\Y$ and $\Xc$, respectively. Recall the commutative  diagram from~\ref{s:Sp2naction}:
\[
	\xymatrix{E \ar@{^(->}[r]^-{j_\Xc}  \ar@/^1.5pc/[rr]^-{q_\Y} \ar[d]_-{q_\Xc} & \Xt_\Xc \ar[r]^-{p_\Y} \ar[d]^-{\pi_\Xc} & \Y \\
  	\Xc \ar@{^(->}[r]^-{i_\Xc} & X & \\} 
\]
and consider the varieties $\pi_\Xc(p_\Y^{-1}(\Y(u)))$ and $i_\Xc(\Xc(v))$. In~\cite{GPPS}*{Prop. 5.3}, it is shown that these varieties coincide with the Schubert basis $X(w)$ of $X$. Here we rephrase this identification in terms of the minimal length representatives introduced in Section~\ref{s:weyl}.

\begin{prop}\label{prop:YZ-permutations}
    The odd symplectic permutations $w \in \Wo \cap W^P$ indexing Schubert varieties of the open $\Sp_{2n+1}$-orbit $\Xo$ are in bijection with the elements $u \in W^\Y$ indexing Schubert varieties of the $\Y$-orbit. Explicitly, the bijection $\Phi \colon \Wo \cap W^P \to W^\Y$ is as follows. 
    
    Let $w=(a_1 < \dots <a_r < \bar a_k < \dots < \bar a_{r+1})$, then 
    \[
        \Phi(w) = (a_1-1 < \dots <a_r-1 < \overline{a_k-1} < \dots < \overline{a_{r+1}-1}).
    \]
    
    Similarly, the odd symplectic permutations $w \in \Wodd \cap W^P$ indexing Schubert varieties of the closed $\Sp_{2n+1}$-orbit $\Xo$, i.e., those with $w(1)=1$, are in bijection with the elements $v \in W^\Xc$ indexing Schubert varieties of the $\Xc$-orbit, and the bijection is as follows:
    \[
        \Phi_\Xc(w) = (a_2-1 < \dots <a_r-1 < \overline{a_k-1} < \dots < \overline{a_{r+1}-1}),
    \]
    where $w=(1<a_2<\dots <a_r< \overline{a_k} < \dots < \overline{a_{r+1}})$.
    
    In terms of the Schubert varieties themselves we get
    \begin{itemize}
        \item $X(w)=\pi_\Xc(p_\Y^{-1}(\Y(\Phi(w))))$ if $w \in \Wo \cap W^P$;
        \item $X(w)=i_\Xc(\Xc(\Phi_\Xc(w)))$ if $w \in \Wodd \cap W^P$ is such that $w(1)=1$.
    \end{itemize}
\end{prop}

\begin{proof}
This result is a direct consequence of \cite[Proposition 5.3]{GPPS}, except that here we are indexing Schubert varieties with Weyl group elements instead of $k$-strict partitions.
\end{proof}

\subsection{The Hecke product}\label{s:hecke}
The Weyl group $\Wlarge$ admits a partial ordering $\leq$ given by the \emph{Bruhat order}. Its covering relations are given by $w < ws_\alpha$ where $\alpha \in R^+$ is a root and $\ell(w) < \ell(w s_\alpha)$. We will use the \emph{Hecke product} on the Weyl group $\Wlarge$. For a simple reflection $s_i$ the product is defined by 
\begin{align*}
    w \cdot s_i =   \begin{cases}
                        w s_i & \text{ if $\ell(ws_i)>\ell(w)$;} \\
                        w & \text{otherwise.}
                    \end{cases}
\end{align*}
The Hecke product gives $\Wlarge$ a structure of an associative monoid; see e.g.~\cite{buch.m:nbhds}*{\S 3} for more details. Given $u,v \in \Wlarge$, the product $uv$ is called \emph{reduced} if $\ell(uv)=\ell(u)+\ell(v)$, or, equivalently, if $uv = u \cdot v$. For any parabolic group $P$, the Hecke product determines a left action $W \times W/\Wlarge_P \to W/\Wlarge_{P}$ defined by 
\[
    u \cdot (w \Wlarge_{P}) = (u \cdot w) \Wlarge_P. 
\] 
We recall the following properties of the Hecke product (cf.~e.g. \cite{buch.m:nbhds}).

For later use we introduce a \emph{modified Hecke product} $\cdot _k$ on the indexing set $\Wo$ of Schubert varieties of the open orbit $\Xo$ of $X=\IG(k,2n+1)$. In this case, we want to multiply by the simple root $s_k$ whether or not it increases the length of the word. Namely
\begin{align*}
    w \cdot_k s_i =   \begin{cases}
                        w s_i & \text{ if $\ell(ws_i)>\ell(w)$ or $i=k$;} \\
                        w & \text{otherwise.}
                    \end{cases}
\end{align*}

Note that the modified Hecke product on $\Xo$ corresponds to the usual Hecke product on the closed $\Sp_{2n}$-orbit $\Y \subset \Xo$, as showed in the following statement. This fact will be useful in computing curve neighbourhoods of the Schubert varieties of the open orbit $\Xo$.

\begin{prop} \label{prop:bijY2W}
Recall the map $\Phi \colon \Wo \cap W^P \to W^\Y$ from Proposition~\ref{prop:YZ-permutations}. We define a map $\psi$ mapping the simple reflections in $W_{2n}$ to reflections in $W$ as follows:
\[
    \psi(s_i) = \begin{cases}
                    s_i & \text{for $1 \leq i \leq k-1$}, \\
                    s_k s_{k+1} s_k & \text{for $i=k$}, \\
                    s_{i+1} & \text{for $k+1 \leq i \leq n$}.
                \end{cases}
\]
Let $w=s_{i_1}s_{i_2} \ldots s_{i_r}.$ Then for any $v \in \Wo$,
\[ 
    \Phi(\left(v \cdot_k \psi(s_{i_1}) \ldots \psi(s_{i_r}) \right) W_P)=\left( \Phi(v W_P) \cdot w \right)W_Y. 
\]
\end{prop}

\begin{proof}
We use induction on $r$, the number of reflections in the expression for $w$. Start with $v \in \Wo$ and $w=s_i$ with $i \neq k$. By definition of $\Phi$ we have $\Phi(vW_P) \in W^\Y$, hence $\ell(\Phi(vW_P)s_i)>\ell(\Phi(vW_P))$. We deduce that $\left( \Phi(vW_P) \cdot w \right)W_\Y = \Phi(vW_P) W_\Y$. On the other hand, since $i \neq k$ then $(v \cdot_k \psi(s_i))W_P=vW_P$, hence $\Phi(\left(v \cdot_k \psi(s_i) \right) W_P)=\Phi(vW_P) W_\Y$. 

Now assume $w=s_k$ and let $\tilde{v} \in \Wo \cap W^P$ be the minimal length representative of $v$. We have $\tilde v(k+1)=1$, and $\Phi(vW_P) = (\tilde v(1)-1< \dots< \tilde v(k)-1)$. Hence 
\[
    (\Phi(vW_P) \cdot s_k) W_\Y = \begin{cases}
                            (\tilde v(1)-1< \dots< \tilde v(k)-1) & \text{if $\tilde v(k)>\tilde v(k+2)$,} \\
                            (\tilde v(1)-1 < \dots < \tilde v(k-1)-1 < \tilde v(k+2)-1) & \text{otherwise.}
                        \end{cases}
\]
On the other hand
\[
    (v \cdot_k (s_ks_{k+1}s_k)) W_P = \begin{cases}
                                            (\tilde v(1)< \dots < \tilde v(k)) & \text{if $\tilde v(k)>\tilde v(k+2)$,} \\
                            (\tilde v(1)< \dots<\tilde v(k-1)<\tilde v(k+2)) & \text{otherwise.}
                                        \end{cases}
\]
Therefore in any case $\Phi(\left(v \cdot_k \psi(s_k) \right) W_P)=\left( \Phi(vW_P) \cdot s_k \right)W_Y$. This concludes the proof of the identity for $r=1$.

Now assume that $\Phi(\left(v \cdot_k \psi(s_{i_1}) \ldots \psi(s_{i_r}) \right) W_P)=\left( \Phi(v W_P) \cdot w \right)W_Y$ for any $w=s_{i_1}s_{i_2} \ldots s_{i_r}$ and consider $w'=ws_j$. Let $u:=v \cdot_k \psi(s_{i_1}) \ldots \psi(s_{i_r})$.
By definition of the modified Hecke product
\[
    \Phi(\left(v \cdot_k \psi(s_{i_1}) \ldots \psi(s_{i_r}) \psi(s_j) \right) W_P) = \Phi(\left(u \cdot_k \psi(s_j) \right) W_P) = \left( \Phi(u W_P) \cdot s_j \right)W_Y. 
\]
Here the second equality comes from applying the $r=1$ case. By induction we have $\Phi(u W_P) = (\Phi(vW_P) \cdot w) W_Y$,
hence 
\[
    \left( \Phi(u W_P) \cdot s_j \right)W_Y = (\Phi(vW_P) \cdot w) \cdot s_j W_Y = (\Phi(vW_P) \cdot w') W_Y,
\]
where the second equality comes from the definition of the Hecke product. This concludes the proof.
\end{proof}

We have a similar result for the closed $\Sp_{2n+1}$-orbit $\Xc$.

\begin{prop} \label{prop:bijZ}
Recall the map
\[
    \Phi_\Xc \colon \{ w \in \Wodd \cap W^P \mid w(1)=1 \} \to W^\Xc 
\]
from Proposition~\ref{prop:YZ-permutations}. Let $w=s_{i_1}s_{i_2} \ldots s_{i_r}$. Then for any $v \in \Wodd \cap W^P$ with $v(1)=1$,
\[ 
    \Phi_\Xc(\left(v \cdot s_{i_1+1} \ldots s_{i_r+1} \right) W_P)=\left( \Phi_\Xc(v W_P) \cdot w \right)W^\Xc. 
\]
\end{prop}

\begin{proof}
We use induction on $r$, the number of reflections in the expression for $w$. 

Start with $v \in \Wodd \cap W^P$ with $v(1)=1$ and $w=s_i$ with $i \neq k-1$. By definition of $\Phi_\Xc$ we have $\Phi_\Xc(vW_P) \in W^\Xc$, hence $\ell(\Phi_\Xc(vW_P)s_i)>\ell(\Phi_\Xc(vW_P))$ since $i \neq k-1$. We deduce that $\left( \Phi_\Xc(vW_P) \cdot s_i \right)W_\Xc = \Phi_\Xc(vW_P) W_\Xc$. On the other hand, since $i+1 \neq k$ then $(v \cdot s_{i+1})W_P=vW_P$, hence $\Phi_\Xc(\left(v \cdot s_{i+1} \right) W_P)=\Phi_\Xc(vW_P) W_\Xc$. Now assume $w=s_{k-1}$, then $\ell(\Phi_\Xc(vW_P)s_{k-1})<\ell(\Phi_\Xc(vW_P))$, hence $\Phi_\Xc(vW_P) \cdot s_{k-1}=\Phi_\Xc(vW_P)$. But in that case we have $v \cdot s_k =v$, hence we still get $\Phi_\Xc(\left(v \cdot s_k \right) W_P)=\left( \Phi_\Xc(v W_P) \cdot s_{k-1} \right)W^\Xc$.

Finally, the proof of the induction step is similar to that in the proof of the induction step in Proposition~\ref{prop:bijY2W}.
\end{proof}



We now introduce some special odd symplectic permutations, which will be used to define line neighborhoods in $X=\IG(k,2n+1)$, namely
\[
    \Oo(1) := s_1 \cdots s_{n+1} \cdots s_1(=s_{2t_1})
\] 
for the open $\Sp_{2n+1}$-orbit $\Xo$, and
\[
    \OY(1) := s_1s_2 \cdots s_{k-1}s_{k+1} \cdots s_{n+1} \cdots s_2s_1 \text{ and } \OZ(1) := s_2 \cdots s_{n+1} \cdots s_2
\]
for the closed $\Sp_{2n+1}$-orbit $\Xc$. The following Lemma follows from a direct calculation of the Hecke product. To simplify notation we will distinguish two cases, depending on whether $k<n+1$ or $k=n+1$.

\begin{lemma}\label{lem:comp}
Assume $k<n+1$ and let $w=(1<a_2 < \dots <a_r < \bar a_k < \dots < \bar a_{r+1}) \in \Wodd \cap W^P$ and $v=(b_1 < \dots <b_s < \bar b_k < \dots < \bar b_{s+1}) \in \Wo \cap W^P$, that is, $X(w) \subset \Xc$ and $X(w) \cap \Xo \neq \emptyset$. Then 
\begin{align*}
    w \cdot \OY(1) &=(w(2)<\cdots<\bar j_\Y<\cdots<w(k)), \\
    w \cdot \OZ(1) &= (1<w(3)<\cdots< \bar j_\Xc<\cdots<w(k)), \\
    v \cdot_k \Oo(1) &= (v(2)<v(3)<\cdots< \bar j^\circ<\cdots<v(k)),
\end{align*}
where 
\begin{align*}
    j_\Y &=\min \{2, \dots, n+1\} \setminus \{a_2,\dots,a_k\}, \\
    j_\Xc &= \min \{2, \dots, n+1\} \setminus \{ a_3,\dots,a_k\}, \\
    j^\circ &= \min \{2, \dots n+1\} \setminus \{ b_2,\dots, b_k\}.
\end{align*}
Note that $\bar j_\Y,\bar j_\Xc,\bar j^\circ$ are possibly smaller than $w(2),w(3),v(2)$ or larger than $w(k),v(k)$.

Assume $k=n+1$, in which case $X=\Xc \cong \IG(n,2n)$, and let $w=(1<a_2 < \dots < a_r< \bar a_{n+1} < \dots < \bar a_{r+1})$. Then 
\[
    w \cdot \OZ(1) = (1<w(3)<\cdots< \bar j_\Xc<\cdots<w(k)), 
\]
where $j_\Xc$ equals $a_2$ if $r \geq 2$ and $n+1$ otherwise,
and $\bar j_\Xc$ is possibly smaller than $w(3)$ or possibly larger than $w(k)$.
\end{lemma}


\begin{example}\label{ex:hecke}
Let us illustrate the lemma when $k=4$ and $n=6$. Consider $w=(1<2 <\bar 6< \bar 3)$ and $v=(2<4<6<\bar 5)$. Then $j_\Y=4$ and $j_\Xc=j^\circ=2$, hence
\begin{align*}
    w \cdot \OY(1) &=(2<\bar 6<\bar 4<\bar 3), \\
    w \cdot \OZ(1) &= (1<\bar 6<\bar 3<\bar 2), \\
    v \cdot_k \Oo(1) &= (4<6<\bar 5<\bar 2).
\end{align*}
\end{example}

\section{The moment graph}\label{s:moment} 

Sometimes called the GKM graph, the \emph{moment graph} of a variety with an action of a torus $T$ has a vertex for each $T$-fixed point, and an edge for each $1$-dimensional torus orbit. The description of the moment graphs for flag manifolds is well known, and it can be found e.g in \cite{kumar:kacmoody}*{Ch. XII}. In this section we consider the moment graphs for $X=\IG(k, 2n+1) \subset \Xev = \IG(k, 2n+2)$. As before let $P = P_k \subset \Sp_{2n+2}$ be the maximal parabolic for $\Xev$. Recall that the minimal length representatives in $w \in W^{P}$ are in one to one correspondence to sequences $1 \leq w(1) < \ldots < w(k) \leq \bar{1}$, and that those corresponding to the odd symplectic Grassmannian satisfy in addition that $w(i) \leq \bar{2}$ for $1 \leq i \leq n+1$. 

\subsection{Moment graph structure of \texorpdfstring{$\IG(k,2n+2)$}{IG(k,2n+2)}}

The moment graph of $\Xev$ has a vertex for each $w \in W^P$, and an edge $w \rightarrow ws_\alpha$ for each \[ 
    \alpha \in R^+ \setminus R_{P}^+ = \{ t_i - t_j \mid 1 \leq i \leq k < j \leq n+1 \} \cup \{ t_i + t_j, 2 t_i \mid 1 \leq i <j \leq n+1, i \leq k  \}. 
\]
Geometrically, this edge corresponds to the unique torus-stable curve $C_\alpha(w)$ joining $w$ and $ws_\alpha$. The curve $C_\alpha(w)$ has degree $d$, where $\alpha^\vee + \Delta_P^\vee = d \alpha_k^\vee + \Delta_P^\vee$. 

\begin{defn}\label{def:moment-graph-combinat}
Define the following to describe moment graph combinatorics.
\begin{enumerate}
    \item First we will partition the positive roots into two sets.
    \begin{enumerate}
        \item $R^+_1=\{t_i \pm t_j:1 \leq i \leq k< j \leq n+1 \} \cup \{2t_i: 1 \leq i \leq k \};$ 
        \item $R^+_2=\{t_i+t_j: 1 \leq i <j \leq k \}$;
    \end{enumerate}
    \item Let $u \in W^P$ then define $\varphi(u)=\#\{u(i):u(i) \geq \overline{n+1}\}$; this is the number of `barred' elements in the symplectic permutation $u$;
    \item A \emph{chain of degree $d$} is a path in the (unoriented) moment graph where the sum of the edge degrees equals $d$. We will often use the notation $uW_P     \overset{d}\to vW_P$ to denote such a path.
\end{enumerate}
\end{defn}

The next proposition describes the degrees of edges in the moment graph.

\begin{prop}\label{prop:MGDisciption}
Consider the moment graph of $\IG(k,2n+2)$. Let $v,w \in W^P.$ There is an edge between $v$ and $w$ if $w=vs_\alpha$ for some $\alpha \in R^+$. 
\begin{enumerate}
\item The edge has degree $i \in \{1,2\}$ if $\alpha \in R_i^+;$
\item If $\alpha \in R_1^+$ then $\varphi(w) \leq \varphi(v)+1$;
\item If $\alpha \in R_2^+$ then $\varphi(w) \leq \varphi(v)+2$.
\end{enumerate}
\end{prop}

\begin{proof} The fact that an edge exists between $v$ and $w$ if $w=s_\alpha$ for some $\alpha \in R^+$ follows from the definition of the moment graph. 

Moreover, if $\alpha \in R_1^+$ then we are in one of the following three situations:
\begin{align*}
    t_i-t_j =& (t_i-t_{i+1}) + (t_{i+1}-t_{i+1}) + \cdots + \mathbf{(t_k-t_{k+1})} + \cdots + (t_{j-1}-t_j),\\
    t_i+t_j =& (t_i-t_{i+1}) + (t_{i+1}-t_{i+1}) + \cdots + \mathbf{(t_k-t_{k+1})} + \cdots + 2(t_{j}+t_{j+1}) \\
        &+ \cdots + 2(t_{n}-t_{n+1}) + 2t_{n+1},\\
    t_i =& (t_i-t_{i+1}) + \cdots + \mathbf{(t_{k}-t_{k+1})} + \cdots + (t_n-t_{n+1}) + t_{n+1}.
\end{align*}
In any case we see that the corresponding edge has degree $1$. On the other hand if $\alpha \in R_2^+$ then
\[
    (t_i-t_{i+1})+\cdots+2(t_j-t_{j-1})+\cdots+\mathbf{2(t_k-t_{k+1})}+\cdots+2(t_n-t_{n+1})+2t_{n+1},
\]
hence the corresponding edge has degree $2$. This proves Part (1) of the statement. Parts (2) and (3) are clear. 
\end{proof}

\begin{example}
    Assume $k=4$ and $n=6$ and consider the vertex indexed by $w=(1<2<\bar 6<\bar 3)$. Let $\alpha_1=t_1-t_5$, $\alpha_2=t_1+t_5$ and $\alpha_3=t_1+t_2$. The minimal length representatives of $ws_{\alpha_i}$ for $i=1,2,3$ are $u_1=(2<4<\bar 6 <\bar 3)$, $u_2=(2<\bar 6 <\bar 4<\bar 3)$, and $u_3=(\bar 6<\bar 3<\bar 2<\bar 1)$. The edges $w \to u_1$, $w \to u_2$ both have degree $1$ while the edge $w \to u_3$ has degree $2$. We also notice that $\phi(w)=2$, $\phi(u_1)=2$, $\phi(u_2)=3$, and $\phi(u_3)=4$, in agreement with Proposition~\ref{prop:MGDisciption}.
\end{example}

The next lemma gives a necessary condition on $u,v \in W^P$ for a chain of degree $d$ to exist from $u$ to $v$.

\begin{lemma}\label{lem:phi_degree}
Let $u,v \in W^P$ be connected by a degree $d$ chain
\[
    (u W_P \overset{d} \to vW_P) = (uW_P \to us_{\alpha_1} W_P \to \dots \to us_{\alpha_1}s_{\alpha_2}\ldots s_{\alpha_t}W_P)
\]
where $vW_P=us_{\alpha_1}s_{\alpha_2}\ldots s_{\alpha_t}W_P$ and the $\alpha_i$ are in $ R^+ \backslash R^+_P$. Then:
\begin{enumerate}
    \item $d=\#\{\alpha_i \in R^+_1\}+2 \cdot \#\{\alpha_i \in R^+_2\}$,
    \item $\varphi(v) \leq d+\varphi(u)$,
    \item if $u=id$ then at least $k-d$ elements of $\{v(1),v(2), \cdots, v(k) \}$ must be smaller than or equal to $k$.
\end{enumerate}
\end{lemma}

\begin{proof}
For (1) just recall from Proposition \ref{prop:MGDisciption} that edges corresponding to $\alpha_i \in R^+_1$ have degree $1$ while edges corresponding to $\alpha_i \in R^+_2$ have degree $2$. For (2), use (1) and apply Proposition \ref{prop:MGDisciption} (2,3) to the chain. Finally, for (3) notice that all $k$ entries of $id=(1<2<\dots<k)$ are smaller than or equal to $k$, and that a degree $i$ edge ($i=1,2$) sends at most $i$ entries smaller than or equal to $k$ to entries larger than $k$.
\end{proof}

\subsection{Moment graph structure of \texorpdfstring{$\IG(k,2n+1)$}{IG(k,2n+1)}}
In this subection we will assume that $k<n+1.$ The moment graph of $X$ is the full subgraph of that of $\Xev$ determined by the vertices $w \in W^P \cap \Wodd$. Notice that the orbits of $T$ and $T_{2n+2}$ coincide, therefore we do not distinguish between the moment graphs for these tori. 

\begin{prop}\label{prop:T-fixed-YZ}
The $\Tsmall$-fixed points of $\IG(k,2n+1)$ are all contained in $\Y$ or $\Xc$.
\end{prop}

\begin{proof}
The $\Tsmall$-fixed points of $\Y$ and $\Xc$, which are both $\Sp_{2n}$-homogeneous spaces, are given by $u P_\Y$ for $u \in W^\Y$ and $v P_\Xc$ for $v \in W^\Xc$. Since both $\Y$ and $\Xc$ embed $\Sp_{2n}$-equivariantly into $X=\IG(k,2n+1)$, these points are also $\Tsmall$-fixed points of $X$.
    
To see that there are no other $\Tsmall$-fixed points in $X$, notice that all the $\Tsmall$-fixed points above are also fixed by $T$ and $T_{2n+2}$. But there are as many $T_{2n+2}$-fixed points in $X$ as the rank of its cohomology, and that is also the number of points of the form $u P_\Y$ for $u \in W^\Y$ and $v P_\Xc$ for $v \in W^\Xc$.
\end{proof}

\begin{lemma}
Let $1 \leq d \leq k$, and $v \in W^P \cap \Wodd$. If $\varphi(v)<d$ then there exists $u \in  W^P \cap \Wodd$ such that $u>v$ (for the Bruhat order) and $\varphi(u)=d$.
\end{lemma}

\begin{proof}
Our assumptions imply that $v=(a_1<\dots<a_r<\bar a_k<\dots<\bar a_{r+1})$ with $r>k-d$. The element $u=(a_1<\dots<a_{k-d}<\bar a_k < \dots <\bar a_{k-d+1})$ clearly satisfies $\varphi(u)=d$, and it is larger than $v$ for the Bruhat order since $r>k-d$ and $a_1 < \dots <a_k$.
\end{proof}

\begin{prop}\label{prop:open-orbit-GKM}
The function $\Phi \colon \Wo \cap W^P \to W^\Y$ from Proposition~\ref{prop:YZ-permutations} induces a moment graph isomorphism between:
\begin{itemize}
    \item the full subgraph $MG_{\Xo}$ of the moment graph of $\IG(k,2n+1)$ induced by the $T$-fixed points contained in $\Xo$,
    \item the moment graph $MG_\Y$ of $\Y=\IG(k,2n)$.
\end{itemize}
The graph isomorphism is in the sense that the graphs are labeled with 1's and 2's instead of coroots.
\end{prop}

\begin{proof}
Recall that the $T$-fixed points in the open $\Sp_{2n+1}$-orbit $\Xo$ are indexed by the elements of $\Wo \cap W^P$. From Proposition~\ref{prop:YZ-permutations} we know that $\Phi$ is bijective, therefore the vertices of $MG_{\Xo}$ and $MG_\Y$ are in bijection.

Moreover, the edges of $MG_{\Xo}$ correspond to the $T$-stable curves in $\Xo$,
\[
    E^{\circ} = \{ C(w,v) \mid w,v \in \Wo \cap W^P, v=ws_{\alpha} \text{ for some $\alpha \in R^+$} \}.
\] 
Likewise, the edges of $MG_\Y$ correspond to the $T$-stable curves in $\IG(k,2n)$: 
\[
    E_\Y= \{ C(w,v) \mid w,v \in W^\Y, v=ws_{\alpha} \text{ for some $\alpha \in R_{2n}^+$} \}.
\]
The map $\Phi_E \colon E^{\circ} \to E$ given by
\[
    \Phi_E(C(w,v))=C(\Phi(w),\Phi(v)),
\]
is clearly a bijection between both edge sets.

To conclude we show that the edge degrees are preserved. Let $C(w,v) \in E^\circ$. Then $v=ws_\alpha$ for some $\alpha\in R^+$, and we have:
\begin{enumerate}
    \item $v=ws_{t_i \pm t_j}$ if and only if $\Phi(v)=\Phi(w)s_{t_i \pm t_{j-1}}$ for $1 \leq i \leq k$ and $k+1 \leq j \leq n+1$;
    \item $v=ws_{2t_i}$ if and only if $\Phi(v)=\Phi(w)s_{2t_i}$ for $1 \leq i \leq k$;
    \item $v=ws_{t_i+t_j}$ if and only if $\Phi(v)=\Phi(w)s_{t_i+t_{j}}$ for $1 \leq i < j \leq k$.
\end{enumerate}
Thus by Proposition \ref{prop:MGDisciption}, the degree $\deg C(u,v)$ in $\Xo$ is equal to the degree of $C(\Phi(u),\Phi(v))$ in $\Y$.
\end{proof}

The next Theorem gives a description of the $T$-fixed curves in the moment graph that has a $T$-fixed point in the closed orbit and another $T$-fixed in the open orbit. We reserve its proof until Subsection~\ref{subs:proofofthm:ZtoYlines} as it requires combinatorial objects which are yet to be defined.

\begin{thm} \label{thm:ZtoYlines}
Let $w,v \in W^P \cap \Wodd$ where $v=w s_\alpha$ for some root $\alpha$, $X(w) \subset \Xc$, and $X(v) \cap \Xo \neq \emptyset$. Then 
\begin{enumerate}
    \item The $T$-stable curve $C(w,v)$ from $w$ to $v$ in the moment graph has degree 1;
    \item The inequality $\ell(v) >\ell(w)$ holds.
\end{enumerate}
\end{thm}

The next proposition uses Lemma~\ref{lem:comp} to prove the ``square" property. It will later be interpreted in the following way. If $X(w) \subset \Xc$ is a Schubert variety of the closed $\Sp_{2n+1}$-orbit, then its line neighborhood generally consists of two connected components, $X(w \cdot \OY(1))$, which intersects $\Xo$, and $X(w \cdot \OZ(1))$, which is contained in $\Xc$. Now the proposition says that the line neighborhood of $X(w \cdot \OY(1))$ coincides with the connected component of the line neighborhood of $X(w \cdot \OZ(1))$ which intersects $\Xo$.

\begin{prop} \label{prop:square}
Let $w \in \Wodd \cap W^P$ be such that $X(w) \subset \Xc$. Then 
\[
    \left( \left(w \cdot \OZ(1) W_P\right) \cdot \OY(1) \right) W_P = \left(\left(w \cdot \OY(1) W_P\right) \cdot_k \Oo(1) \right)W_P.
\]
\end{prop}

\begin{proof}
We have $w=(1<a_2<\dots<a_r<\bar a_k<\dots<\bar a_{k-r})$. We use Lemma~\ref{lem:comp} throughout. Evaluating the left side of the equation we have
\begin{align*}
    w \cdot \OZ(1) \cdot \OY(1) &= (1<w(3)<\cdots<\bar j_\Xc<\cdots<w(k)) \cdot \OY(1)\\
        &=(w(3)<\cdots< \bar j_\Y<\cdots<\bar j_\Xc<\cdots<w(k))
\end{align*}
where 
\[
    j_\Xc = \min \left( \{2, \cdots, n+1 \} \setminus \{a_3,\cdots, a_k\} \right)
\] 
and 
\[
    j_\Y=\min \left( \{2, \cdots, n+1 \} \setminus \{a_3,\cdots,a_k,j_\Xc\} \right).
\]
We now evaluate the right side of the equation.
\begin{align*}
    w \cdot \OY(1) \cdot_k \Oo(1) &= (w(2)<w(3)<\cdots<\bar i_\Y<\cdots<w(k)) \cdot_k O^\circ(1)\\
        &= (w(3)<\cdots<\bar i^\circ<\cdots<\bar i_\Y<\cdots<w(k))
\end{align*}
where 
\begin{align*}
    i_\Y &= \min \left( \{2, \cdots, n+1 \} \setminus \{a_2,\dots,a_k\} \right), \\
    i^\circ &= \min \left( \{2, \cdots, n+1 \} \setminus \{a_3,\dots,a_k,i_\Y\} \right),
\end{align*}
and $i^\circ$ is possibly smaller than $i_\Y$.

By definition $j_\Xc \leq i_\Y$. If $j_\Xc=i_\Y$ then $j_\Y$ and $i^\circ$ are clearly also equal, and the result follows. On the other hand, if $j_\Xc < i_\Y$ then $j_\Xc$ must be equal to $a_2$, so by their definitions $j_\Y=i_\Y$. Moreover
$j_\Xc$ can only be different from $i^\circ$ if 
\[
    i_\Y=\min \left( \{2, \cdots, n+1 \} \setminus \{a_3,\dots,a_k\} \right) =j_\Xc,
\]
which would contradict our assumption that $j_\Xc < i_\Y$. Therefore $j_\Xc=i^\circ$ and $j_\Y=i_\Y$, which concludes the proof.
\end{proof}

\begin{example}
Assume $k=4$, $n=6$, and $w=(1<2<\bar 6<\bar 3)$. Then $w \cdot \OY(1)=(2<\bar 6<\bar 4<\bar 3)$ and $w \cdot \OZ(1)=(1<\bar 6 <\bar 3<\bar 2)$, see Example~\ref{ex:hecke}. Applying Lemma~\ref{lem:comp} we get that $w \cdot \OY(1) \cdot_k \Oo(1)$ and $w \cdot \OZ(1) \cdot \OY(1)$ are both equal to $(\bar 6<\bar 4<\bar 3<\bar 2)$, as claimed in Proposition~\ref{prop:square}.
\end{example}

\begin{cor}\label{cor:square}
Let $w \in \Wodd \cap W^P$ be such that $X(w) \subset \Xc$. Then if $d_1+d_2=d-1$,
\[
    \left( w \cdot \OY(d) \right) W_P = \left( \left( \left(w \cdot \OZ(d_1) \right) W_P \cdot \OY(1) \right) W_P  \cdot_k \Oo(d_2) \right) W_P
\]
\end{cor}

\begin{proof}
Note that in the following proof, even though we will be working with $W_P$-cosets, we may omit them from the notation to improve readability.

By definition of $\OZ(d_1)$, if $d_1 \geq 1$ we have
\[
    \left(w \cdot \OZ(d_1) \right) \cdot \OY(1) = \left(\left(w \cdot \OZ(d_1-1)\right) \cdot \OZ(1)\right) \cdot \OY(1).
\]
Applying Proposition~\ref{prop:square} on the right-hand side we deduce
\[
    \left(\left(w \cdot \OZ(d_1-1)\right) \cdot \OZ(1)\right) \cdot \OY(1) = \left(\left(w \cdot \OZ(d_1-1)\right) \cdot \OY(1)\right) \cdot_k \Oo(1),
\]
and therefore
\begin{align*}
    \left( \left( \left(w \cdot \OZ(d_1) \right) \cdot \OY(1) \right) \cdot_k \Oo(d_2) \right) &= \left(\left( \left( \left(w \cdot \OZ(d_1-1) \right) \cdot \OY(1) \right) \cdot_k \Oo(1) \right) \cdot_k \Oo(d_2) \right) \\
    &=\left( \left(w \cdot \OZ(d_1-1) \right) \cdot \OY(1) \right) \cdot_k \Oo(d_2+1),
\end{align*}
where the second equality comes from the definition of $\Oo(d_2+1)$. Iterating this process $d_1-1$ more times we obtain
\[
    \left( \left( \left(w \cdot \OZ(d_1) \right) \cdot \OY(1) \right) \cdot_k \Oo(d_2) \right) = \left(w \cdot \OY(1) \right) \cdot_k \Oo(d_1+d_2) = \left(w \cdot \OY(1) \right) \cdot_k \Oo(d-1).
\]
As $\OY(1) \cdot_k \Oo(d-1)=\OY(d)$, the result follows.
\end{proof}

Figure~\ref{fig:IG(2,6)} illustrates the moment graphs of $\IG(2,5)$ and $\IG(2,6)$. The blue portion corresponds to vertices and edges outside the Schubert variety $\IG(2,5)$, while other vertices and edges form the moment graph of $\IG(2,5)$. Within this moment graph lies that of $\Xc=\IG(1,4)= \mathbb{P}^3$, depicted in red, and that of $\Y=\IG(2,4) \subset \Xo$, depicted in black. The green edges link torus-fixed points of the closed $\Sp_{2n+1}$-orbit $\Xc$ to torus-fixed points of the open $\Sp_{2n+1}$-orbit $\Xo$. Note that the edges involved in Proposition~\ref{prop:square} are of that type.

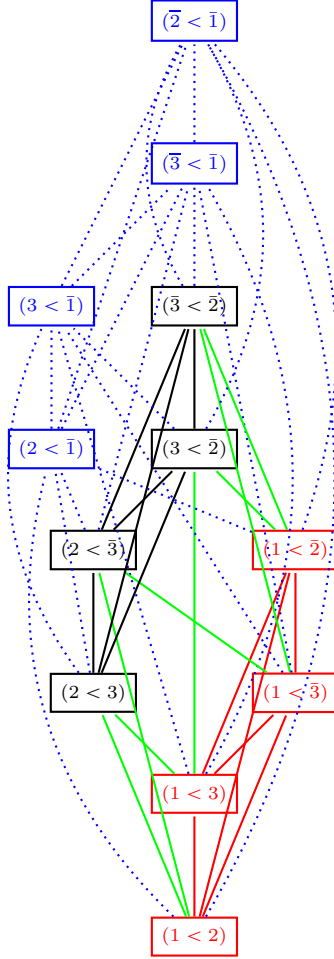
\begin{figure}
\caption{The moment graph of $\IG(2,5)$ inside that of $\IG(2,6)$, without degree labels. }
\label{fig:IG(2,6)}
\begin{tikzpicture}[-,>=stealth',shorten >=1pt,auto,node distance=1.9cm,
                    thick,main node/.style={draw,font=\sffamily\tiny\bfseries}]

  \node[main node] (1) {$(\bar{3}<\bar{2})$};
  \node[main node] (2) [below of=1] {$(3<\bar{2})$};
    \node[main node] (3) [below left of=2] {$(2<\bar{3})$};
  \node[main node] (4) [ red, below right of=2] {$(1<\bar{2})$};
  \node[main node] (5) [below of=3]{$(2<3)$};
  \node[main node] (6) [ red, below of=4]{$(1<\bar{3})$};
   \node[main node] (7) [red, below right of=5]{$(1<3)$};
   \node[main node] (8) [ red, below of=7]{$(1<2)$};
   \node[main node] (9) [blue,left of=2] {$(2<\bar{1})$};
    \node[main node] (10) [blue, left of=1] {$(3<\bar{1})$};
     \node[main node] (11) [blue, above of=1] {$(\overline{3}<\bar{1})$};
\node[main node] (12) [blue, above of=11] {$(\overline{2}<\bar{1})$};

  \path[every node/.style={font=\sffamily\large}]
    (8) edge [red] node {}  (7)
         edge [green] node {} (5)
          edge [red] node {} (6)
          edge [green] node {} (3)
          edge [red] node {} (4)
          edge [blue,dotted,bend left] node {} (9)
          edge [blue, dotted,bend right] node {} (12)
    (7) edge [green] node {} (5)
          edge [green] node {} (2)
          edge [red] node  {} (4) 
          edge [red] node {} (6)
            edge [blue,dotted] node {} (10)
          edge [blue, dotted,bend right] node {} (11)
    (6) edge [green] node [right] {} (3)
          edge [green] node [right] {} (1)
          edge [red] node [right] {} (4)
             edge [blue,dotted] node {} (10)
          edge [blue, dotted] node {} (11)
    (5) edge  node {} (3)
          edge  node {} (1)
          edge node {} (2)
           edge [blue,dotted] node {} (9)
          edge [blue, dotted, bend left] node {} (10)
    (4) edge [green] node {} (2)      
          edge [green] node {} (1)
           edge [blue,dotted] node {} (9)
          edge [blue, dotted,bend right] node {} (12)
    (3) edge  node {} (1)
          edge  node [right] {} (2)
          edge [blue,dotted] node {} (9)
          edge [blue, dotted] node {} (11)
    (2) edge node {} (1)
    edge [blue,dotted] node {} (10)
    edge [blue, dotted,bend right] node {} (12)
    (1)edge [blue,dotted] node {} (11)
          edge [blue, dotted,bend left] node {} (12)
       (9)edge [blue,dotted] node {} (10)
          edge [blue, dotted] node {} (11)
          edge [blue, dotted] node {} (12)
        (10)edge [blue,dotted] node {} (11)
          edge [blue, dotted] node {} (12)
          (11)edge [blue,dotted] node {} (12);
\end{tikzpicture}
\end{figure} 

\section{Curve Neighborhoods}\label{s:curvenbhds} 

Let $\cX \in \{ X, \Xev \}$ and let $d \in H_2(\cX,\Z)$ be an effective degree. Recall that the moduli space of genus $0$, degree $d$ stable maps with two marked points $\Mb_{0,2}(\cX,d)$ is endowed with two evaluation maps $\ev_i \colon \Mb_{0,2}(\cX,d) \to \cX$, $i=1,2$ which stable maps at the $i$-th marked point.

Let $\Omega \subset \cX$ be a closed subvariety. The \emph{curve neighborhood} of $\Omega$ is the subscheme 
\[ 
    \Gamma_d(\Omega) := \ev_2( \ev_1^{-1} \Omega) \subset \cX
\] 
endowed with the reduced scheme structure. This notion was introduced by Buch, Chaput, Mihalcea and Perrin \cite{BCMP:qkfin} to help study the quantum K-theory ring of cominuscule Grassmannians. It was analyzed further for any homogeneous space by Buch and Mihalcea \cite{buch.m:nbhds}, in relation to $2$-point K-theoretic Gromov-Witten invariants, and to a new proof of the quantum Chevalley formula. Often, estimates for the dimension of the curve neighborhoods provide vanishing conditions for certain Gromov-Witten invariants. 

\begin{remark}\label{rmk:curve-nbhd}
We start with the observation (going back to \cite{BCMP:qkfin}) that if $\Omega$ is a Schubert variety of $\cX$, then $\Gamma_d(\Omega)$ must be a (finite) union of Schubert varieties, stable under the same Borel subgroup. This follows because $\Omega$ is stable under the appropriate Borel subgroup, and $\ev_1, \ev_2$ are proper, equivariant maps; thus $\Gamma_d(\Omega)$ is closed and Borel stable. 
\end{remark}

It is possible to say more for homogeneous spaces. Indeed, it was proved in \cite{BCMP:qkfin} that the curve neighborhood of any Schubert variety of a homogeneous space $V=G/P$ is again a Schubert variety. This Schubert variety was described in \cite{buch.m:nbhds}, namely, $\Gamma_d(V(w)) = V(w \cdot z_d W_P )$, where $z_d \in W$ is \emph{defined} by the condition that $\Gamma_d(1.P) = V(z_d W_P)$. The Weyl group element $z_d$ can be constructed recursively.

Coming back to $\Xev$. The maximal elements of the set $\{\beta \in R^+ \setminus R_{P}^+: \beta^\vee+\Delta_{P}^\vee \leq d \}$ are called \emph{maximal roots} of $d$. The following follows from \cite{buch.m:nbhds}*{Corollary 4.12}. 
The following follows from \cite{buch.m:nbhds}*{Corollary 4.12}.

\begin{prop}\label{prop:zd} Let $d \in H_2(\IG(k,2n+2))$ be an effective degree. If $\alpha \in R^+ \setminus R_{P}^+$ is a maximal root of $d$, then $s_\alpha \cdot z_{d - \alpha^\vee} W_P = z_d W_P$.
\end{prop}

\begin{cor}\label{cor:evest} 
\begin{enumerate}[(a)]
    \item If $k >1$ then there is an equality $z_1 W_P = s_{2 t_1} W_P$ and the minimal length representative of $z_1 W_P$ is $(2<3<\cdots<k<\overline{1})$. 
    \item For $d \geq 1$ it follows that $z_dW_P=s_{2t_1}\cdot s_{2t_1} \cdot \ldots \cdot s_{2t_1}W_P$ (where $s_{2t_1}$ appears $d$-times)
\end{enumerate}
\end{cor}

\begin{proof} The first part follows directly from Proposition \ref{prop:zd}. For part (b), notice that $2t_1$ is a maximal root of $d=1$, therefore $z_1 W_P = s_{2 t_1} W_P$. By the recursion in Proposition \ref{prop:zd}, since $2t_1$ is a maximal element of any $d \geq 1$ we obtain $z_d W_P=s_{2t_1}\cdot s_{2t_1} \cdot \ldots \cdot s_{2t_1}W_P$ (where $s_{2t_1}$ appears $d$-times). 
\end{proof}

Let $w \in W^P \cap \Wodd$ and let $d \in H_2(X,\Z)$ be an effective degree. As mentioned above, the curve neighborhood $\Gamma_d(X(w))$ of $X(w)$ is a closed, $B$-stable subvariety of $X$, therefore it must be a union of Schubert varieties: 
\[ 
    \Gamma_d(X(w)) = X(w^1) \cup  \cdots \cup X(w^r)
\] 
where $w^i \in W^P \cap \Wodd$. As noticed in \cite{buch.m:nbhds}*{\S 5.2} and \cite{mare.mihalcea}*{Cor. 5.5}, the permutations $w^i$ can be determined combinatorially from the moment graph. 

\begin{prop}\label{prop:moment-odd} 
Let $w \in W^P \cap \Wodd$. In the moment graph of $X=\IG(k,2n+1)$, let $\{v^1, \cdots, v^s\}$ be the maximal vertices (for the Bruhat order) which can be reached from any $u \leq w$ using a path of degree $d$ or less. Then $\Gamma_d(X(w))=X(v^1) \cup \cdots \cup X(v^s)$.
\end{prop}

\begin{proof}
Let $Z_{w,d}=X(v^1) \cup \cdots \cup X(v^s)$. Let $v:= v^i \in Z_{w,d}$ be one of the maximal $T$-fixed points. By the definition of $v$ and the moment graph there exists a chain of $T$-stable rational curves of degree less than or equal to $d$ joining $u \leq w$ to $v$. It follows that there exists a degree $d$ stable map joining $u \leq w$ to $v$. Therefore $v \in \Gamma_d(X(w))$, thus $X(v) \subset \Gamma_d(X(w))$, and finally $Z_{w,d} \subset \Gamma_d(X(w))$.

For the converse inclusion, let $v \in \Gamma_d(X(w))$ be a $T$-fixed point. By \cite{mare.mihalcea}*{Lemma 5.3} there exists a $T$-stable curve joining a fixed point $u \in X(w)$ to $v$. This curve corresponds to a path of degree $d$ or less from some $u \leq w$ to $v$ in the moment graph of $\IG(k,2n+1)$. By maximality of the $v^i$ it follows that $v \leq v^i$ for some $i$, hence $v  \in X(v^i) \subset Z_{w,d}$, which completes the proof.
\end{proof}

\subsection{Curve neighborhoods in the open orbit \texorpdfstring{$\Xo$}{Xo}}

We now consider the curve neighborhoods of Schubert varieties which intersect the open $\Sp_{2n+1}$-orbit $\Xo$. Such Schubert varieties are indexed by the elements of $\Wo \cap W^P$, see Proposition~\ref{prop:YZ-permutations}. Our strategy consists of three steps:
\begin{itemize}
    \item we show in Proposition~\ref{prop:irredopen} that $\Gamma_d(X(w))$ is irreducible, hence a Schubert variety, when $w \in \Wo \cap W^P$;
    \item we compute the curve neighborhood of a particular Schubert variety $X(id_\Y)$ in Proposition~\ref{prop:curveidY};
    \item we deduce in Proposition~\ref{prop:nbhd-open} the curve neighborhood of any Schubert variety which intersect $\Xo$ using the $\Sp_{2n}$-action on the $T$-fixed points of the open orbit.
\end{itemize}

\begin{prop} \label{prop:irredopen}
Let $X(w) \subset \IG(k,2n+1)$ be a Schubert variety such that $X(w) \cap X^\circ \neq \emptyset$ and let $d$ be an effective degree. Then $\Gamma_d(X(w))$ is irreducible.
\end{prop}

\begin{proof}
First apply \cite{BCMP:qkfin}*{Prop. 2.3} to the evaluation morphism $ev_2 \colon \overline{\mathcal{M}}_{0,2}(\IG(k,2n+1),d) \to \IG(k,2n+1)$ to deduce that this morphism is a locally trivial fibration over the open orbit, and that the fibers over the open orbit are irreducible. The map is $\Sp_{2n+1}$-equivariant and both varieties are irreducible (see \cite{GPPS}*{Theorem 2.5} for the irreducibility of the moduli space). We also need $\IG(k,2n+1)$ to be $\Sp_{2n+1}$-split in order to apply \cite{BCMP:qkfin}*{Prop. 2.3}. We already know it is $B_{2n+2}$-split by \cite{BCMP:qkfin}*{Prop. 2.2}, since it is a Schubert variety of $\IG(k,2n+2)$. This $B_{2n+2}$-splitting is the inverse of the isomorphism $U_{2n+2} \to \Omega^\circ$ from \cite{BCMP:qkfin}*{Prop. 2.2}. Here $U_{2n+2}$ denotes the maximal unipotent subgroup of $\Sp_{2n+2}$, and $\Omega^\circ$ is the big cell in $\IG(k,2n+1)$. Let $s : \Omega^\circ \to U_{2n+2}$ denote this splitting. Composing with the map $B_{2n+2} \to \Sp_{2n+1}, g \mapsto g_{\mid E}$ gives us a morphism $\Omega^\circ \to \Sp_{2n+1}$ which is a $\Sp_{2n+1}$-splitting.

Now let us check that if $X(w) \subset \IG(k,2n+1)$ is a Schubert variety such that $X(w) \cap X^\circ \neq \emptyset$, then the Gromov-Witten variety $ev_2^{-1}(X(w))$ is irreducible. We use Lemma 5.8.12 from \cite{stacks-project}. We look at the map $ev_2^{-1}(X(w)) \to X(w)$, which is open. We have $X(w)$ irreducible, and there does exist a dense collection of points over which the fibre is irreducible (all the points in $X(w) \cap X^\circ$). So the Gromov-Witten variety $ev_2^{-1}(X(w))$ is be irreducible.

Then $\Gamma_d(X(w))$ is the image by $ev_1$ of the irreducible variety $ev_2^{-1}(X(w))$, so it is also irreducible. Moreover, we know that it is a disjoint union of Schubert varieties in $\IG(k,2n+1)$, as observed in Remark~\ref{rmk:curve-nbhd}. Therefore $\Gamma_d(X(w))$ is of the form $X(v)$ for some $v \in W^P \cap W^{odd}$, hence it is irreducible.
\end{proof}

From Proposition~\ref{prop:T-fixed-YZ} we know that all the $T$-fixed points of the open $\Sp_{2n+1}$-orbit are in fact contained in the $\Sp_{2n}$-orbit $\Y =\IG(k,2n)$. Therefore there is a $T$-fixed point in $\Y$ which plays an analogous role to that of the Schubert point $X(id)$. Namely, we define
\[
    id_\Y = (2 < 3 < \dots < k) \in \Wlarge^P.
\]
Recall the notation $\Oo(1) := s_1 \cdots s_{n+1} \cdots s_1$ from Section~\ref{s:hecke} and define
\[ 
    \Oo(d) = \Oo(d-1) \cdot_k \Oo(1)
\]
for $d>1$.

\begin{prop} \label{prop:curveidY}
The curve neighborhood of the Schubert variety $X(id_\Y)$ is given by 
\[
    \Gamma_d(X(id_\Y)) = X(id_\Y \cdot_k \Oo(d)).
\]
\end{prop} 

To prove Proposition~\ref{prop:curveidY} we take a closer look at chains of curves from $id_\Y$.

\begin{defn}
Let $A_\Y(d)$ denote the set of $T$-fixed points in $\Y$ connected to $id_\Y$ by a degree $d$ chain, that is, 
\[
    A_\Y(d) = \left\{ u \in W^P \cap \Wodd \mid u(1) \neq 1,  \text{ and there exists $id_YW_P \overset{d} {\longrightarrow} uW_P$} \right\}.
\]
Here $\varphi(u)=\#\{u(i) \mid u(i) \geq \overline{n+1}\}$, see Definition~\ref{def:moment-graph-combinat}.
\end{defn}

\begin{lemma}\label{lem:Ocirc}
Let $k<n+1$ and $1 \leq d \leq k$. The maximal element of $A_\Y(d)$ (w.r.t the Bruhat order) is $id_\Y \cdot_k \Oo(d)$. The minimal length representative of $id_\Y \cdot_k \Oo(d)$ is:
\begin{align*}
    \begin{cases}
        (d+2<d+3<\cdots<k<\overline{d+1}<\cdots<\overline{3}<\overline{2}) & \text{if $1 \leq d<k$;} \\
        (\overline{k+1}<\overline{k}<\cdots<\overline{3}<\overline{2}) &\text{if $d=k$.}
    \end{cases}
\end{align*}
\end{lemma}

\begin{proof} 
To check that the minimal length representative of $id_\Y \cdot_k \Oo(d) =: v_d$ is as stated we repeatedly apply Lemma~\ref{lem:comp}.

We now prove that $v_d$ is the maximal element of $A_\Y(d)$. First, since the curve neighborhood $\Gamma_d(X(id_\Y))$ is irreducible by Proposition \ref{prop:irredopen}, we know that $A_\Y(d)$ has a maximum element, which we denote by $z_d$ . 

Next we argue that a chain $id_\Y W_P\overset{d} {\longrightarrow} z_d$ cannot have any $T$-fixed point in the closed $\Sp_{2n+1}$-orbit $\Xc$. Indeed let $id_\Y W_P \overset{d} {\longrightarrow} uW_P$ be a chain that includes at least one $T$-fixed point in $\Xc$. Such a chain is of the form 
\[
    id_\Y \to u_1 \to \cdots \to u_s \to u_{s+1} \to \cdots \to u_r \to u 
\] 
where $u_s \in \Y$, $u_{s+1} \in \Xc$, and the degrees of the edges add up to $d$. Write $u_s=u_{s+1} s_\alpha$. By Theorem \ref{thm:ZtoYlines} we know that the edge $u_s \to u_{s+1}$ has degree one, and that $\ell(u_s) > \ell(u_{s+1})$. Since $u_s \in \Y$ and $u_{s+1} \in \Xc$ the positive root $\alpha$ must be of the form $t_1 \pm t_j$ for $k>j$. Therefore, the number $\varphi(u_{s+1})$ of barred entries in $u_{s+1}$ cannot exceed $\varphi(u_s)$. We decompose the degree $d$ chain as follows:
\[
    id_\Y \overset{d_1}{\longrightarrow} u_s \overset{1}{\longrightarrow} u_{s+1} \overset{d_2}{\longrightarrow} u,
\]
where $d_1+d_2+1=d$. By Lemma~\ref{lem:phi_degree} we know that $\varphi(u_s) \leq \varphi(id_\Y)+d_1=d_1$ and $\varphi(u) \leq \varphi(u_{s+1})+d_2$. Since $\varphi(u_{s+1}) \leq \varphi(u_s)$ it follows that $\varphi(u) \leq d_1+d_2<d$. 

However, if we start from $id_\Y$ and apply the reflection $s_{2t_1}$ $d$ times, taking minimal $W_P$-coset representatives at each step, we obtain the element $v_d$ defined earlier. This gives us a degree $d$ chain $id_Y \to v_d$, and we clearly have $\varphi(v_d)=d$. If $u$ was the maximal element of $A_\Y(d)$ then we would have $u \geq v_d$, which is impossible since $\varphi(u)<\varphi(v_d)$. Thus a chain $id_\Y \overset{d} {\longrightarrow} z_d$ has all its $T$-fixed points in $\Y$, and we also have that $v_d \leq z_d$ since there is a degree $d$ chain $id_Y \to v_d$.

We now compute the curve neighbourhood of the Schubert point in $\Y \cong \IG(k,2n)$, namely $Y(id)$. Here $id$ denotes the identity in the smaller Weyl group $W_{2n}$. We remark that $\Phi(id_\Y)=id$. Applying the Buch-Mihalcea Recursion, see Proposition~\ref{prop:zd}, we get that 
\[
    \Gamma_d(Y(id)) = Y(s_{2t_1} \cdot \ldots \cdot s_{2t_1} W_\Y)  (\text{ $s_{2t_1}$ appearing $d$-times})
\]
and we easily compute that
\[
    s_{2t_1} \cdot \ldots \cdot s_{2t_1} W_\Y = (d+1<d+2<\dots <k< \bar d < \dots < \bar 2 <\bar 1) = \Phi(v_d).
\]
This means the maximal degree $d$ chain from $id$ in $\IG(k,2n)$ ends at $\Phi(v_d)$. 

Now recall that the chain $id_\Y \overset{d} {\longrightarrow} z_d$ has all its $T$-fixed points in $\Y$, hence we can take its image by $\Phi$. This image is a degree $d$ chain from $id \to \Phi(z_d)$, hence $\Phi(z_d) \leq \Phi(v_d)$ by maximality of $\Phi(v_d)$. Since $\Phi$ is clearly order-preserving is follows that $z_d \leq v_d$, hence finally $z_d=v_d$, which concludes the proof.
\end{proof}

\begin{proof}[Proof of Proposition \ref{prop:curveidY}]
Combining Propositions~\ref{prop:moment-odd} and~\ref{prop:irredopen} we get that $\Gamma_d(X(id_\Y))=X(u_d)$, where $u_d$ is the maximal vertex (for the Bruhat order), which can be reached from some $u \leq id_\Y$ using a path of degree $d$ or less. Let $u \to u_d$ be such a path. Note that $u_d \in \Y$, otherwise the whole curve neighborhood $\Gamma_d(X(id_\Y))$ would be contained in $\Xc$, which is impossible since $id_\Y \not \in \Xc$. If $u=id_\Y$ it follows from Lemma~\ref{lem:Ocirc} that $u_d=id_\Y \cdot_k \Oo(d)$, hence the result. But if $u<id_\Y$ then $u \in \Xc$, and we can show as in the proof of Lemma~\ref{lem:Ocirc} that $\varphi(u_d)<\varphi(u)+d$, hence $\varphi(u_d)<d$, and we argue as before that this contradicts the maximality of $u_d$.
\end{proof}

From Proposition \ref{prop:curveidY} we deduce the situation for an arbitrary Schubert variety $X(w)$ intersecting $\Xo$.

\begin{prop}\label{prop:nbhd-open}
Let $X(w) \subset \IG(k,2n+1)$ be a Schubert variety such that $X(w) \cap X^\circ \neq \emptyset$ and $d$ be an effective degree. Then $\Gamma_d(X(w))=X(w \cdot_k O^\circ(d))$. 
\end{prop} 

\begin{proof}
From Propositions~\ref{prop:moment-odd} and~\ref{prop:irredopen} we know that the curve neighborhood $\Gamma_d(X(w))$ is a Schubert variety. The fact that it is $X(w \cdot_k O^\circ(d))$ follows now directly from Proposition~\ref{prop:curveidY} and the recursion of~\cite[Theorem 5.1]{buch.m:nbhds}.
\end{proof}

\subsection{Curve neighborhoods in the \texorpdfstring{$\Xc$}{Z}-orbit}

In this section we describe the curve neighborhoods of Schubert varieties contained in the closed $\Sp_{2n+1}$-orbit $\Xc$. The following result shows that such a curve neighborhood is in general \emph{not} irreducible; instead, it can have two connected components.

\begin{prop}
Let $X(w) \subset \IG(k,2n+1)$ be a Schubert variety such that $X(w) \subset \Xc$ and let $d$ be an effective degree. Then $\Gamma_d(X(w))$ has one or two connected components. More precisely, there is always a component intersecting the open $\Sp_{2n+1}$-orbit $\Xo$, and there may be an additional component contained in the closed $\Sp_{2n+1}$-orbit $\Xc$.
\end{prop}

This time our strategy consists of two steps:
\begin{itemize}
    \item we compute the curve neighborhood of the Schubert point $X(id)$ in Proposition~\ref{prop:curvidZ};
    \item we deduce in Proposition~\ref{prop:crvnbdweyl} the curve neighborhood of any Schubert variety contained in $\Xc$ using the $\Sp_{2n}$-action on the $T$-fixed points of the open orbit.
\end{itemize}

To state our results we recall the elements $\OY(1) = s_1s_2 \cdots s_{k-1}s_{k+1} \cdots s_{n+1} \cdots s_2s_1$ and $\OZ(1) = s_2 \cdots s_{n+1} \cdots s_2$ from Section~\ref{s:hecke}, and for $d>1$ we define
\[ 
    \OZ(d) := \OZ(d-1) \cdot \OZ(1) \text{ and } \OY(d) := \OZ(d_1) \cdot \OY(1) \cdot_k O^{\circ}(d_2) \text{ where $d_1+d_2+1=d$}. 
\] 

\begin{prop} \label{prop:curvidZ}
The curve neighborhood of the Schubert point is given by
\[
    \Gamma_d(X(id)) = X(\OY(d)) \cup X(\OZ(d)).
\]
Here $X(\OZ(d))$ is contained in the closed orbit, while $X(\OY(d))$ intersects the open orbit.
\end{prop}

\begin{defn}
Let $A_\Xc(d)$ denote the set of $T$-fixed points in $\Y$ such that there is a degree $d$ chain from $id$ to those $T$-fixed points. That is, 
\[ 
    A_\Xc(d) = \left\{ u \in W^P \cap\Wodd \mid u(1) \neq 1 \text{ and there exists $idW_P \overset{d} {\longrightarrow} uW_P$} \right\}.
\]
Note that $A_\Xc(d) \neq \emptyset$ as soon as $d \neq 0$.
\end{defn}

\begin{lemma}\label{lem:OY(d)}
Let $k<n+1$ and $1 \leq d \leq k$. The maximal element (w.r.t. the Bruhat order) of $A_\Xc(d)$ is $\OY(d)$. The minimal length representative of $\OY(d)$ is:
\[
    \begin{cases}
        (2<3<\cdots<k<\overline{k+1}) & \text{if $d=1$}; \\
        (d+1<d+2<\cdots<k<\overline{k+1}<\overline{d}<\cdots<\overline{3}<\overline{2}) & \text{if $1<d<k$}; \\
        (\overline{k+1}<\overline{k}<\cdots<\overline{3}<\overline{2}) & \text{if $d=k$}.
    \end{cases}
\]
\end{lemma}

\begin{proof}
To check that the minimal length representative of $\OY(d)$ is as stated we repeatedly apply Lemma~\ref{lem:comp}. 

First we notice that there is a degree $d$ chain from $id$ to $\OY(d)$, obtained by applying once the reflection $s_{t_1+t_{k+1}}$, and $d-1$ times the reflection $s_{2t_1}$, taking minimal $W_P$-coset representatives at each step. It follows that $\OY(d) \in A_\Xc(d)$.

Now consider $u \in A_\Xc(d)$; if we can prove that $u \leq \OY(d)$ then we will be able to conclude that $\OY(d)$ is the maximal element of $A_\Xc(d)$. Since $u \in A_\Xc(d)$ we must have $r:=\varphi(u) \leq d$, hence $u$ is of the form
\[
    (a_1 < \dots < a_{k-r} < \bar a_k < \dots < \bar a_{k-r+1})
\]
We know from part (3) of Lemma~\ref{lem:phi_degree} that at least $k-d$ elements of $a_1,\dots,a_{k-r},\bar a_k,\dots,\bar a_{k-r+1}$ must be smaller than or equal to $k$. This means $a_1,\dots,a_{k-d}$ are all smaller than or equal to $k$. Therefore we must have $a_1 \leq d+1,a_2 \leq d+2, \dots, a_{k-d} \leq k$. Hence the $k-d$ first entries of $u$ are smaller than or equal to the corresponding entries of $\OY(d)$. 
Moreover, as $u \in \Wodd \cap W_P$ and $u(1) \neq 1$, none of the $a_i$ can be equal to $1$. It follows that the last $d-1$ entries of $u$ are smaller than or equal to the corresponding entries of $\OY(d)$. To conclude that $u \leq \OY(d)$ it only remains to show that $u(k-d+1) \leq \overline{k+1}$. If $u(k-d+1)$ is not a barred element (i.e., if $r=\varphi(u)<d$), then the equality is clearly true. Otherwise $u(k-d+1)=\bar a_k$, and we need to show that $a_k \geq k+1$. Indeed, if we had $a_k \leq k$, then $a_{k-d+1},\dots,a_k$ must all be smaller than or equal to $k$. As $a_1,\dots,a_{k-d}$ are also smaller than or equal to $k$, one of the $a_i$ must be equal to $1$, a contradiction since $u \in \Y$, hence $u \leq \OY(d)$ as claimed.
\end{proof}

\begin{proof}[Proof of Proposition \ref{prop:curvidZ}]
By Proposition~\ref{prop:moment-odd} we have
\[
    \Gamma_d(X(id)) = X(v^1) \cup \cdots \cup X(v^s),
\]
where $v^1, \cdots, v^s$ are the maximal vertices which can be reached from $id$ using a path of degree $d$ or less. Among these $v^i$, some are in $\Y$ while the others are in $\Xc$. Our first goal is to prove that exactly one of the $v^i$ is in $\Y$. It is immediate from Lemma~\ref{lem:OY(d)} that at least one of the $v^i$ is in $\Y$ as long as $d \geq 1$, as there always exists a degree $d$ curve from $id$ to $\OY(d)$. Now take any $v^i \in \Y$; by definition of $A_\Xc(d)$, $v^i$ must be an element of $A_\Xc(d)$, and by maximality and Lemma~\ref{lem:OY(d)} we must have $v^i=\OY(d)$.




There is at most one irreducible component of $\Gamma_d(X(id))$ contained in the closed orbit, otherwise this would contradict the irreducibility result in \cite[Proposition 3.2]{BCMP:qkfin}. Applying the Buch-Mihalcea Recursion, see Proposition~\ref{prop:zd}, we get that 
\[
    \Gamma_d(\Xc(id)) = \Xc(s_{2t_1} \cdot \ldots \cdot s_{2t_1} W_\Xc)  (\text{$s_{2t_1}$ appearing $d$-times})
\]
and we easily compute that
\[
    s_{2t_1} \cdot \ldots \cdot s_{2t_1} W_\Xc = (d+1<d+2<\dots <k-1< \bar d < \dots < \bar 2 <\bar 1).
\]
On the other hand, a quick calculation using Lemma~\ref{lem:comp} shows that the minimal length representative of $\OZ(d)$ is
\[
    (1<d+2<d+3<\dots <k< \overline{d+1} < \dots < \bar 3 <\bar 2).
\]
It is immediate that the image by $\Phi_\Xc$ of the above element is $(d+1<d+2<\dots <k-1< \bar d < \dots < \bar 2 <\bar 1)$.

It follows that $\Gamma_d(X(id))= X(\OY(d)) \cup X(\OZ(d))$ as claimed.
\end{proof}

The situation for an arbitrary Schubert variety $X(w) \subset \Xc$ is as follows.

\begin{prop} \label{prop:crvnbdweyl}
Let $w \in \Wlarge^P \cap \Wodd$ be an odd symplectic minimal length representative such that $X(w) \subset \Xc$ (i.e. $w(1)=1$). Then the cosets $w \cdot \OY(d) \Wlarge^P$ and $w \cdot \OZ(d) \Wlarge^P$ have representatives in $\Wodd$ and 
\[
    \Gamma_d(X(w)) = X(w \cdot \OY(d) \Wlarge^P) \cup X(w \cdot \OZ(d) \Wlarge^P). 
\]
\end{prop}

\begin{proof}
The $d=1$ case follows from \cite[Theorem 9.2]{mihalcea.shifler:qhodd}. For $d>1$, since $\OY(d)=\Oo(d-1) \cdot_k \OY(1)$, the open orbit component follows from the Buch-Mihalcea recursion in $\Y$ while the closed orbit component follows from the recursion in $\Xc$.
\end{proof}

\section{Partitions} \label{s:Partitions}

So far all our results on curve neighborhoods have been expressed in terms of Weyl group elements, specifically in terms of elements of $\Wodd \cap W^P$. However as for Grassmannians there are alternative indexations for Schubert varieties in terms of partitions. These indexations will allow us to refine the results of Section~\ref{s:curvenbhds};  in particular we will be able to say explicitly when the curve neighborhood of a Schubert variety in the closed $\Sp_{2n+1}$-orbit has two components.

In this section we introduce two families of partitions indexing Schubert varieties of  (even or odd) symplectic Grassmannians, namely \emph{$BC$-partitions} in Section~\ref{ss:BC} and \emph{$BKT$-partitions} in Section~\ref{ss:BKT}. In Section~\ref{ss:nbhd-part} we reformulate our results in terms of these indexations.

\subsection{\texorpdfstring{$BC$}{BC}-partitions}\label{ss:BC}

In~\cite{ShiflerWithrow} an indexing set for type $B$ and $C$ isotropic Grassmannians is introduced. The advantage of this set of partitions is that the Bruhat order corresponds to inclusion of the Young diagrams, which is not the case in general for the $BKT$-partitions we will introduce in the next section; the drawback is that codimension cannot be readily computed by summing the parts of the partition. We call these collection of partitions $BC$-partitions.\footnote{The ``BC" comes from the fact that this set of partitions is ``{\bf B}ruhat {\bf C}ompatible" for isotropic Grassmannians in Types {\bf B} and {\bf C}.}

If $1 \leq k \leq N-1$ we denote by $\Part(k,N)$ the set of partitions whose Young diagram fits inside a $k \times (N-k)$ rectangle, namely
\begin{equation}
    \Part(k,N) = \left\{ \mu = (\mu_1 \geq \mu_2 \geq \dots \geq \mu_k \geq 0) \mid \mu_1 \leq N-k \right\}.
\end{equation}

If $N=2n+2$ is even we also encode elements of $\Part(k,2n+2)$ by sequences of $0$s and $1$s, called \emph{$01$-words}, as follows. The boundary of the Young diagram of $\mu \in \Part(k,2n+2)$ consists of $2n+2$ steps, either horizontal of vertical, going from the northeast corner of the $k \times (2n+2-k)$ rectangle to its southwest corner. The total number of vertical steps is $k$. We associate to $\mu$ a $01$-word, denoted by $D(\mu)$, as follows: if the $i$-th step is horizontal we set $D(\mu)(i)=0$, otherwise $D(\mu)=1$.

\begin{defn}
    Let $\BC(k,2n+2)$ denote the set of partitions $\mu \in \Part(k,2n+2)$ such that if $D(\mu)(i)=D(\mu)(2n+3-i)$ for some $1 \leq i \leq n+1$, then $D(\mu)(i)=0$. We also let $\BCodd$ denote the set of partitions in $\BC(k,2n+2)$ whose  first column has $k$ boxes, and we define
    \[
        \BC(k,2n+1) := \{ \lambda \in \Part(k,2n+1) \mid \lambda + 1^k \in \BC(k,2n+2)\} \cong \BCodd.
    \]
    We call the elements of $\BC(k,2n+2)$ \emph{$BC$-partitions}.
\end{defn}

\begin{example}\label{ex:BC_permissible}
    If $\lambda \in \BC(k,2n+1)$ where $\lambda_1=2n+1-k$ then we have
    \[
        \lambda_1-\lambda_2 \geq k-\ell(\lambda),
    \]
    where $\ell(\lambda)$ denotes the number of parts of $\lambda$.
\end{example}

Elements of $\BC(k,2n+2)$ are in bijection with $W^P$, while elements of $\BCodd$, hence $\BC(k,2n+1)$, are in bijection with $\Wodd \cap W^P$. If $\lambda \in \BC(k,2n+1)$ we define
\[
    D(\lambda) := D(\lambda+1^k),
\]
noting that $\lambda+1^k \in \BC(k,2n+2)$.

\begin{lemma}
If $\mu \in \BC(k,2n+2)$, let $r$ be the number of $1$ in the $n+1$ first entries of the associated $01$-word $D(\mu)$. We denote by $1 \leq a_1<\dots<a_r \leq n+1$ the integers such that $D(\mu)(a_i)=1$, and by $1 \leq a_{r+1},\dots,a_k \leq n+1$ those such that $D(\mu)(2n+3-a_i)=1$. 
    
The map $\BC(k,2n+2) \to W^P$ which to $\mu$ associates the Weyl group element 
\[
    (a_1<\dots<a_r<\bar a_k<\dots<\bar a_{r+1})
\]
is a bijection, which restricts to a bijection $\BCodd \to \Wodd \cap W^P$.
\end{lemma}

The proof of the lemma is immediate, and we deduce in particular that the Schubert varieties in $\IG(k,2n+1)$ are indexed by the partitions of $\BCodd$, or equivalently of $\BC(k,2n+1)$. We will use this indexation throughout the section.

\begin{example}
Let $k=5$, $n=7$, and $w=(1<6<\bar{8}<\bar{7}<\bar{2}) \in \Wodd$. Then $\mu=(\mu_1 \geq \mu_2 \geq \mu_3 \geq \mu_4 \geq \mu_5) \in \BCodd \subset \BC(5,2\cdot 7+2)$ is given by $\mu = (11,7,5,5,1)$ and the corresponding partition in $\BC(5,2\cdot 7+1)$ is $\lambda=(10,6,4,4,0)$.

Pictorially, 
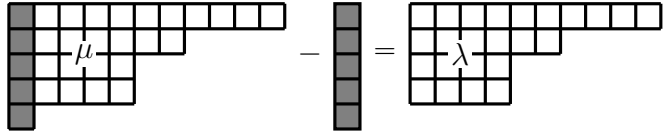
\[
    \begin{tikzpicture}[scale=1/3]
        \draw[fill=gray] (0,0) rectangle (1,5);
        \draw[very thick] (0,0) -- (0,5) ;
        \draw[very thick] (1,0) -- (1,5) ;
        \draw[very thick] (2,1) -- (2,5) ;
        \draw[very thick] (3,1) -- (3,5) ;
        \draw[very thick] (4,1) -- (4,5) ;
        \draw[very thick] (5,1) -- (5,5) ;
        \draw[very thick] (6,3) -- (6,5) ;
        \draw[very thick] (7,3) -- (7,5) ;
        \draw[very thick] (8,4) -- (8,5) ;
        \draw[very thick] (9,4) -- (9,5) ;
        \draw[very thick] (10,4) -- (10,5) ;
        \draw[very thick] (11,4) -- (11,5) ;
        
        \draw[very thick] (0,0) -- (1,0) ;
        \draw[very thick] (0,1) -- (5,1) ;
        \draw[very thick] (0,2) -- (5,2) ;
        \draw[very thick] (0,3) -- (7,3) ;
        \draw[very thick] (0,4) -- (11,4) ;
        \draw[very thick] (0,5) -- (11,5) ;
        
        \node[fill=white,inner sep=1pt] at (3,3) {\large {$\mu$}};
        
        \node[fill=white,inner sep=1pt] at (12,3) {\large {$-$}};
        
        \draw[very thick,fill=gray] (13,0) rectangle (14,5);
        \draw[very thick] (13,1) -- (14,1) ;
        \draw[very thick] (13,2) -- (14,2) ;
        \draw[very thick] (13,3) -- (14,3) ;
        \draw[very thick] (13,4) -- (14,4) ;
        
        \node[fill=white,inner sep=1pt] at (15,3) {\large {$=$}};
        
        \draw[very thick] (16,1) -- (16,5) ;
        \draw[very thick] (17,1) -- (17,5) ;
        \draw[very thick] (18,1) -- (18,5) ;
        \draw[very thick] (19,1) -- (19,5) ;
        \draw[very thick] (20,1) -- (20,5) ;
        \draw[very thick] (21,3) -- (21,5) ;
        \draw[very thick] (22,3) -- (22,5) ;
        \draw[very thick] (23,4) -- (23,5) ;
        \draw[very thick] (24,4) -- (24,5) ;
        \draw[very thick] (25,4) -- (25,5) ;
        \draw[very thick] (26,4) -- (26,5) ;
        
        \draw[very thick] (16,1) -- (20,1) ;
        \draw[very thick] (16,2) -- (20,2) ;
        \draw[very thick] (16,3) -- (22,3) ;
        \draw[very thick] (16,4) -- (26,4) ;
        \draw[very thick] (16,5) -- (26,5) ;
        
        \node[fill=white,inner sep=1pt] at (18,3) {\large {$\lambda$}};
    \end{tikzpicture}
\]
\end{example}

We will introduce the following definition which will be used later to describe certain curve neighborhood components.

\begin{defn}
Let $\mu \in \Part(k,N)$. Let $\mu^t$ denote the transpose of $\mu$, that is, the partition whose Young diagram is obtained by reflecting that of $\mu$ along the northwest-to-southeast diagonal. 

We say that $\mu$ is \emph{$m$-wingtip symmetric}\footnote{To the best of the authors knowledge this property is unnamed in the literature. The part of the name ``wingtip" came from considering Young diagrams as a fixed wing aircraft.} if $m$ is the largest nonnegative integer such that 
\[ 
    D(\mu)(i)=D(\mu^t)(i) 
\] 
for $1 \leq i \leq m$. Furthermore, we will say $\lambda \in \BC(k,2n+1)$ is $m$-wingtip symmetric if $\lambda+1^k \in \BC(k,2n+2)$ is $m$-wingtip symmetric.
\end{defn}

\begin{example}
Consider $\mu=(10 \geq 8 \geq 3 \geq 1 \geq 0) \in \BC(5,2 \cdot 7+2)$. The partition $\mu$ is $3$-wingtip symmetric. It corresponds to the 01-word $0100100000100101$. Here the first and last character, the second and second to last, and the third and third to last character have opposite parity. Pictorially,
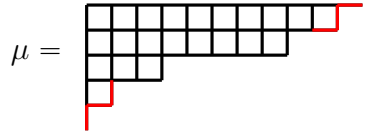
\[
    \begin{tikzpicture}[scale=1/3]
        \draw[very thick] (0,0) -- (0,5) ;
        \draw[very thick] (1,1) -- (1,5) ;
        \draw[very thick] (2,2) -- (2,5) ;
        \draw[very thick] (3,2) -- (3,5) ;
        \draw[very thick] (4,3) -- (4,5) ;
        \draw[very thick] (5,3) -- (5,5) ;
        \draw[very thick] (6,3) -- (6,5) ;
        \draw[very thick] (7,3) -- (7,5) ;
        \draw[very thick] (8,3) -- (8,5) ;
        \draw[very thick] (9,4) -- (9,5) ;
        \draw[very thick] (10,4) -- (10,5) ;
        
        \draw[very thick] (0,1) -- (1,1) ;
        \draw[very thick] (0,2) -- (3,2) ;
        \draw[very thick] (0,3) -- (8,3) ;
        \draw[very thick] (0,4) -- (10,4) ;
        \draw[very thick] (0,5) -- (11,5) ;
        
        \node[fill=white,inner sep=1pt] at (-2,3) {\large {$\mu=$}};
        
        \draw[red,very thick] (0,0) -- (0,1) ;
        \draw[red,very thick] (0,1) -- (1,1) ;
        \draw[red,very thick] (1,1) -- (1,2) ;
        
        \draw[red,very thick] (9,4) -- (10,4) ;
        \draw[red,very thick] (10,4) -- (10,5) ;
        \draw[red,very thick] (10,5) -- (11,5) ;
    \end{tikzpicture}
\]
The first 3 bottom left hand boundary edges yield the partition $\gamma=\yng(1,0)$ and the first 3 top right hand boundary edges yields the partition $\gamma^t$.
\end{example}

For Schubert varieties $X(\lambda)$ intersecting the open orbit we introduce partitions $\lambda^{\Oo(d)}$ which will be used in Section~\ref{ss:nbhd-part} to recover a description of the curve neighborhood $\Gamma_d(X(\lambda))$.

\begin{defn}
Let $\lambda \in \BC(k,2n+1)$ such that $X(\lambda) \cap \Xo \neq \emptyset$ (equivalently, $\lambda_1<2n+1-k$). Then define $\lambda^{\Oo(d)}$ in the following way:
\begin{enumerate}
    \item for $d=1$ define $\lambda^{\Oo(1)}=(\lambda_2-1,\lambda_3-1,\cdots,\lambda_{\ell(\lambda)}-1,0,\cdots,0).$
    \item for $d>1$ define $\lambda^{\Oo(d)}=(\lambda^{\Oo(d-1)})^{\Oo(1)}.$
\end{enumerate}
Explicitly, $\lambda^{\Oo(1)}$ is constructed by removing a hook from $\lambda$.
\end{defn}

For Schubert varieties $X(\lambda)$ contained in the closed $\Sp_{2n+1}$-orbit we similarly introduce the partitions $\lambda^{\OY(d)}$ and $\lambda^{\OZ(d)}$.

\begin{defn}
Let $\lambda \in \BC(k,2n+1)$ be $m$-wingtip symmetric such that $X(\lambda) \subset \Xc$ (equivalently, $\lambda_1=2n+1-k$). Then define $\lambda^{\OY(d)}$ in the following way.
\begin{enumerate} 
    \item If $d=1$ then $\lambda^{\OY(1)}$ is defined by
        \begin{enumerate}
            \item If $D(\lambda)(\bar m)=1$ is the $i$th $1$ in the $01$-word $D(\lambda)$, then 
                \[
                     \lambda^{\OY(1)}:=(\lambda_2-1 \geq \cdots \geq \lambda_{i-1}-1, \lambda_i,\lambda_i, \lambda_{i+1}, \cdots, \lambda_k ); 
                \] 
                Equivalently this means that $D(\lambda)(\bar m)=1$ corresponds to a vertical step at the end of the $i$th row of $\lambda$;
            \item If $D(\lambda)(\bar m)= 0$ corresponds to a horizontal step at the bottom of the $i$th row and in the $j$th column, then  
                \[
                    \lambda^{\OY(1)}:=(\lambda_2-1\geq \cdots \geq \lambda_{i}-1, j, \lambda_{i+1}, \cdots, \lambda_k ). 
                \]
        \end{enumerate}
    \item For $d>1$, 
        \[
            \lambda^{\OY(d)}=\left(\lambda^{\OY(1)}\right)^{O^\circ(d-1)}.
        \]
\end{enumerate}
\end{defn}

\begin{example}
Consider $\lambda=(10,10,3,1,1) \in \BC(5,2\cdot 7+1)$. The partition $\lambda$ is $4$-wingtip symmetric and the corresponding $01$-word is
\[
    D(\lambda) = 1100000001001100.
\]
We see that $D(\lambda)(\bar 4)=1$, and this corresponds the $4$th vertical step, hence $i=4$. Therefore Part (1)(a) of the definition says that $\lambda^{\OY(1)}=(9,2,1,1,1)$. 

Note that the Weyl group element associated with $\lambda$ is $w=(1<2<\bar 7<\bar 4<\bar 3)$, and that $w \cdot \OY(1)=(2<\bar 7<\bar 5<\bar 4<\bar 3)$ by Lemma~\ref{lem:comp}. As expected $w \cdot \OY(1)$ is the Weyl group element associated with $\lambda^{\OY(1)}$. 

We also see that $\lambda^{\OY(2)}=(1)$ and $\lambda^{\OY(3)}=\emptyset$, in agreement with Lemma~\ref{lem:comp}. Pictorially:

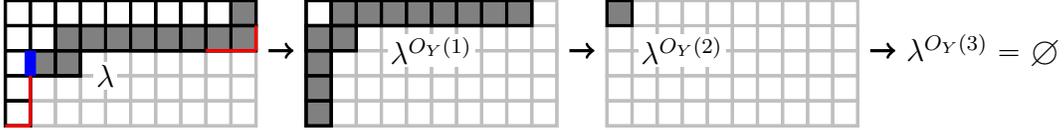
\begin{tikzpicture}[scale=1/3]
\draw[very thick,lightgray] (0,0) rectangle (1,5) ;
\draw[very thick,lightgray] (1,0) rectangle (2,5) ;
\draw[very thick,lightgray] (2,0) rectangle (3,5) ;
\draw[very thick,lightgray] (3,0) rectangle (4,5) ;
\draw[very thick,lightgray] (4,0) rectangle (5,5) ;
\draw[very thick,lightgray] (5,0) rectangle (6,5) ;
\draw[very thick,lightgray] (6,0) rectangle (7,5) ;
\draw[very thick,lightgray] (7,0) rectangle (8,5) ;
\draw[very thick,lightgray] (8,0) rectangle (9,5) ;
\draw[very thick,lightgray] (9,0) rectangle (10,5) ;

\draw[very thick,lightgray] (0,0) rectangle (10,1) ;
\draw[very thick,lightgray] (0,1) rectangle (10,2) ;
\draw[very thick,lightgray] (0,2) rectangle (10,3) ;
\draw[very thick,lightgray] (0,3) rectangle (10,4) ;

\draw[very thick,fill=gray] (9,4) rectangle (10,5) ;
\draw[very thick,fill=gray] (2,3) rectangle (10,4) ;
\draw[very thick,fill=gray] (1,2) rectangle (3,3) ;

\draw[very thick] (0,4) rectangle (10,5) ;
\draw[very thick] (0,3) rectangle (10,4) ;
\draw[very thick] (0,2) rectangle (3,3) ;
\draw[very thick] (0,1) rectangle (1,2) ;
\draw[very thick] (0,0) rectangle (1,1) ;

\draw[very thick] (0,0) rectangle (1,5) ;
\draw[very thick] (1,2) rectangle (2,5) ;
\draw[very thick] (2,2) rectangle (3,5) ;
\draw[very thick] (3,3) rectangle (4,5) ;
\draw[very thick] (5,3) rectangle (6,5) ;
\draw[very thick] (6,3) rectangle (7,5) ;
\draw[very thick] (7,3) rectangle (8,5) ;
\draw[very thick] (8,3) rectangle (9,5) ;
\draw[blue, line width=1.5mm] (1,2) -- (1,3) ;
\draw[red, very thick] (0,0) -- (1,0) ;
\draw[red, very thick] (1,0) -- (1,2) ;
\draw[red, very thick] (10,3) -- (10,4) ;
\draw[red, very thick] (10,3) -- (8,3) ;

\node[fill=white,inner sep=1pt] at (4,2) {\large {$\lambda$}};

\draw[very thick,->] (10.5,3) -- (11.5,3);

\draw[very thick,lightgray] (12,0) rectangle (13,5) ;
\draw[very thick,lightgray] (13,0) rectangle (14,5) ;
\draw[very thick,lightgray] (14,0) rectangle (15,5) ;
\draw[very thick,lightgray] (15,0) rectangle (16,5) ;
\draw[very thick,lightgray] (16,0) rectangle (17,5) ;
\draw[very thick,lightgray] (17,0) rectangle (18,5) ;
\draw[very thick,lightgray] (18,0) rectangle (19,5) ;
\draw[very thick,lightgray] (19,0) rectangle (20,5) ;
\draw[very thick,lightgray] (20,0) rectangle (21,5) ;
\draw[very thick,lightgray] (21,0) rectangle (22,5) ;

\draw[very thick,lightgray] (12,0) rectangle (22,1) ;
\draw[very thick,lightgray] (12,1) rectangle (22,2) ;
\draw[very thick,lightgray] (12,2) rectangle (22,3) ;
\draw[very thick,lightgray] (12,3) rectangle (22,4) ;

\draw[very thick,fill=gray] (13,4) rectangle (21,5) ;
\draw[very thick,fill=gray] (12,0) rectangle (13,4) ;
\draw[very thick,fill=gray] (13,3) rectangle (14,4) ;

\draw[very thick] (12,4) rectangle (21,5) ;
\draw[very thick] (12,3) rectangle (14,4) ;
\draw[very thick] (12,2) rectangle (13,3) ;
\draw[very thick] (12,1) rectangle (13,2) ;
\draw[very thick] (12,0) rectangle (13,1) ;

\draw[very thick] (12,3) rectangle (13,5) ;
\draw[very thick] (13,4) rectangle (14,5) ;
\draw[very thick] (14,4) rectangle (15,5) ;
\draw[very thick] (15,4) rectangle (16,5) ;
\draw[very thick] (16,4) rectangle (17,5) ;
\draw[very thick] (17,4) rectangle (18,5) ;
\draw[very thick] (18,4) rectangle (19,5) ;
\draw[very thick] (19,4) rectangle (20,5) ;

\node[fill=white,inner sep=1pt] at (17,3) {\large {$\lambda^{\OY(1)}$}};

\draw[very thick,->] (22.5,3) -- (23.5,3);

\draw[very thick,lightgray] (24,0) rectangle (25,5) ;
\draw[very thick,lightgray] (25,0) rectangle (26,5) ;
\draw[very thick,lightgray] (26,0) rectangle (27,5) ;
\draw[very thick,lightgray] (27,0) rectangle (28,5) ;
\draw[very thick,lightgray] (28,0) rectangle (29,5) ;
\draw[very thick,lightgray] (29,0) rectangle (30,5) ;
\draw[very thick,lightgray] (30,0) rectangle (31,5) ;
\draw[very thick,lightgray] (31,0) rectangle (32,5) ;
\draw[very thick,lightgray] (32,0) rectangle (33,5) ;
\draw[very thick,lightgray] (33,0) rectangle (34,5) ;

\draw[very thick,lightgray] (24,0) rectangle (34,1) ;
\draw[very thick,lightgray] (24,1) rectangle (34,2) ;
\draw[very thick,lightgray] (24,2) rectangle (34,3) ;
\draw[very thick,lightgray] (24,3) rectangle (34,4) ;

\draw[very thick,fill=gray] (24,4) rectangle (25,5) ;

\node[fill=white,inner sep=1pt] at (27,3) {\large {$\lambda^{\OY(2)}$}};

\draw[very thick,->] (34.5,3) -- (35.5,3);

\node[fill=white,inner sep=1pt] at (39,3) {\large {$\lambda^{\OY(3)}=\emptyset$}};
\end{tikzpicture}
\end{example}

\begin{example}
Consider $\lambda=(10,9,9,3) \in \BC(5,2\cdot 7+1)$. The partition $\lambda$ is $4$-wingtip symmetric and the corresponding $01$-word is
\[
    D(\lambda) = 1011000000100010.
\]
We see that $D(\lambda)(\bar 4)=0$, and this corresponds a horizontal step at the bottom of the $4$th row and in the $3$rd column, hence $i=4$ and $j=3$. Therefore Part (1)(b) of the definition says that $\lambda^{\OY(1)}=(8,8,2,2)$. Note that the Weyl group element associated with $\lambda$ is $w=(1<3<4<\bar 6<\bar 2)$, and that $w \cdot \OY(1)=(3<4<\bar 6<\bar 5<\bar 2)$ by Lemma~\ref{lem:comp}, which is the Weyl group element associated with $\lambda^{\OY(1)}$.

We also see that $\lambda^{\OY(2)}=(7,1,1)$ and $\lambda^{\OY(3)}=\emptyset$, in agreement with Lemma~\ref{lem:comp}. Pictorially:
\begin{tikzpicture}[scale=1/3]
\draw[very thick,lightgray] (0,0) rectangle (1,5) ;
\draw[very thick,lightgray] (1,0) rectangle (2,5) ;
\draw[very thick,lightgray] (2,0) rectangle (3,5) ;
\draw[very thick,lightgray] (3,0) rectangle (4,5) ;
\draw[very thick,lightgray] (4,0) rectangle (5,5) ;
\draw[very thick,lightgray] (5,0) rectangle (6,5) ;
\draw[very thick,lightgray] (6,0) rectangle (7,5) ;
\draw[very thick,lightgray] (7,0) rectangle (8,5) ;
\draw[very thick,lightgray] (8,0) rectangle (9,5) ;
\draw[very thick,lightgray] (9,0) rectangle (10,5) ;

\draw[very thick,lightgray] (0,0) rectangle (10,1) ;
\draw[very thick,lightgray] (0,1) rectangle (10,2) ;
\draw[very thick,lightgray] (0,2) rectangle (10,3) ;
\draw[very thick,lightgray] (0,3) rectangle (10,4) ;

\draw[very thick,fill=gray] (8,4) rectangle (10,5) ;
\draw[very thick,fill=gray] (8,2) rectangle (9,4) ;
\draw[very thick,fill=gray] (2,2) rectangle (8,3) ;
\draw[very thick,fill=gray] (2,1) rectangle (3,2) ;

\draw[very thick] (0,4) rectangle (10,5) ;
\draw[very thick] (0,3) rectangle (9,4) ;
\draw[very thick] (0,2) rectangle (9,3) ;
\draw[very thick] (0,1) rectangle (3,2) ;
\draw[very thick] (0,4) rectangle (1,5) ;
\draw[very thick] (1,1) rectangle (2,5) ;
\draw[very thick] (2,1) rectangle (3,5) ;
\draw[very thick] (3,2) rectangle (4,5) ;
\draw[very thick] (4,2) rectangle (5,5) ;
\draw[very thick] (5,2) rectangle (6,5) ;
\draw[very thick] (6,2) rectangle (7,5) ;
\draw[very thick] (7,2) rectangle (8,5) ;
\draw[very thick] (8,2) rectangle (9,5) ;

\draw[blue, line width=1.5mm] (2,1) -- (2,2) ;
\draw[red, very thick] (0,0) -- (0,1) ;
\draw[red, very thick] (0,1) -- (2,1) ;
\draw[red, very thick] (9,4) -- (10,4) ;
\draw[red, very thick] (9,2) -- (9,4) ;

\node[fill=white,inner sep=1pt] at (4,1) {\large {$\lambda$}};

\draw[very thick,->] (10.5,3) -- (11.5,3);

\draw[very thick,lightgray] (12,0) rectangle (13,5) ;
\draw[very thick,lightgray] (13,0) rectangle (14,5) ;
\draw[very thick,lightgray] (14,0) rectangle (15,5) ;
\draw[very thick,lightgray] (15,0) rectangle (16,5) ;
\draw[very thick,lightgray] (16,0) rectangle (17,5) ;
\draw[very thick,lightgray] (17,0) rectangle (18,5) ;
\draw[very thick,lightgray] (18,0) rectangle (19,5) ;
\draw[very thick,lightgray] (19,0) rectangle (20,5) ;
\draw[very thick,lightgray] (20,0) rectangle (21,5) ;
\draw[very thick,lightgray] (21,0) rectangle (22,5) ;

\draw[very thick,lightgray] (12,0) rectangle (22,1) ;
\draw[very thick,lightgray] (12,1) rectangle (22,2) ;
\draw[very thick,lightgray] (12,2) rectangle (22,3) ;
\draw[very thick,lightgray] (12,3) rectangle (22,4) ;

\draw[very thick,fill=gray] (19,4) rectangle (20,5) ;
\draw[very thick,fill=gray] (13,3) rectangle (20,4) ;
\draw[very thick,fill=gray] (13,1) rectangle (14,3) ;
\draw[very thick,fill=gray] (12,1) rectangle (13,2) ;

\draw[very thick] (12,4) rectangle (20,5) ;
\draw[very thick] (12,3) rectangle (20,4) ;
\draw[very thick] (12,2) rectangle (14,3) ;
\draw[very thick] (12,1) rectangle (14,2) ;
\draw[very thick] (13,1) rectangle (14,5) ;
\draw[very thick] (15,3) rectangle (16,5) ;
\draw[very thick] (17,3) rectangle (18,5) ;
\draw[very thick] (19,3) rectangle (20,5) ;

\node[fill=white,inner sep=1pt] at (17,2) {\large {$\lambda^{\OY(1)}$}};

\draw[very thick,->] (22.5,3) -- (23.5,3);

\draw[very thick,lightgray] (24,0) rectangle (25,5) ;
\draw[very thick,lightgray] (25,0) rectangle (26,5) ;
\draw[very thick,lightgray] (26,0) rectangle (27,5) ;
\draw[very thick,lightgray] (27,0) rectangle (28,5) ;
\draw[very thick,lightgray] (28,0) rectangle (29,5) ;
\draw[very thick,lightgray] (29,0) rectangle (30,5) ;
\draw[very thick,lightgray] (30,0) rectangle (31,5) ;
\draw[very thick,lightgray] (31,0) rectangle (32,5) ;
\draw[very thick,lightgray] (32,0) rectangle (33,5) ;
\draw[very thick,lightgray] (33,0) rectangle (34,5) ;

\draw[very thick,lightgray] (24,0) rectangle (34,1) ;
\draw[very thick,lightgray] (24,1) rectangle (34,2) ;
\draw[very thick,lightgray] (24,2) rectangle (34,3) ;
\draw[very thick,lightgray] (24,3) rectangle (34,4) ;

\draw[very thick,fill=gray] (24,4) rectangle (31,5) ;
\draw[very thick,fill=gray] (25,4) rectangle (26,5) ;
\draw[very thick,fill=gray] (27,4) rectangle (28,5) ;
\draw[very thick,fill=gray] (29,4) rectangle (30,5) ;
\draw[very thick,fill=gray] (24,3) rectangle (25,4) ;
\draw[very thick,fill=gray] (24,2) rectangle (25,3) ;

\node[fill=white,inner sep=1pt] at (27,3) {\large {$\lambda^{\OY(2)}$}};

\draw[very thick,->] (34.5,3) -- (35.5,3);

\node[fill=white,inner sep=1pt] at (39,3) {\large {$\lambda^{\OY(3)}=\emptyset$}};
\end{tikzpicture}
\end{example}

\begin{defn}
Let $\lambda \in \BC(k,2n+1)$ such that $X(\lambda) \subset \Xc$ (equivalently, $\lambda_1=2n+1-k$). Then define $\lambda^{\OZ(d)}$ in the following way:
\begin{enumerate}
    \item for $d=1$ define $\lambda^{\OZ(1)}=(\lambda_1,\lambda_3-1,\cdots,\lambda_{\ell(\lambda)}-1,0,\cdots,0)$;
    \item for $d>1$ define $\lambda^{\OZ(d)}=(\lambda^{\OZ(d-1)})^{\OZ(1)}.$
\end{enumerate}
Explicitly, $\lambda^{\OZ(1)}$ is constructed by removing the second row of $\lambda$ as well as a box from all subsequent rows.
\end{defn}

\begin{example}
Consider $\lambda=(10,10,3,1,1) \in \BC(5,2\cdot 7+1)$. We see that $\lambda^{\OZ(1)}=(10,2)$. 

Note that the Weyl group element associated with $\lambda$ is $w=(1<2<\bar 7<\bar 4<\bar 3)$, and that $w \cdot \OZ(1)=(1<\bar 7<\bar 4<\bar 3<\bar 2)$ by Lemma~\ref{lem:comp}. As expected $w \cdot \OZ(1)$ is the Weyl group element associated with $\lambda^{\OZ(1)}$. 

We also see that $\lambda^{\OZ(d)}=(10)$ for $d \geq 2$, in agreement with Lemma~\ref{lem:comp}.
\end{example}

\subsection{BKT-partitions}\label{ss:BKT}

In~\cite{BKT2} an alternative indexation is introduced for Schubert varieties of odd symplectic Grassmannians, in terms of what the authors call \emph{$k$-strict partitions}, and which we will call $BKT$-partitions. This is analogous the classical indexation of Schubert varieties of type A Grassmannians using partitions or Young diagrams. In particular with this indexation the number of boxes in the Young diagram indexing a Schubert variety recovers the codimension of said variety. In~\cite{pech:quantum} a variant of this indexation is extended to the odd symplectic Grassmannian $X=\IG(k,2n+1)$.

\begin{defn}
We say that a partition $\beta \in \Part(k,2n+2)$ is \emph{$(n+1-k)$-strict} if $\beta_j>\beta_{j+1}$ whenever $\beta_j>n+1-k$. We denote the set of such partitions by $\BKT(k,2n+2)$. We also let $\BKTodd$ denote the set of partitions in $\beta \in \BKT(k,2n+2)$ such that either $\beta_1=2n+2-k$ or $\beta_k \geq 1$; in other words, if the first column is not full, then the first row must be full.
We define
\begin{align*}
    \BKT(k,2n+1) &:= \{ \beta-1^k \mid \beta \in \BKT(k,2n+2)\} \cong \BKTodd.
\end{align*}
We call the elements of $\BKT(k,2n+2)$ \emph{$BKT$-partitions}.
\end{defn}

There is a bijection between $\BKT(k,2n+2)$ and the set $\Wlarge^P$ of minimal length representatives given by:
\begin{align*}
    \beta \mapsto w &\text{ is defined by $w(j)=2n+3-k-\beta_j+\#\{ i<j \mid \beta_i+\beta_j \leq 2(n+1-k)+j-i\}$},\\
    w \mapsto \beta &\text{ is defined by $\beta_j=2n+3-k-w(j)+\#\{i<j \mid w(i)+w(j)>2n+3\}$},
\end{align*}
see \cite{BKT2}*{Proposition 4.3}. Recall that the minimal length representative of the element $w_0$ defined in Equation \eqref{E:w0} indexes $\IG(k,2n+1)$ as a Schubert variety inside $\IG(k,2n+2)$. Under the bijection above, the coset of $w_0 \Wlarge_P$ corresponds to the $(n+1-k)$-strict partition $1^k:= (1,1, \ldots , 1)$ if $k < n+1$ and to $(k,0,\ldots , 0 )$ if $k=n+1$. The minimal length representatives for odd symplectic permutations $w \in \Wodd$ are in bijection with the subset $\BKTodd$ of $\BKT(k,2n+2)$ consisting of those $(n+1-k)$-strict partitions satisfying the additional condition that if $\beta_k = 0$ then $\beta_1 = 2n+2 - k$. \begin{footnote} {One word of caution: the Bruhat order does not translate into partition inclusion. For example, $(2n+2-k, 0, \ldots , 0) \leq (1, 1, \ldots 1) $ in the Bruhat order for $k < n+1$.}\end{footnote} 

The equivalent indexing set $\BKT(k,2n+1)$ is introduced in~\cite{pech:quantum}. It is more convenient in the context of the odd symplectic Grassmannians since the sum of the parts of $\alpha \in \BKT(k,2n+1)$ still indexes the codimension in $\IG(k,2n+1)$ of the corresponding Schubert variety. Namely, for $\alpha \in \BKT(k,2n+1)$ define $|\alpha| = \alpha_1 + \ldots + \alpha_k$. If $w$ corresponds to $\alpha$ then $\ell(w)=k(2n+1-k)-\frac{k(k-1)}{2}-|\alpha|$, i.e. the codimension of the Schubert variety $X(w)$ in $X$ equals $|\alpha|$; see \cite{BKT2}*{Proposition 4.4} and \cite{pech:quantum}*{Section 1.1.1}. 

Pictorially, the partitions in $\BKT(k,2n+1)$ are obtained by removing the full first column $1^k$ from the partitions in $\BKT(k,2n+2)$, regardless of whether a part equal to $0$ is present or not.

\begin{example}
Let $k=5$, $n=7$, and $w=(1<6<\bar{8}<\bar{7}<\bar{2}) \in \Wodd$. Then $\beta=(\beta_1 \geq \beta_2 \geq \beta_3 \geq \beta_4 \geq \beta_5) \in \BKT(k,2n+2)$ is given by $\beta = (11, 6, 3,3,0)$ and the corresponding partition in $\BKT(k,2n+1)$ is $\alpha=(10,5,2,2,-1)$.
Pictorially, 
\[ 
    \tableau{6}{&{}&{}&{}&{}&{}&{}&{}&{}&{}&{}&{}\\&{}&{}&{}&{}&{}&{}\\&{}&{}&{}\\ &{}&{}&{}}   - \tableau{6}{{}&\\ {}&\\ {}&\\ {}&\\ {}} = \tableau{6}{&{}&{}&{}&{}&{}&{}&{}&{}&{}&{}\\&{}&{}&{}&{}&{}\\&{}&{}\\ &{}&{} \\ {}&&&&&&& } 
\]   
\end{example}

\begin{example} Let $k= n+1 =5$, so $\IG(5,9) \simeq \IG(4,8)$ is the Lagrangian Grassmannian. Then the codimension $0$ class is the $(-1)$-strict partition $\alpha = (4,-1,-1,-1,-1)= \tableau{5}{&{}&{}&{}&{}\\{}\\{}\\{}\\{}} $.
\end{example}

\begin{defn}\label{defn:Ypartitions}
Let $\alpha \in \BKT(k,2n+1)$ where $\alpha_1<2n+1-k$. Define $\alpha^{\Oo(d)}$ in the following way:
\begin{enumerate}
    \item If $\alpha_1+\alpha_j > 2(n-k)+j-1$ for all $2 \leq j \leq k$ then define 
    \[ 
        \alpha^{\Oo(1)}=(\alpha_2 \geq \alpha_3 \geq \cdots \geq \alpha_k \geq 0) \in \BKT(k,2n+1); 
    \] 
    \item Otherwise, find the smallest $j$ such that $\alpha_1+\alpha_j \leq 2(n-k)+j-1$. Define 
    \[ 
        \alpha^{\Oo(1)}=(\alpha_2 \geq \alpha_3 \geq \cdots \geq \alpha_{j-1} \geq \alpha_j-1 \geq \cdots \geq \alpha_{k}-1 \geq 0) \in  \BKT(k,2n+1) 
    \] where $-1$'s are replaced by 0;
\item Define $\alpha^{\Oo(d)}=\left( \alpha^{\Oo(d-1)}\right)^{\Oo(1)}$ for $d>1$.
\end{enumerate}
\end{defn}

\begin{example}
Consider the case $n=7$ and $k=6$. Let $w=(6<8<\bar{7}<\bar{5}<\bar{3}<\bar{2}) \in W^P \cap \Wodd$. Then $w\cdot \Oo(1) = (8<\bar{7}<\bar{5}<\bar{4}<\bar{3}<\bar{2})$ and 

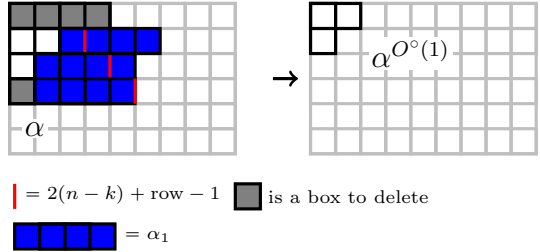
\begin{center}
\begin{tikzpicture}[scale=1/3]
\draw[very thick,lightgray] (0,0) rectangle (1,6) ;
\draw[very thick,lightgray] (1,0) rectangle (2,6) ;
\draw[very thick,lightgray] (2,0) rectangle (3,6) ;
\draw[very thick,lightgray] (3,0) rectangle (4,6) ;
\draw[very thick,lightgray] (4,0) rectangle (5,6) ;
\draw[very thick,lightgray] (5,0) rectangle (6,6) ;
\draw[very thick,lightgray] (6,0) rectangle (7,6) ;
\draw[very thick,lightgray] (7,0) rectangle (8,6) ;
\draw[very thick,lightgray] (8,0) rectangle (9,6) ;
\draw[very thick,lightgray] (0,5) rectangle (9,6) ;
\draw[very thick,lightgray] (0,0) rectangle (9,1) ;
\draw[very thick,lightgray] (0,1) rectangle (9,2) ;
\draw[very thick,lightgray] (0,2) rectangle (9,3) ;
\draw[very thick,lightgray] (0,3) rectangle (9,4) ;

\draw[very thick,fill=gray] (0,5) rectangle (4,6) ;
\draw[very thick,fill=gray] (0,2) rectangle (1,3) ;
\draw[very thick,fill=blue] (2,4) rectangle (6,5) ;
\draw[very thick,fill=blue] (1,3) rectangle (5,4) ;
\draw[very thick,fill=blue] (1,2) rectangle (5,3) ;

\draw[very thick,red] (3,4) rectangle (3,5) ;
\draw[very thick,red] (4,3) rectangle (4,4) ;
\draw[very thick,red] (5,2) rectangle (5,3) ;

\draw[very thick,red] (.2,-1.2) rectangle (.2,-2.2) ;
\node[fill=white,inner sep=1pt] at (4.5,-1.6) {\tiny {$=2(n-k)+\text{row}-1$}};
\draw[very thick,fill=blue] (.2,-2.8) rectangle (4.2,-3.8) ;
\draw[very thick] (1.2,-3.8) rectangle (2.2,-2.8) ;
\draw[very thick] (2.2,-3.8) rectangle (3.2,-2.8) ;
\node[fill=white,inner sep=1pt] at (5.5,-3.3) {\tiny {$=\alpha_1$}};

\draw[very thick,fill=gray] (9,-1.2) rectangle (10,-2.2) ;
\node[fill=white,inner sep=1pt] at (13.5,-1.7) {\tiny {is a box to delete}};

\draw[very thick] (0,5) rectangle (4,6) ;
\draw[very thick] (0,4) rectangle (2,5) ;
\draw[very thick] (0,3) rectangle (1,4) ;
\draw[very thick] (0,2) rectangle (1,3) ;
\draw[very thick] (1,4) rectangle (2,6) ;
\draw[very thick] (2,5) rectangle (3,6) ;
\draw[very thick] (2,2) rectangle (3,4) ;
\draw[very thick] (3,2) rectangle (4,3) ;
\draw[very thick] (4,4) rectangle (5,5) ;

\node[fill=white,inner sep=1pt] at (1,1) {\large {$\alpha$}};

\draw[very thick,->] (10.5,3) -- (11.5,3);

\draw[very thick,lightgray] (12,0) rectangle (13,6) ;
\draw[very thick,lightgray] (13,0) rectangle (14,6) ;
\draw[very thick,lightgray] (14,0) rectangle (15,6) ;
\draw[very thick,lightgray] (15,0) rectangle (16,6) ;
\draw[very thick,lightgray] (16,0) rectangle (17,6) ;
\draw[very thick,lightgray] (17,0) rectangle (18,6) ;
\draw[very thick,lightgray] (18,0) rectangle (19,6) ;
\draw[very thick,lightgray] (19,0) rectangle (20,6) ;
\draw[very thick,lightgray] (20,0) rectangle (21,6) ;

\draw[very thick,lightgray] (12,0) rectangle (21,1) ;
\draw[very thick,lightgray] (12,2) rectangle (21,3) ;
\draw[very thick,lightgray] (12,4) rectangle (21,5) ;
\draw[very thick,lightgray] (12,5) rectangle (21,6) ;

\draw[very thick] (12,5) rectangle (14,6) ;
\draw[very thick] (12,4) rectangle (13,6) ;

\node[fill=white,inner sep=1pt] at (16,4) {\large {$\alpha^{\Oo(1)}$}};
\end{tikzpicture}
\end{center}

This uses part (2) of Definition \ref{defn:Ypartitions} and we have that $j=4$ since the fourth row is the first occurrence of the red line being to the right of the shaded blue region.
\end{example}

\begin{remark}
    Note the similarity between the notation $\alpha^{\Oo(1)}$ above, and the notation $\lambda^{\Oo(1)}$ in Section~\ref{ss:BC}; this is because if $\alpha \in \BKT(k,2n+1)$ and $\lambda \in \BC(k,2n+1)$ correspond to the same Weyl group element $w$, then $\alpha^{\Oo(1)}$ and $\lambda^{\Oo(1)}$ will both correspond to $w \cdot \Oo(1)$.
\end{remark}

Let $\alpha \in \BKT(k,2n+1)$ and $\ell_i(\alpha)=\max \{j \mid \alpha_j >i\}$.

\begin{defn}
Let $\alpha \in \BKT(k,2n+1)$ where $\alpha_1=2n+1-k$. Define $\alpha^{\OY(d)}$ in the following way:
\begin{enumerate}
    \item For $d=1$, 
    \[ 
        \alpha^{\OY(1)} = (\alpha_2 \geq \alpha_3 \geq \cdots \geq \alpha_{\ell_{-1}(\alpha)} \geq \overbrace{0 \geq \cdots \geq 0}^{k+1-\ell_{-1}(\alpha)}) \in \BKT(k,2n+1); 
    \]
    \item For $d>1$, 
    \[
        \alpha^{\OY(d)} = \left(\alpha^{\OY(1)}\right)^{\Oo(d-1)}\]
\end{enumerate}
\end{defn}

\begin{example}
Consider the case $n=7$ and $k=6$. Let $w = (1<7<\bar 8<\bar 5<\bar 4 < \bar 2 ) \in W^P \cap \Wodd$. Then $w \cdot \OY(1)=(7<\bar 8<\bar 5< \bar 4< \bar 3< \bar 2)$ and 

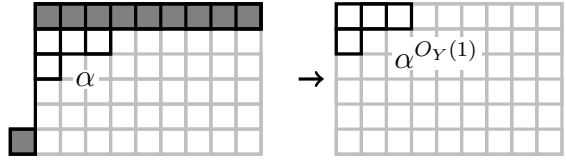
\begin{center}
\begin{tikzpicture}[scale=1/3]
\draw[very thick,lightgray] (0,0) rectangle (1,6) ;
\draw[very thick,lightgray] (1,0) rectangle (2,6) ;
\draw[very thick,lightgray] (2,0) rectangle (3,6) ;
\draw[very thick,lightgray] (3,0) rectangle (4,6) ;
\draw[very thick,lightgray] (4,0) rectangle (5,6) ;
\draw[very thick,lightgray] (5,0) rectangle (6,6) ;
\draw[very thick,lightgray] (6,0) rectangle (7,6) ;
\draw[very thick,lightgray] (7,0) rectangle (8,6) ;
\draw[very thick,lightgray] (8,0) rectangle (9,6) ;

\draw[very thick,lightgray] (0,0) rectangle (9,1) ;
\draw[very thick,lightgray] (0,1) rectangle (9,2) ;
\draw[very thick,lightgray] (0,2) rectangle (9,3) ;
\draw[very thick,lightgray] (0,3) rectangle (9,4) ;

\draw[very thick,fill=gray] (0,5) rectangle (9,6) ;
\draw[very thick,fill=gray] (-1,0) rectangle (0,1) ;

\draw[very thick] (0,5) rectangle (9,6) ;
\draw[very thick] (0,4) rectangle (3,5) ;
\draw[very thick] (0,3) rectangle (1,4) ;
\draw[very thick] (0,1) rectangle (0,3) ;
\draw[very thick] (-1,0) rectangle (0,1) ;
\draw[very thick] (0,3) rectangle (1,6) ;
\draw[very thick] (2,4) rectangle (3,6) ;
\draw[very thick] (4,5) rectangle (5,6) ;
\draw[very thick] (6,5) rectangle (7,6) ;
\draw[very thick] (8,5) rectangle (9,6) ;

\node[fill=white,inner sep=1pt] at (2,3) {\large {$\alpha$}};

\draw[very thick,->] (10.5,3) -- (11.5,3);

\draw[very thick,lightgray] (12,0) rectangle (13,6) ;
\draw[very thick,lightgray] (13,0) rectangle (14,6) ;
\draw[very thick,lightgray] (14,0) rectangle (15,6) ;
\draw[very thick,lightgray] (15,0) rectangle (16,6) ;
\draw[very thick,lightgray] (16,0) rectangle (17,6) ;
\draw[very thick,lightgray] (17,0) rectangle (18,6) ;
\draw[very thick,lightgray] (18,0) rectangle (19,6) ;
\draw[very thick,lightgray] (19,0) rectangle (20,6) ;
\draw[very thick,lightgray] (20,0) rectangle (21,6) ;

\draw[very thick,lightgray] (12,0) rectangle (21,1) ;
\draw[very thick,lightgray] (12,2) rectangle (21,3) ;
\draw[very thick,lightgray] (12,4) rectangle (21,5) ;
\draw[very thick,lightgray] (12,5) rectangle (21,6) ;

\draw[very thick] (12,5) rectangle (15,6) ;
\draw[very thick] (12,4) rectangle (13,5) ;
\draw[very thick] (13,5) rectangle (14,6) ;

\node[fill=white,inner sep=1pt] at (16,4) {\large {$\alpha^{\OY(1)}$}};
\end{tikzpicture}
\end{center}

This uses part (2) of Definition \ref{defn:Zpartitions} and $j=4$.
\end{example}

\begin{defn}\label{defn:Zpartitions}
Let $\alpha \in  \BKT(k,2n+1)$ where $\alpha_1=2n+1-k$. Define $\alpha^{\OZ(d)}$ in the following way:
\begin{enumerate}
    \item If $\alpha_2+\alpha_j > 2(n-k)+j-2$ for all $3 \leq j \leq k$ then define 
    \[ 
        \alpha^{\OZ(1)}=(\alpha_1 \geq \alpha_3 \geq \cdots \geq \alpha_k \geq -1) \in \BKT(k,2n+1); 
    \] 
    \item Otherwise, find the smallest $j$ such that $\alpha_2+\alpha_j \leq 2(n-k)+j-2$. Define 
    \[ 
        \alpha^{\OZ(1)}=(\alpha_1 \geq \alpha_3 \geq \cdots \geq \alpha_{j-1} \geq \alpha_j-1 \geq \cdots \geq \alpha_{k}-1 \geq -1) \in \BKT(k,2n+1) 
    \] 
    where $-2$'s are replaced by $-1$'s;
\item Define $\alpha^{\OZ(d)}=\left( \alpha^{\OZ(d-1)}\right)^{\OZ(1)}$ for $d>1$.
\end{enumerate}
\end{defn}

\begin{example}
Consider the case $n=7$ and $k=6$. Let $w=(1<7<\bar 8<\bar 5<\bar 4 < \bar 2 ) \in W^P \cap \Wodd$. Then $w \cdot \OZ(1)=(1<\bar 8<\bar 5< \bar 4< \bar 3< \bar 2)$ and 

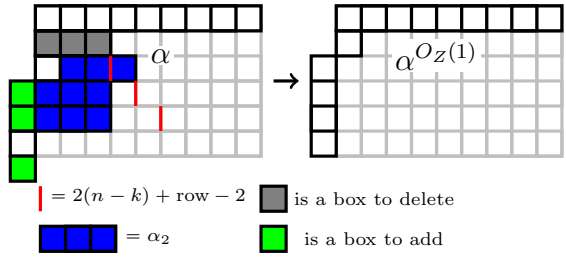
\begin{center}
\begin{tikzpicture}[scale=1/3]
\draw[very thick,lightgray] (0,0) rectangle (1,6) ;
\draw[very thick,lightgray] (1,0) rectangle (2,6) ;
\draw[very thick,lightgray] (2,0) rectangle (3,6) ;
\draw[very thick,lightgray] (3,0) rectangle (4,6) ;
\draw[very thick,lightgray] (4,0) rectangle (5,6) ;
\draw[very thick,lightgray] (5,0) rectangle (6,6) ;
\draw[very thick,lightgray] (6,0) rectangle (7,6) ;
\draw[very thick,lightgray] (7,0) rectangle (8,6) ;
\draw[very thick,lightgray] (8,0) rectangle (9,6) ;
\draw[very thick,lightgray] (0,5) rectangle (9,6) ;
\draw[very thick,lightgray] (0,0) rectangle (9,1) ;
\draw[very thick,lightgray] (0,1) rectangle (9,2) ;
\draw[very thick,lightgray] (0,2) rectangle (9,3) ;
\draw[very thick,lightgray] (0,3) rectangle (9,4) ;

\draw[very thick,fill=gray] (0,4) rectangle (3,5) ;
\draw[very thick, fill=green] (-1,1) rectangle (0,2) ;
\draw[very thick, fill=green] (-1,2) rectangle (0,3) ;

\draw[very thick] (0,5) rectangle (9,6) ;
\draw[very thick] (0,4) rectangle (3,5) ;
\draw[very thick] (0,3) rectangle (1,4) ;
\draw[very thick] (0,1) rectangle (0,3) ;
\draw[very thick] (-1,0) rectangle (0,1) ;
\draw[very thick] (1,4) rectangle (2,6) ;
\draw[very thick] (3,5) rectangle (4,6) ;
\draw[very thick] (5,5) rectangle (6,6) ;
\draw[very thick] (7,5) rectangle (8,6) ;

\draw[very thick,fill=blue] (1,3) rectangle (4,4) ;
\draw[very thick,fill=blue] (0,2) rectangle (3,3) ;
\draw[very thick,fill=blue] (0,1) rectangle (3,2) ;
\draw[very thick,red] (3,3) rectangle (3,4) ;
\draw[very thick,red] (4,2) rectangle (4,3) ;
\draw[very thick,red] (5,1) rectangle (5,2) ;
\draw[very thick] (1,1) rectangle (2,4) ;

\draw[very thick,red] (.2,-1.2) rectangle (.2,-2.2) ;
\node[fill=white,inner sep=1pt] at (4.5,-1.6) {\tiny {$=2(n-k)+\text{row}-2$}};
\draw[very thick,fill=blue] (.2,-2.8) rectangle (3.2,-3.8) ;
\draw[very thick] (1.2,-3.8) rectangle (2.2,-2.8) ;
\draw[very thick] (2.2,-3.8) rectangle (3.2,-2.8) ;
\node[fill=white,inner sep=1pt] at (4.5,-3.3) {\tiny {$=\alpha_2$}};
\draw[very thick,fill=gray] (9,-1.2) rectangle (10,-2.2) ;
\node[fill=white,inner sep=1pt] at (13.5,-1.7) {\tiny {is a box to delete}};
\draw[very thick,fill=green] (9,-2.8) rectangle (10,-3.8) ;
\node[fill=white,inner sep=1pt] at (13.5,-3.3) {\tiny {is a box to add}};
\draw[very thick,fill=green] (-1,-1) rectangle (0,0) ;

\node[fill=white,inner sep=1pt] at (5,4) {\large {$\alpha$}};

\draw[very thick,->] (9.5,3) -- (10.5,3);

\draw[very thick,lightgray] (12,0) rectangle (13,6) ;
\draw[very thick,lightgray] (13,0) rectangle (14,6) ;
\draw[very thick,lightgray] (14,0) rectangle (15,6) ;
\draw[very thick,lightgray] (15,0) rectangle (16,6) ;
\draw[very thick,lightgray] (16,0) rectangle (17,6) ;
\draw[very thick,lightgray] (17,0) rectangle (18,6) ;
\draw[very thick,lightgray] (18,0) rectangle (19,6) ;
\draw[very thick,lightgray] (19,0) rectangle (20,6) ;
\draw[very thick,lightgray] (20,0) rectangle (21,6) ;

\draw[very thick,lightgray] (12,0) rectangle (21,1) ;
\draw[very thick,lightgray] (12,2) rectangle (21,3) ;
\draw[very thick,lightgray] (12,4) rectangle (21,5) ;
\draw[very thick,lightgray] (12,5) rectangle (21,6) ;

\draw[very thick] (12,5) rectangle (21,6) ;
\draw[very thick] (12,4) rectangle (13,5) ;
\draw[very thick] (11,0) rectangle (12,4) ;
\draw[very thick] (13,5) rectangle (14,6) ;
\draw[very thick] (15,5) rectangle (16,6) ;
\draw[very thick] (17,5) rectangle (18,6) ;
\draw[very thick] (19,5) rectangle (20,6) ;
\draw[very thick] (11,1) rectangle (12,2) ;
\draw[very thick] (11,3) rectangle (12,3) ;

\node[fill=white,inner sep=1pt] at (16,4) {\large {$\alpha^{\OZ(1)}$}};
\end{tikzpicture}
\end{center}

This uses part (2) of Definition~\ref{defn:Zpartitions} and we have that $j=4$ since the fourth row is the first occurrence of the red line being to the right of the shaded blue region.
\end{example}

\subsection{Curve neighborhoods in terms of partitions}\label{ss:nbhd-part}

We now rephrase Proposition~\ref{prop:YZ-permutations}, Proposition~\ref{prop:nbhd-open}, and Proposition~\ref{prop:crvnbdweyl} in terms of both $BC$-partitions and $BKT$-partitions.

From Section~\ref{ss:BC} we see that the set of partitions $\lambda \in \BC(k,2n+1)$ with $\lambda_1<2n+1-k$, denoted by $\BC^\circ$ and which indexes Schubert varieties of the open $\Sp_{2n+1}$-orbit $\Xo$, is in bijection with the elements of $\BC(k,2n)$, which index Schubert varieties of the $\Y$-orbit. Explicitly, the bijection $\BC^\circ \to \BC(k,2n)$ is just the identity. Abusing notation we denote it by 
\[
    \Phi \colon \BC^\circ \to \BC(k,2n),
\]
as it is the `partitions' equivalent of the bijection $\Phi \colon \Wo \cap W^P \to W_\Y$ from Proposition~\ref{prop:YZ-permutations}.

Similarly, partitions $\lambda \in \BC(k,2n+1)$  indexing Schubert varieties of the closed $\Sp_{2n+1}$-orbit $\Xc$, i.e., those with $\lambda_1=2n+1-k$, are in bijection with the elements of $\BC(k-1,2n)$, which index Schubert varieties of the $\Xc$-orbit, and the bijection is as follows. If $\lambda \in \BC(k,2n+1)$ is such that $\lambda_1=2n+1-k$ then
\[
    \Phi_\Xc(\lambda) = (\lambda_2 \geq \ldots \geq \lambda_k).
\]

For partitions in $\BKT(k,2n+1)$ the situation is analogous; we obtain a bijection
\[
    \Phi \colon \BKT^\circ \to \BKT(k,2n),
\]
where $\BKT^\circ$ is the set of partitions $\alpha \in \BKT(k,2n+1)$ with $\alpha_1<2n+1-k$, which indexes Schubert varieties of $\Xo$. This bijection is simply the identity. We also have that partitions $\alpha \in \BKT(k,2n+1)$ with $\alpha_1=2n+1-k$, which index Schubert varieties of the $\Xc$-orbit, are in bijection with the elements of $\BKT(k-1,2n)$, and the bijection is given by
\[
    \Phi_\Xc(\alpha) = (\alpha_2+1 \geq \ldots \geq \alpha_k+1).
\]

Now Proposition~\ref{prop:YZ-permutations} can be reformulated.
\begin{prop}\label{prop:YZ-Part}
The Schubert varieties of $\IG(k,2n+1)$ are related to those of the closed $\Sp_{2n}$-orbits $\Y$ and $\Xc$ as follows:
\begin{itemize}
    \item $X(\lambda) = \pi_\Xc(p_\Y^{-1}(\Y(\Phi(\lambda))))$ if $\lambda \in \BC^\circ$ (and similarly for $\alpha \in \BKT^\circ$);
    \item $X(\lambda)=i_\Xc(\Xc(\Phi_\Xc(\lambda)))$ if $\lambda \in \BC(k,2n+1)$ is such that $\lambda_1=2n+1-k$ (and similarly for $\alpha \in \BKT(k,2n+1)$ with $\alpha_1=2n+1-k$).
\end{itemize}
\end{prop}

The main elements of the proof of the next proposition are in Subsection \ref{subs:proofofcrvnbdpart}.

\begin{prop} \label{prop:crvnbdpart}
Let $I \in \{\BC(k,2n+1),\BKT(k,2n+1) \}$ and let $\lambda \in I$. 
\begin{enumerate}
    \item If $X(\lambda) \cap \Xo \neq \emptyset$ then 
    \[
        \Gamma_d(X(\lambda))=X\left(\lambda^{\Oo(d)}\right). 
    \]
    \item If $X(\lambda) \subset \Xc$ then $\Gamma_d(X(\lambda)) = X\left(\lambda^{\OY(d)}\right) \cup X\left(\lambda^{\OZ(d)}\right)$. In some cases, $X\left(\lambda^{\OZ(d)}\right) \subset X\left(\lambda^{\OY(d)}\right)$.
\end{enumerate}
\end{prop}

\begin{proof}
Part (1) follows from Proposition~\ref{prop:nbhd-open} and Lemma~\ref{lem:Ocirc1}. Precisely, let $\lambda \in I$ be such that $X(\lambda) \cap \Xo \neq \emptyset$, and denote by $v \in \Wo \cap W^P$ be the corresponding Weyl group element, i.e., $\Psi_I(\lambda)=v$. By Proposition~\ref{prop:nbhd-open} we have $\Gamma_d(X(v))=X(v \cdot_k \Oo(d))$. For $d=1$, Lemma~\ref{lem:Ocirc1} implies that $\Psi_I(\lambda^{\Oo(1)})=v \cdot_k \Oo(1)$, hence $\Gamma_1(X(\lambda))=X\left(\lambda^{\Oo(1)}\right)$. The case $d>1$ follows by induction, given that $\Oo(d)=\Oo(d-1) \cdot_k \Oo(1)$.

For part (2), if $\lambda \in I$ is such that $X(\lambda) \subset \Xc$, denote by $u$ the corresponding Weyl group element, i.e., $\Psi_I(\lambda)=u$. By Proposition~\ref{prop:crvnbdweyl} we have $\Gamma_d(X(u))=X(u \cdot \OY(d)) \cup X(u \cdot \OZ(d))$. For $d=1$, Lemmas \ref{lem:O11Part} and \ref{lem:O11BKT} imply that $\Psi_I(\lambda^{\OY(1)}) = u \cdot \OY(1)$, and Lemma \ref{lem:OZ1} implies that $\Psi_I(\lambda^{\OZ(1)}) = u \cdot \OZ(1)$, hence the result. The case $d>1$ follows again by induction.
\end{proof}

\begin{cor} \label{cor:curvgreed}
Let $I \in \{\BC(k,2n+1),\BKT(k,2n+1) \}$. Let $\lambda \in I$ where $X(\lambda) \subset \Xc$ (equivalently, $\lambda_1=2n+1-k$). Then 
\[
    \lambda^{\OY(d)} = \left( \left(\lambda^{\OZ(d_1)}\right)^{\OY(1)}\right)^{\Oo(d_2)} \text{ where $d_1+d_2=d-1$.}
\]
\end{cor}

\begin{proof}
This follows immediately from Corollary~\ref{cor:square} and Lemmas~\ref{lem:Ocirc1}, \ref{lem:OZ1}, \ref{lem:O11Part}, and \ref{lem:O11BKT}.
\end{proof}

\section{Classification of Irreducible components of curve neighborhoods}

We will define two sets that will index those partitions in $\BC(k,2n+1)$ and $\BKT(k,2n+1)$ where the associated curve neighborhoods have two components.
We start by introducing notation for this section. 

Let $\ell^d_i(\lambda)=\# \{j>d \mid \lambda_j>i \}$. 
Explicitly, $\ell^d_i(\lambda)$ counts, among the $k-d$ last parts of $\lambda$, how many are larger than $i$.

\begin{defn}
We define:
\begin{align*}
    \Comp_{\BC}(d) &:= \{ \lambda \in \BC(k,2n+1) \mid \lambda_1 = 2n+1-k, \lambda_{d+1}-\ell_{d-1}^{d+1}(\lambda)-d=2(n+1-k) \}; \\
    \Comp_{\BKT}(d) &:= \{ \alpha \in \BKT(k,2n+1) \mid \alpha_1=2n+1-k, \alpha_2^{\OZ(d-1)}-\ell_{-1}^2(\alpha^{\OZ(d-1)})=2(n+1-k)\}.
\end{align*}
\end{defn}

Note that if $I \in \{\BC(k,2n+1), \BKT(k,2n+1)\}$ and $\lambda \not \in \Comp_I(d_1)$ for some $d_1 \geq 1$, then $\lambda \not \in \Comp_I(d)$ for any $d \geq d_1$. The explanation for this fact is geometrical. Namely, if $X(\lambda) \subset \Xc$ and $\Gamma_{d_1}(X(\lambda))$ has no irreducible component contained in $\Xc$, then 
\[ 
    \Gamma_d(X(\lambda)) = \Gamma_{d-d_1}(\Gamma_{d_1}(X(\lambda))) = \Gamma_{d-d_1}\left(X\left(\lambda^{\OY(d_1)}\right)\right) = X\left( \left(\lambda^{\OY(d_1)}\right)^{\Oo(d-d_1)}\right) 
\] 
is irreducible by Corollary \ref{cor:curvgreed}.

We now state the main theorem of this section.

\begin{thm} \label{thm:mainthmsec}
Let $I \in \{\BC(k,2n+1), \BKT(k,2n+1)\}$. If $\lambda \in I$ then
\[ 
    \Gamma_d(X(\lambda)) =  \begin{cases}
                                  X\left(\lambda^{\OY(d)}\right) \cup X\left(\lambda^{\OZ(d)}\right) & \text{ if $\lambda \in \Comp_I(d)$}\\ X\left(\lambda^{\OY(d)}\right) & \text{ if $\lambda_1=2n+1-k$ and $\lambda \notin \Comp_{I}(d)$} \\
                            X\left( \lambda^{\Oo(d)}\right) & \text{if $\lambda_1 < 2n+1-k$.}
                            \end{cases}
\]
Moreover, in the first case $(\lambda \in \Comp_I(d))$, the Schubert varieties $X\left(\lambda^{\OY(d)}\right)$ and  $X\left(\lambda^{\OZ(d)}\right)$ form two irreducible components.
\end{thm}

\begin{example} \label{exam:minthm}
Let $k=5$, $n=7$, and consider the following partitions in $\BC(5,15)$:
\[
    \lambda = (10,9,9,5), \mu = (9,8,8,3,1).
\]
We see that $\lambda \in \Comp_{\BC}(d)$ for $d=1,2$ but $\lambda \not \in \Comp_{\BC}(d)$ for $d \geq 3$. We compute
\[
    \lambda^{\OY(1)} = (8,8,4,2), \lambda^{\OZ(1)} = (10,8,4),
\]
which are indeed incomparable for inclusion, then
\[
    \lambda^{\OY(2)} = (7,3,1), \lambda^{\OZ(2)} = (10,3),
\]
which are still incomparable, and finally
\[
    \lambda^{\OY(3)} = (2) \subset \lambda^{\OZ(3)} = (10).
\]

For $\mu$ we obtain
\[
    \mu^{\Oo(1)} = (7,7,2), \mu^{\Oo(2)} = (6,1).
\]
\end{example}

To prove Theorem~\ref{thm:mainthmsec} we require several lemmas, which we now state and prove.

\begin{lemma} \label{lem:compadd}
If $d_2 \geq 1$ then the partition $\lambda$ is in $\Comp_{\BC}(d_1+d_2)$ if and only if $\lambda^{\OZ(d_1)} \in \Comp_{\BC}(d_2)$.
\end{lemma}

\begin{proof}
It follows from the recursive definition of $\lambda^{\OZ(d)} = \left( \lambda^{\OZ(d-1)} \right)^{\OZ(1)}$ that it is enough to prove the result for $d_1=1$. Recall that
\[
    \lambda^{\OZ(1)} = (2n+1-k \geq \lambda_3-1 \geq \dots \geq \lambda_{\ell(\lambda)}-1),
\]
hence $\lambda^{\OZ(1)}_{d_2+1}=\lambda_{d_2+2}-1$. Moreover, the number of parts of $\lambda^{\OZ(1)}$ which are at least equal to $d_2$ is one less than the number of parts of $\lambda$ which are at least equal to $d_2+1$. This translates as
\[
    \ell_{d_2-1}^{d_2+1}(\lambda^{\OZ(1)}) = \ell_{d_2}^{d_2+2}(\lambda),
\]
therefore
\[
    \lambda^{\OZ(1)}_{d_2+1} - \ell_{d_2-1}^{d_2+1}(\lambda^{\OZ(1)}) - d_2 = \lambda_{d_2+2} - \ell_{d_2}^{d_2+2}(\lambda) -(d_2+1).
\]
It is now immediate that $\lambda^{\OZ(1)} \in \Comp_{\BC}(d_2)$ if and only if $\lambda \in \Comp_{\BC}(d_2+1)$.
\end{proof}

The next result follows immediately from the recursive definition of $\Comp_{\BKT}(d)$.
\begin{lemma} \label{lem:compaddBKT}
If $d_2 \geq 1$ then the partition $\alpha$ is in $\Comp_{\BKT}(d_1+d_2)$ if and only if $\alpha^{\OZ(d_1)} \in \Comp_{\BKT}(d_2)$.
\end{lemma}

Lemma~\ref{lem:compsetcompat}, shows that $\Comp_{\BC}(d)$ and $\Comp_{\BKT}(d)$ index the same Schubert varieties.

\begin{lemma}\label{lem:compsetcompat}
The following equality holds
\[
    \left\{X(\lambda) \subset \IG(k,2n+1) \mid \lambda \in \Comp_{\BC}(d) \} = \{ X(\lambda) \subset \IG(k,2n+1) \mid \lambda \in \Comp_{\BKT}(d) \right\}.
\]
\end{lemma}

\begin{proof}
Let $X(\lambda) \subset \IG(k,2n+1)$ where $\lambda \in \Comp_{\BC}(d)$. By Lemma~\ref{lem:compadd} we have $\lambda^{\OZ(d-1)} \in \Comp_{\BC}(1)$.

Then $\mu := \lambda^{\OZ(d-1)}$ corresponds to a partition $\beta \in \BKT(k,2n+1)$, and $\beta=\alpha^{\OZ(d-1)}$ for some $\alpha \in \BKT(k,2n+1)$. Explicitly, $\alpha$ is such that $X(\alpha)=X(\lambda)$.

We have $\alpha^{\OZ(d-1)}_1=2n+1-k$, $\alpha^{\OZ(d-1)}_2 = \lambda_2^{\OZ(d-1)}-1$, and $\ell_{-1}^2(\alpha^{\OZ(d-1)})=\ell_{0}^{2}(\lambda^{\OZ(d-1)})$. So, $\mu_2^{\OZ(d-1)}-\ell_{-1}^2(\alpha^{\OZ(d-1)}) = 2(n+1-k)$ implies that $\alpha \in \Comp_{\BKT}(d)$. So $X(\lambda)=X(\alpha)$ and the set on the left side of the equality is a subset of the set on the right side of the equality. The reverse inclusion follows by reversing the argument, replacing $BC$ with $BKT$ and Lemma~\ref{lem:compadd} with Lemma~\ref{lem:compaddBKT}. The result follows.
\end{proof}

\begin{lemma}\label{lem:BCincl}
If $\lambda \in \BC(k,2n+1) \setminus \Comp_{\BC}(d)$ and $\lambda_1=2n+1-k$ then $X\left(\lambda^{\OZ(d)}\right) \subset X\left(\lambda^{\OY(d)}\right)$. In particular this implies that $\Gamma_d(X(\lambda))$ is irreducible.
\end{lemma}

\begin{proof} 
It is easy to show by induction on $d$ that the partition $\lambda^{\OZ(d)}$ is given by
\[
    \lambda^{\OZ(d)}=(2n+1-k \geq \lambda_{2+d}-d \geq \cdots \geq \lambda_{\ell_{d}^{d+2}(\lambda)+d+2} -d> \overbrace{ 0 \geq \cdots \geq 0}^{k-\ell_{d}^{d+2}(\lambda)-2}).
\]
To find an expression of $\lambda^{\OY(d)}$ we note that $\lambda^{\OY(d)}=\left(\lambda^{\OZ(d-1)}\right)^{\OY(1)}$. As before we find that
\[
    \lambda^{\OZ(d-1)}=(2n+1-k \geq \lambda_{2+(d-1)}-(d-1) \geq \cdots \geq \lambda_{\ell_{d-1}^{d+1}(\lambda)+d+1} -(d-1)> \overbrace{ 0 \geq \cdots \geq 0}^{k-\ell_{d-1}^{d+1}(\lambda)-2}).
\]
Therefore the 01-word corresponding to $\lambda^{\OZ(d-1)}$ takes the following form
\[ 
    1\overbrace{0\cdots \cdots \cdots 0}^{2n-k-\lambda_{d+1}+d}1w_1w_2\cdots w_j0 \overbrace{1 \cdots \cdots \cdots 1}^{k-\ell_{d-1}^{d+1}(\lambda)-2}0,
\]  
From Example~\ref{ex:BC_permissible} we deduce that
\[
    2n-k-\lambda_{d+1}+d \geq k-\ell_{d-1}^{d+1}(\lambda)-2,
\]
and since $\lambda \in \Comp_{\BC}(d)$ we know that the inequality is strict.

Therefore $\lambda^{\OZ(d-1)}$ is $m$-wingtip symmetric, where $m = k-\ell_{d-1}^{d+1}(\lambda)-2$. Moreover $D(\lambda^{\OZ(d-1)})(\bar m)=1$ corresponds to a vertical step at the end of Row $i=\ell_{d-1}^{d+1}(\lambda)+d+2$ of $\lambda^{\OZ(d-1)}$, and $\lambda^{\OZ(d-1)}_i=0$. Hence by definition
\[
    \lambda^{\OY(d)} = \left(\lambda^{\OZ(d-1)}\right)^{\OY(1)} = (\lambda_{2+(d-1)}-d \geq \lambda_{3+(d-1)}-d \geq \dots \geq \lambda_{\ell_{d-1}^{d+1}(\lambda)+d+1} -d \geq \overbrace{ 0 \geq \cdots \geq 0}^{k-\ell_{d-1}^{d+1}(\lambda)-1}).
\]
Therefore, all but the first part of $\lambda^{\OY(d)}$ and $\lambda^{\OZ(d)}$ are identical, and clearly
\[
    \lambda^{\OY(d)}_1 = \lambda_{d+1}-d < \lambda^{\OZ(d)}_1=2n+1-k,
\]
hence $\lambda^{\OY(d)} \subset \lambda^{\OZ(d)}$. 
\end{proof}

\begin{lemma}\label{lem:BCcomp}
Let $I \in \{\BC(k,2n+1), \BKT(k,2n+1)\}$. Let $\lambda \in I$ such that $X(\lambda) \subset \Xc$ (equivalently, $\lambda_1=2n+1-k$). Then
\[
    \Gamma_1(X(\lambda)) =  \begin{cases}
                                X \left( \lambda^{\OY(1)} \right) \cup X\left(\lambda^{\OZ(1)}\right) & \text{if $\lambda \in  \Comp_{I}(1)$} \\
                                X \left( \lambda^{\OY(1)} \right) & \text{if $\lambda \not \in \Comp_{I}(1)$.}
                            \end{cases}
\]
\end{lemma}

\begin{proof}
By Lemma \ref{lem:compsetcompat} it suffices to consider the set of partitions $\BC(k,2n+1)$. Let $\lambda \in \Comp_{\BC}(1)$ be $m$-wingtip symmetric. 

The $01$-word corresponding to $\lambda$ is of the form
\[ 
    1\overbrace{0\cdots \cdots \cdots 0}^{2n+1-k-\lambda_{2}}1w_1w_2\cdots w_j0 \overbrace{1 \cdots \cdots \cdots 1}^{k-2-\ell_{0}^{2}(\lambda)}0.
\]
Since
\[ 
    \lambda_{2}-\ell_{0}^{2}(\lambda)-1=2(n+1-k) \text{ (equivalently, } 2n+1-k-\lambda_2=k-2-\ell_0^2(\lambda)),
\] 
looking at the $01$-word we see that  $m \geq k-\ell_0^2(\lambda)$, hence we can rewrite the 01-word corresponding to $\lambda$ as
\[ 
    {\color{blue}1}\overbrace{0\cdots \cdots \cdots 0}^{2n+1-k-\lambda_{2}}{\color{red}1}w_1w_2\cdots w_{i-1}{\color{blue}0}\underbrace{w_{i+1} \cdots w_j{\color{red}0} \overbrace{1 \cdots \cdots \cdots 1}^{k-2-\ell_{0}^{2}(\lambda)}0}_{m}.
\]
The 01-word corresponding to $\lambda^{\OY(1)}$ is  
\[ 
    {\color{blue}0}\overbrace{0\cdots \cdots \cdots 0}^{2n+1-k-\lambda_{2}}1w_1w_2\cdots w_{i-1}{\color{blue}1}\underbrace{w_{i+1} \cdots w_j0 \overbrace{1 \cdots \cdots \cdots 1}^{k-2-\ell_{0}^{2}(\lambda)}0}_{m}
\] 
and the 01-word corresponding to $\lambda^{\OZ(1)}$ is 
\[ 
    1\overbrace{0\cdots \cdots \cdots 0}^{2n+1-k-\lambda_{2}}{\color{red}0}w_1w_2\cdots w_{i-1}0\underbrace{w_{i+1} \cdots w_j{\color{red}1} \overbrace{1 \cdots \cdots \cdots 1}^{k-2-\ell_{0}^{2}(\lambda)}0}_{m}.
\]
Translating the 01-words into partitions it follows that
\begin{enumerate}
    \item If $D(\lambda)(\overline m )= 1$ corresponds to a vertical step at the end of the $i$th row then 
    \[
        \lambda^{\OY(1)}:=(\lambda_2-1 \geq \cdots \geq \lambda_{i-1}-1 \geq \lambda_i\geq \lambda_i \geq  \lambda_{i+1} \geq \cdots \geq \lambda_{\ell(\lambda)} >0 \geq \cdots\geq 0 ); 
    \] 
    \item If $D(\lambda)(\overline m )= 0$ corresponds to a horizontal step at the bottom of the $i$th row and in the $j$th column then  
    \[
        \lambda^{\OY(1)}:=(\lambda_2-1\geq \cdots \geq \lambda_{i}-1 \geq j \geq \lambda_{i+1} \geq \cdots \geq \lambda_{\ell(\lambda)}>0 \geq \cdots \geq 0 ); 
    \]
\end{enumerate}
and  
\[
    \lambda^{\OZ(1)}=(2n+1-k \geq \lambda_{3}-1 \geq \cdots \geq \lambda_{\ell(\lambda)} -1 \geq  0 \geq \cdots \geq 0).
\]
As $\lambda^{\OY(1)}_1<\lambda^{\OZ(1)}_1$ and $\ell(\lambda^{\OY(1)})=\ell(\lambda)>\ell(\lambda^{\OY(1)})$, the two partitions are not comparable for inclusion and it follows that there are two irreducible curve neighborhood components.

The case where $\lambda \in \BC(k,2n+1) \setminus \Comp_{\BC}(1)$ follows from Lemma \ref{lem:BCincl}.
\end{proof}

\begin{thm} \label{thm:BKTcomp}
Let $I \in \{\BC(k,2n+1), \BKT(k,2n+1)\}$. Let $\lambda \in I$ such that $X(\lambda) \subset \Xc$ (equivalently, $\lambda_1=2n+1-k$). Then
\[
    \Gamma_d(X(\lambda)) =  \begin{cases}
                                X\left(\lambda^{\OY(d)}\right) \cup X\left(\lambda^{\OZ(d)}\right) & \text{$\lambda \in \Comp_{I}(d)$}\\
                                X\left(\lambda^{\OY(d)}\right) & \text{$\lambda \in I \setminus \Comp_{I}(d)$}
                            \end{cases}
\]
\end{thm}

\begin{proof}
By Lemma \ref{lem:compsetcompat} it suffices to consider the set of partitions $\BC(k,2n+1)$. By Lemma \ref{lem:compadd} we have that $\lambda \in \Comp_{\BC}(d)$ if and only if $\lambda^{\OZ(d-1)} \in \Comp_{\BC}(1)$. The result for $\lambda \in \Comp_{\BC}(d)$ follows from Lemma \ref{lem:BCcomp} since  $\Gamma_d(X(\lambda))=\Gamma_1(X(\lambda^{\OZ(d-1)}))$. The case where $\lambda \in \BC(k,2n+1) \setminus \Comp_{\BC}(d)$ follows from Lemma \ref{lem:BCincl}.
\end{proof}

\begin{figure}\label{fig:IG(5,13)}
\caption{Consider the case $k=5$ and $n=6$. Here we are starting from the Schubert point. Each edge is the $T$-fixed line ($d=1)$ in the moment graph of $\IG(5,13)$ that is adjacent to a maximal vertex. Observe the independence of minimal path length for the curve neighborhood component intersecting the open orbit. The ``Buch-Mihalcea" recursion that we write down should account for this independence.}
\begin{tikzpicture}[-,>=stealth',shorten >=1pt,auto,node distance=4cm,
                    thick,main node/.style={draw,font=\sffamily\tiny\bfseries}]

  \node[main node] (1) {$(\bar{6}<\bar{5}<\bar{4}<\bar{3}<\bar{2})$, $\emptyset$};
  \node[main node] (2) [below left of=1] {$(5<\bar{6}<\bar{4}<\bar{3}<\bar{2})$, $\tableau{6}{&{}&{}&{}&{}}$};
  \node[main node] (3) [red, below right of=1] {$(1<\bar{5}<\bar{4}<\bar{3}<\bar{2})$, $ \red \tableau{6}{&{}&{}&{}&{}&{}&{}&{}&{}\\ {}&&&&&&&&\\ {}&&&&&&&&\\ {}&&&&&&&&\\ {}&&&&&&&&}$};
   \node[main node] (4) [below of=2] {$(4<5<\bar{6}<\bar{3}<\bar{2})$, $\tableau{6}{&{}&{}&{}&{}&{}\\&{}&{}&{}&{}}$};
     \node[main node] (5) [red, below of=3] {$(1<5<\bar{4}<\bar{3}<\bar{2})$, $ \red \tableau{6}{&{}&{}&{}&{}&{}&{}&{}&{}\\ &{}&{}&{}&{}\\ {}&&&&&&&&\\ {}&&&&&&&&\\ {}&&&&&&&&}$};
      \node[main node] (6) [below of=4] {$(3<4<5<\bar{6}<\bar{2})$, $\tableau{6}{&{}&{}&{}&{}&{}&{}\\&{}&{}&{}&{}&{}\\&{}&{}&{}&{}}$};
\node[main node] (7) [red, below of=5] {$(1<4<5<\bar{3}<\bar{2})$, $ \red \tableau{6}{&{}&{}&{}&{}&{}&{}&{}&{}\\ &{}&{}&{}&{}&{}\\ &{}&{}&{}&{}\\ {}&&&&&&&&\\ {}&&&&&&&&}$};
 \node[main node] (8) [below of=6] {$(2<3<4<5<\bar{6})$, $\tableau{6}{&{}&{}&{}&{}&{}&{}&{}\\&{}&{}&{}&{}&{}&{}\\&{}&{}&{}&{}&{}\\&{}&{}&{}&{}}$};
 \node[main node] (9) [red, below of=7] {$(1<3<4<5<\bar{2})$, $ \red \tableau{6}{&{}&{}&{}&{}&{}&{}&{}&{}\\ &{}&{}&{}&{}&{}&{}\\ &{}&{}&{}&{}&{}\\ &{}&{}&{}&{}\\ {}&&&&&&&&}$};
  \node[main node] (10) [red, below right of=8] {$(1<2<3<4<5)$, $\red \tableau{6}{&{}&{}&{}&{}&{}&{}&{}&{}\\&{}&{}&{}&{}&{}{}&{}&{}\\&{}&{}&{}&{}{}&{}&{}\\&{}&{}&{}{}&{}&{}\\&{}&{}&{}&{}}$};

  \path[every node/.style={font=\sffamily\large}]
       (1) edge node {$\cdot_k O^\circ(1)$}  (2)
       (1) edge node {$\cdot \OY(1)$}  (3)
       (2) edge node {$\cdot_k O^\circ(1) $}  (4)
       (3) edge [red] node {$\cdot \OZ(1) $}  (5)
       (2) edge  node {$\cdot \OY(1)$}  (5)
       (4) edge  node {$\cdot_k O^\circ(1) $}  (6)
       (4) edge  node {$\cdot \OY(1)$}  (7)
       (5) edge [red] node {$ \cdot \OZ(1)$}  (7)
       (6) edge node {$\cdot_k O^\circ(1) $}  (8)
       (6) edge node {$ \cdot \OY(1)$}  (9)
       (7) edge [red] node {$ \cdot \OZ(1)$}  (9)
       (8) edge node {$ \cdot \OY(1)$}  (10)
       (9) edge [red] node {$ \cdot \OZ(1)$}  (10);
\end{tikzpicture}
\end{figure}
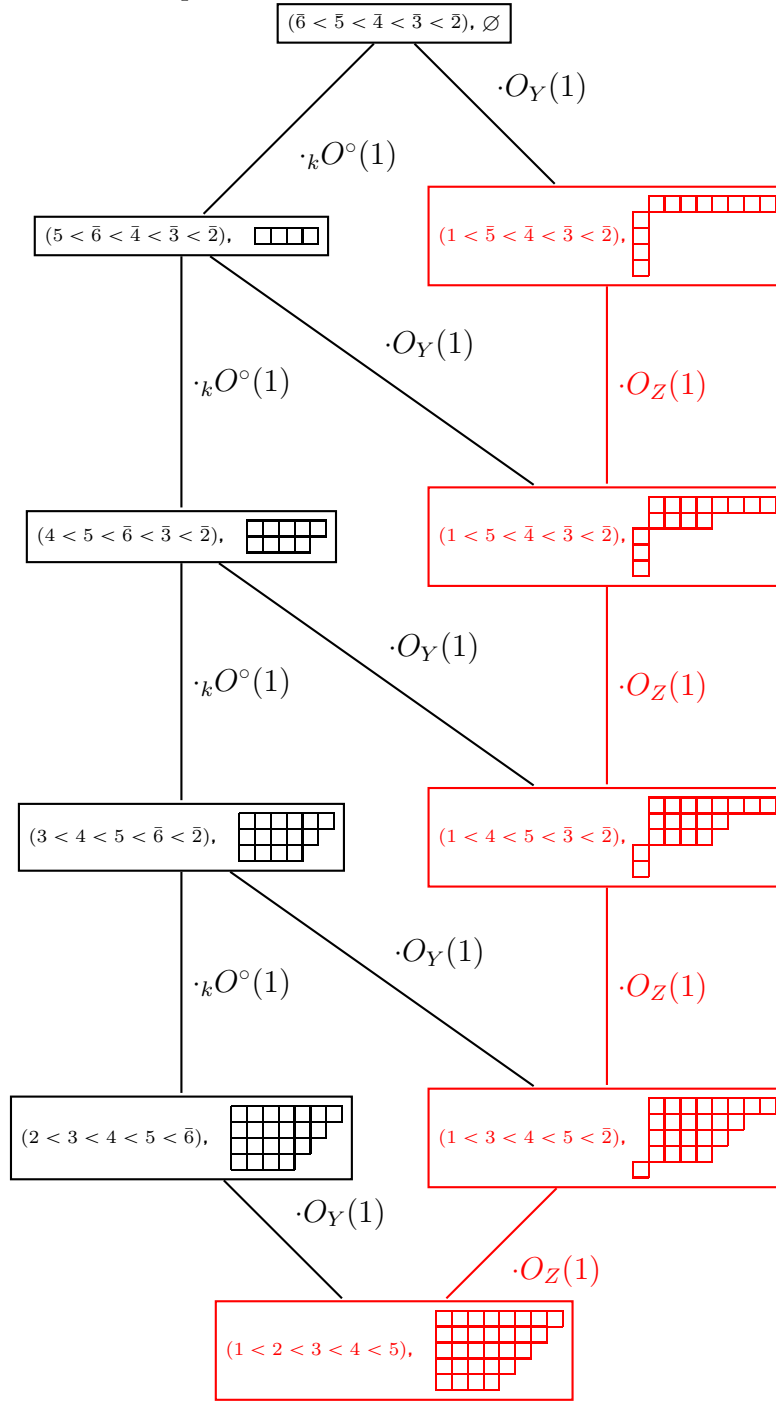

\section{Technical Results}

In this section we regroup some results. Two of the results are from literature and the other are proofs that are omitted from their respective sections.

\subsection{Curve neighborhood calculations for the even symplectic Grassmannian}
Next we will define notation for the statements of the upcoming lemmas. Let $\Gamma_d(Y(\mu))$, for $\mu \in \BC(k,2n)$, denote the degree $d$ curve neighborhood of $X(\mu)$ in $Y \cong \IG(k,2n)$. We define $\Gamma_d(Y(\beta))$ similarly for $\beta \in \BKT(k,2n)$. The next two lemmas follow from \cite[Theorem 5.18]{ShiflerWithrow}.

\begin{lemma} \label{lem:Lambdaline}
Let $\beta \in  \BKT(k,2n)$. Then $\Gamma_1(Y(\beta))=Y(\beta^1)$ where $\beta^1$ is given as follows.
\begin{enumerate}
    \item If $\beta_1+\beta_j > 2(n-k)+j-1$ for all $2 \leq j \leq k$ then define 
    \[ 
        \beta^1=(\beta_2 \geq \beta_3 \geq \cdots \geq \beta_k \geq 0) \in \BKT(k,2n); 
    \] 
    \item Otherwise, find the smallest $j$ such that $\beta_1+\beta_j \leq 2(n-k)+j-1$. Define 
    \[ 
        \beta^1=(\beta_2 \geq \beta_3 \geq \cdots \geq \beta_{j-1} \geq \beta_j-1 \geq \cdots \geq \beta_{k}-1 \geq 0) \in  \BKT(k,2n) 
    \] 
    where $-1$'s are replaced by 0.
\end{enumerate}
\end{lemma}

\begin{lemma}\label{lem:Partline}
Let $\lambda \in  \BC(k,2n+2)$. Then $\Gamma_1(\Xev(\lambda)) = \Xev(\lambda^1)$ where $\lambda^1$ is given as follows. 
\[ 
    \lambda^1=(\lambda_2-1, \cdots, \lambda_{\ell(\lambda)}-1,0,\cdots,0). 
\]
\end{lemma}

\subsection{Proof of Theorem \ref{thm:ZtoYlines}.} 
\label{subs:proofofthm:ZtoYlines}

\begin{proof}
For (1) observe that $w(1)=1$ and $v(1)>1$ since $X(w) \subset \Xc$ and $X(v) \cap \Xo \neq \emptyset$. We claim that $\alpha=t_1 \pm t_{j}$ for some $k+1 \leq j \leq n+1$. Indeed, if $j\leq k$ then $ws_{t_1+t_{j}}(j)=\bar{1}$, hence $v$ would not be in $\Wodd$, and $ws_{t_1-t_{j}}W_P=wW_P$. It directly follows that $\deg C(w,v)=1$ since  
\[
    t_1-t_j=(t_1-t_2)+\cdots+{\color{blue} 1}(t_k-t_{k+1})+\cdots+(t_{j-1}-t_{j}) 
\] 
and 
\[
    t_1+t_j=(t_1-t_2)+\cdots+{\color{blue} 1}(t_k-t_{k+1})+\cdots+(t_{j-1}-t_{j})+2(t_{j}-t_{j+1})+\cdots+2t_{n+1}. 
\]

Next we prove (2). Let $w=(1<w(2)<\cdots<w(k))$. There are two cases for $v$:
\begin{enumerate}[(a)]
    \item there exists an $s$ such that $v$ is the permutation $v=(w(2)<w(3)<\cdots<w(j)<s<w(j+1)<\cdots <w(k)) \in W^P$;
    \item there exists an $s$ such that $v$ is the permutation $v=(w(2)<w(3)<\cdots <w(k)<s) \in W^P$.
\end{enumerate}
Let $w,v \in W^P$ correspond to $\alpha,\beta \in \BKT(k,2n+1)$.

For case (a) we break up the proof into two parts, depending on whether $s \geq k$ or $s<k$. If $s \geq k$, as $w(1)=1$ it immediately follows that $\alpha_{i+1}=\beta_{i}$ for $ 1 \leq i \leq j-1$. Moreover, for $j+1 \leq i \leq k$ we have $\beta_i \in \{\alpha_i,\alpha_i+1\}$, and $\beta_j \leq 2n+2-k-s+j-1$. 
From this we deduce following chain of inequalities.
\begin{align*}
    |\alpha|-|\beta| &= \left( (2n+1-k)+ \sum_{i=2}^{j} \alpha_i+\sum_{i=j+1}^{k} \alpha_i \right)-\left( \sum_{i=1}^{j-1} \beta_i+ \beta_j + \sum_{i=j+1}^{k} \beta_i\right)\\
                    &= (2n+1-k)+\left(\sum_{i=2}^{j} \alpha_i-\sum_{i=1}^{j-1} \beta_i \right)+\left(\sum_{i=j+1}^{k} \alpha_i- \sum_{i=j+1}^{k} \beta_i\right)-\beta_j\\
                    &\geq (2n+1-k)+0+(-(k-j))-((2n+2-k-s+j-1))\\
                    &\geq s-k.
\end{align*}
If $s<k$ then we have that 
\[
    \sum_{i=1}^j \beta_i = \sum_{i=1}^{j} \alpha_i-s+1
\]
and 
\[
    \sum_{i=j+1}^k \beta_i \leq \sum_{i=j+1}^{k} \alpha_i+s-1.
\]
Thus $|\beta| \leq |\alpha|$, hence $\ell(v)>\ell(w)$.

Similarly for case (b), as $w(1)=1$ we get $\alpha_{i+1}=\beta_{i}$ for $1 \leq i \leq k-1$, and $\beta_k \leq 2n+1-s$. This time the chain of inequalities is as follows.
\begin{align*}
    |\alpha|-|\beta| &= \left( (2n+1-k)+ \sum_{i=2}^{k} \alpha_i \right)-\left( \sum_{i=1}^{k-1} \beta_i+ \beta_k\right)\\
                    &= (2n+1-k)+\left(\sum_{i=2}^{k} \alpha_i-\sum_{i=1}^{k-1} \beta_i \right)-\beta_k\\
                    &\geq (2n+1-k)+0-(2n+1-s)\\
                    &\geq s-k.
\end{align*}
Since $s \geq k$ in case (b), the result follows.
\end{proof}

\subsection{Lemmas for Theorem \ref{prop:crvnbdpart}} \label{subs:proofofcrvnbdpart}

\begin{lemma} \label{lem:Ocirc1}
Let $I \in \{\BC(k,2n+1),\BKT(k,2n+1) \}$. Let $\lambda \in I$ be such that $\lambda_1<2n+1-k$, and denote by $v \in \Wo$ the Weyl group element corresponding to $\lambda$. Then $\lambda^{\Oo(1)}$ corresponds to the Weyl group element $v \cdot_k \Oo(1)$.
\end{lemma}

\begin{proof}
From Proposition~\ref{prop:nbhd-open} we know that $\Gamma_1(X(\lambda))=\Gamma_1(X(v))=X(v \cdot_k \Oo(1))$. By Proposition \ref{prop:YZ-Part} we also obtain that
\[
    X(v \cdot_k \Oo(1)) = \pi_\Xc(p_\Y^{-1}(Y(\Phi(v \cdot_k \Oo(1))))).
\]
Recall that $\Oo(1) := s_1 \cdots s_{n+1} \cdots s_1(=s_{2t_1})$. Our strategy is to apply Proposition~\ref{prop:bijY2W}, letting $w=s_{i_1} \dots s_{i_r}$ be a particular representation of $\Oo(1)$ as product of simple reflections. As $s_k^2=id$ and $s_ks_i=s_is_k$ for all $k+2 \leq i \leq n+1$, the following equality holds:
\begin{align*}
    \Oo(1) &= s_1\ldots s_k s_{k+1} s_k s_k s_{k+2} \ldots s_{n+1} \ldots s_{k+2}s_{k+1}s_k \ldots s_1 \\
            &= s_1\ldots s_k s_{k+1} s_k s_{k+2} \ldots s_{n+1} \ldots s_{k+2} s_ks_{k+1}s_k \ldots s_1. 
\end{align*}
Moreover, it is an easy exercise to show that the modified Hecke products of $v$ with either either representation of $\Oo(1)$ are equal, that is,
\[
    v \cdot_k s_1 \ldots s_{n+1} \ldots s_1=v \cdot_k s_1\ldots s_k s_{k+1} s_k s_{k+2} \ldots s_{n+1} \ldots s_{k+2}s_ks_{k+1}s_k \ldots s_1.
\]
We now apply Proposition~\ref{prop:bijY2W}:
\begin{align*}
    \Phi((v \cdot_k \Oo(1))W_P) &= \Phi((v \cdot_k s_1\ldots s_k s_{k+1} s_k s_{k+2} \ldots s_{n+1} \ldots s_{k+2}s_ks_{k+1}s_k \ldots s_1)W_P)\\
                    &= \Phi(v \cdot_k \psi(s_1)  \ldots \psi(s_{n}) \ldots \psi(s_1))W_P)\\
                    &= (\Phi(vW_P) \cdot \overbrace{s_1 \ldots s_n \ldots s_1}^{s_{2t_1} \text{ in  $\IG(k,2n)$}})W_Y,
\end{align*}
hence by Proposition \ref{prop:zd} and Corollary \ref{cor:evest}:
\[
    \Gamma_1(X(\lambda)) = X((v \cdot_k \Oo(1))W_P) = \pi_\Xc(p_\Y^{-1}(\Gamma_1(Y(\Phi(vW_P))))).
\]
Rewriting in terms of partitions, as $\Phi(vW_P)$ corresponds to the same partition $\lambda$ as $vW_P$ the right-hand side becomes $\pi_\Xc(p_\Y^{-1}(\Gamma_1(Y(\lambda))))$. 
Moreover by Lemma \ref{lem:Lambdaline} and Lemma \ref{lem:Partline} we obtain that $\Gamma_1(Y(\lambda))=Y(\Phi(\lambda)^1)$. It follows from the definitions of $\lambda^{\Oo(1)}$ and $\lambda^1$ that
\[
    \Phi(\lambda)^1=\Phi(\lambda^{\Oo(1)}),
\]
hence $\Gamma_1(X(\lambda))=X(\lambda^{\Oo(1)})$ as claimed.
\end{proof}

\begin{lemma} \label{lem:OZ1}
Let $I \in \{ \BC(k,2n+1),\BKT(k,2n+1) \}$. Let $\lambda \in I$ be such that $\lambda_1=2n+1-k$, and denote by $v \in \Wodd \cap W^P$ the Weyl group element corresponding to $\lambda$. Then $\lambda^{\OZ(1)}$ corresponds to the Weyl group element $v \cdot_k \OZ(1)$.
\end{lemma}

\begin{proof}
Recall that $\OZ(1)=s_2 \dots s_{n+1} \dots s_2$. Then by Proposition \ref{prop:bijZ} we have that
\begin{align*}
    \Phi_\Xc(v \cdot O_Z(1)) &= \Phi_\Xc((v \cdot s_2 \dots s_{n+1} \dots s_2)W_P)\\
                        &= (\Phi_\Xc(vW_P) \cdot \overbrace{s_1 \dots s_n \dots s_1}^{s_{2t_1} \text{ in } \IG(k-1,2n)})W_Z.
\end{align*}
The left-hand side corresponds to the partition $\Phi_\Xc(\lambda)^1$ by Proposition \ref{prop:zd} and Corollary \ref{cor:evest}, while Lemma \ref{lem:Lambdaline} and Lemma \ref{lem:Partline} tell us that this partition indexes the curve neighborhood $\Gamma_1(Z(\lambda_2,\dots,\lambda_k))$ for $\lambda \in \BC(k,2n+1)$ (respectively, $\Gamma_1(Z(\alpha_2+1,\dots,\alpha_k+1))$ for $\alpha \in \BKT(k,2n+1)$).
By definition of $\lambda^{\OZ(1)}$ and $\lambda^1$ it is then clear that
\[
    \Phi_\Xc(\lambda)^1 = \Phi_\Xc(\lambda^{\OZ(1)}),
\]
hence $\lambda^{\OZ(1)}$ indexes the closed orbit component of $\Gamma_1(\lambda)$ as claimed.
\end{proof}

\begin{lemma} \label{lem:O11Part}
Let $\lambda \in \BC(k,2n+1)$ be such that $\lambda_1=2n+1-k$, and denote by $w \in \Wodd \cap W^P$ the Weyl group element corresponding to $\lambda$. Then $\lambda^{\OY(1)}$  corresponds to $w \cdot \OY(1) \in W$.
\end{lemma}

\begin{proof}
Let $\lambda \in \BC(k,2n+1)$ be $m$-wingtip symmetric. Denote by $\lambda^1$ the $BC$-partition corresponding to $w \cdot \OY(1)$, see Lemma \ref{lem:comp}, and by $D(\lambda^1)$ the associated 01-word. We will show that $D(\lambda^1)=D(\lambda^{\OY(1)})$. 

First suppose that $D(\lambda)(\overline{m})=0$ corresponds to a horizontal step at the bottom of the $i$th row and in the $j$th column. Note that we cannot have $D(\lambda)(\overline{m+1})=1$, otherwise by definition of $\BC$-partitions we would need to have $D(\lambda)(m+1)=0$ and $\lambda$ would be $(m+1)$-wingtip symmetric and not $m$-wingtip symmetric. Therefore
\begin{align*}
    D(\lambda) &= \overbrace{{\color{blue}1} D(\lambda)(2) \cdots  D(\lambda)(\overline{m+2}){\color{blue}0}}^{ \text{first $2n+2-m$ characters} }0 D(\lambda)(\overline{m-1}) \cdots D(\lambda)(\overline{2})0;\\
    D(\lambda^1) &= \overbrace{{\color{blue}0} D(\lambda)(2) \cdots  D(\lambda)(\overline{m+2}){\color{blue}1}}^{ \text{first $2n+2-m$ characters} }0 D(\lambda)(\overline{m-1}) \cdots D(\lambda)(\overline{1}).
\end{align*}
Observe that the first $2n+2-m$ characters correspond to taking the curve neighborhood in $\mathrm{Gr}(i,2n+2-m)$. Indeed, taking the degree one curve neighborhood in $\mathrm{Gr}(i,2n+2-m)$ corresponds to deleting the top row, as well as a box in each of the subsequent $i-1$ rows. Moreover the last rows of $\lambda$, corresponding to the last $m$ characters of the $01$-word $D(\lambda)$, remain the same in $\lambda^1$. The introduction in $D(\lambda^1)$ of the character 1 in the $\overline{m+1}$ position forces $\lambda^1_i=j$.

Now suppose instead that $D(\lambda)(\overline{m})=1$ corresponds to a vertical step at the end of the $i$th row of $\lambda$. As before we must have $D(\lambda)(\overline{m+1})=0$, otherwise $\lambda$ would not be $m$-wingtip symmetric. So we have that 
\begin{align*}
    D(\lambda) &= \overbrace{{\color{blue}1} D(\lambda)(2) \cdots  D(\lambda)(\overline{m+2}){\color{blue}0}}^{ \text{first $2n+2-m$ characters} }1 D(\lambda)(\overline{m-1}) \cdots D(\lambda)(\overline{2})0;\\
    D(\lambda^1) &= \overbrace{{\color{blue}0} D(\lambda)(2) \cdots  D(\lambda)(\overline{m+2}){\color{blue}1}}^{ \text{first $2n+2-m$ characters} }1 D(\lambda)(\overline{m-1}) \cdots D(\lambda)(\overline{1}).
\end{align*}
The first $2n+2-m$ characters correspond to curve neighborhoods in $Gr(i-1,2n-m)$, deleting the top row and one box out of the subsequent $i-1$ rows. The last rows of $\lambda$, corresponding to the last $m$ characters of $D(\lambda)$, remain the same. The introduction of the character 1 in the $\overline{m+1}$ position forces $\lambda^1_i=\lambda_i$.

Therefore in either case, $D(\lambda)^1=D(\lambda^{\OY(1)})$, and the result follows.
\end{proof}

\begin{lemma} \label{lem:partBKTa} 
Let $w$ be an element of $\Wodd \cap W^P$ with $w(1)=1$ and $\alpha \in \BKT(k,2n+1)$ be the corresponding $\BKT$-partition. We consider the following cases for $w$:
\begin{enumerate} 
    \item If $w=(1<a_2<\dots<a_k)$ is such that $a_k \leq n+1$, then $\alpha_k \geq 0$;
    \item If $w=(1<a_2 < \dots <a_r < \bar a_k < \dots < \bar a_{r+1})$ is such that $\{2, 3, \cdots, a_{r+1}\} \nsubset \{ a_2, \dots,a_k\}$, then $\alpha_k \geq 0$;
    \item If $w=(1<a_2 < \dots <a_r < \bar a_k < \dots  < \bar a_{r+1}) $ is such that $\{2, 3, \dots, a_k\} \subset \{ a_2, \dots,a_k\}$, then $\alpha_{r} \geq 0$ and $\alpha_{r+1}=-1$;
    \item If $w=(1<a_2 < \dots <a_r < \bar a_k < \dots < \bar a_{t+1}<\bar a_t< \dots < \bar a_{r+1}) $ is such that there exists a largest index $t<k$ such that $\{2, 3, \dots, a_t\} \subset \{ a_2, \cdots,a_k\}$, denote by $s$ the index such that $w(s-1)=\bar a_{t+1}$, $w(s)=\bar a_t$. Then $\alpha_{s-1} \geq 0$ and $\alpha_s=-1$.
\end{enumerate}
\end{lemma}


\begin{proof} We will prove the four statements next.
\begin{enumerate}
    \item Since $w(j) \leq n+1$ for any $j \leq k$ in this case, we get that $\alpha_k=2n+2-k-w(k) \geq 0$ as claimed.
    \item We have
        \[
            \alpha_k=a_{r+1}-r-2+\#\{ 2 \leq i \leq r \mid a_i>a_{r+1} \}.
         \]
        Let $1 \leq s \leq r$ be the largest index such that $a_s<a_{r+1}$. Then $\alpha_k=a_{r+1}-s-2$. As $a_{r+1}$ is larger than $1,a_2,\dots,a_s$, which are all distinct elements in $\{1,2,3,\dots,n+1\}$, we deduce that $a_{r+1} \geq s+1$, with equality only possible if $a_i=i$ for $1 \leq i \leq s$.
        However, if that equality holds then $\{2,3,\dots,a_{r+1}\}=\{2,3,\dots,s+1\}= \{a_2,\dots,a_s,a_{r+1}\}$, contradicting our assumption that $\{2, 3, \cdots, a_{r+1}\} \nsubset \{ a_2, \dots,a_k\}$. Therefore $a_{r+1} \geq s+2$, hence $\alpha_k \geq 0$ as claimed.
    \item We have
        \begin{align*}
            \alpha_r &= 2n+2-k-a_r \geq 0 \\
            \alpha_{r+1} &= a_k-k-1+ \#\{2 \leq i \leq r \mid a_i>a_k\}. 
        \end{align*}
        Let $1 \leq s \leq r$ be the largest index such that $a_s<a_k$. Then $\alpha_{r+1}=a_k-k-1+(r-s)$. Furthermore we have that $s+k-r=a_k$ since $a_i \leq a_k$ for all $r+1 \leq i \leq k$. Thus, $\alpha_{r+1}=-1$.
        
    \item We begin with
        \[
            \alpha_s = a_t-t-1+\#\{2 \leq i \leq r \mid a_i>a_{t}\}.
        \]
        Let $1 \leq j \leq r$ be the largest index such that $a_{j}<a_{t}$. Then 
        \[
            \alpha_{s}=a_{t}-t-1+(r-j).
        \] 
        We also have that $a_t=j+(t-r)$ since $a_i \leq a_t$ for all $r+1 \leq i \leq t$. Thus, $\alpha_{s}=-1$. Next we consider
        \[
            \alpha_{s-1} = a_{t+1}-t-2+\#\{2 \leq i \leq r \mid a_i>a_{t+1} \}.
        \]
        Let $ j'=\#\{2 \leq i \leq r \mid a_t < a_i<a_{t+1}-1 \}$. Then $a_{t+1} \geq a_t+j'+2$ and 
        \[
            \alpha_{s-1} = a_{t+1}-t-2+(r-j'-j).
        \] 
        Then
        \begin{eqnarray*}
            \alpha_{s-1} &=& a_{t+1}-t-2+(r-j'-j)\\
                         &\geq & (a_t+j'+2)-t-2+(r-j'-j)\\
                         &=&a_t-t+r-j\\
                         &=&0.
        \end{eqnarray*}
        The result follows.

 \end{enumerate}
 
\end{proof}

\begin{lemma}\label{lem:partBKTb}
Let $w$ be an element of $\Wodd \cap W^P$ with $w(1)>1$ and $\alpha \in \BKT(k,2n+1)$ be the corresponding $\BKT$-partition. We consider the following cases for $w$:
\begin{enumerate}
     \item If $w=(a_1<a_2<\dots<a_k)$ is such that $a_k \leq n+1$, then $\alpha_k \geq 1$;
    \item If $w=(a_1<a_2 < \dots <a_r < \bar a_k < \dots < \bar a_{r+1})$ is such that $\{2, 3, \cdots, a_{r+1}\} \nsubset \{a_1, a_2, \dots,a_k\}$, then $\alpha_k \geq 1$;
    \item If $w=(a_1<a_2 < \dots <a_r < \bar a_k < \dots  < \bar a_{r+1}) $ is such that $\{2, 3, \dots, a_k\} \subset \{ a_1,a_2, \dots,a_k\}$, then $\alpha_{r+1}=0$, and $\alpha_k=0$;
    
    \item If $w=(a_1 < \dots <a_r < \bar a_k < \dots < \bar a_{t+1}<\bar a_t< \cdots < \bar a_{r+2}< \bar a_{r+1})$ and $t$ is the largest index such that $\{2,3, \cdots, a_t\} \subset \{ a_1, \cdots,a_k\}$. Choose the index $s$ such that $w(s)=\bar a_t$. If $w\mapsto \alpha \in \BKT(k,2n+1)$ then $\alpha_{s} = 0$ and $\alpha_k=0$.
\end{enumerate}
\end{lemma}

\begin{proof} We prove the four parts next.
\begin{enumerate}
    \item First observe that if $k=n+1$ then $a_1=1$. Thus we must have that $k<n+1$. It follows that $\alpha_k=2n+2-k-w(k) \geq 1$ as claimed.
    \item We have
        \[
            \alpha_k=a_{r+1}-r-2+\#\{ 1 \leq i \leq r \mid a_i>a_{r+1} \}.
        \]
        Let $1 \leq s \leq r$ be the largest index such that $a_s<a_{r+1}$. Then $\alpha_k=a_{r+1}-s-2$. As $a_{r+1}$ is larger than $a_1,a_2,\dots,a_s$, which are all distinct elements in $\{2,3,\dots,n+1\}$, we deduce that $a_{r+1} \geq s+2$, with equality only possible if $a_i=i+1$ for $1 \leq i \leq s$.
        However, if that equality holds then $\{2,3,\dots,a_{r+1}\}=\{2,3,\dots,s+2\}= \{a_2,\dots,a_s,a_{r+1}\}$, contradicting our assumption that $\{2, 3, \cdots, a_{r+1}\} \nsubset \{ a_2, \dots,a_k\}$. Therefore $a_{r+1} \geq s+3$, hence $\alpha_{k} \geq 1$ as claimed.
        
        \item We begin with
        \begin{align*} 
            \alpha_{r+1} &= a_k-k-1+ \#\{1 \leq i \leq r \mid a_i>a_k\}.
        \end{align*}
         Let $1 \leq s \leq r$ be the largest index such that $a_s<a_k$. Then $\alpha_{r+1}=a_k-k-1+(r-s)$. Furthermore we have that $s+k-r=a_k-1$ since $a_i \leq a_k$ for all $r+1 \leq i \leq k$. Thus, $\alpha_{r+1}=0$.
        
        Next we have that
        \[\alpha_k=a_{r+1}-r-2+\#\{1 \leq i \leq r: a_i>a_{r+1} \}\]
         Let $1 \leq s \leq r$ be the largest index such that $a_s<a_{r+1}$. Then $\alpha_{k}=a_{r+1}-s-2$. Furthermore, $a_i=i+1$ for $1 \leq i \leq s$. Thus $a_{r+1}=s+2$. We conclude that $\alpha_k=0$.

 \item We begin with
\begin{eqnarray*}
\alpha_{s}&=&a_{t}-t-1+\#\{1 \leq i \leq r :a_i>a_{t}\}.
\end{eqnarray*}
Let $1 \leq j \leq r$ be the largest index such that $a_{j}<a_{t}$. Then \[\alpha_{s}=a_{t}-t-1+(r-j).\] We also have that $a_t=j+(t-r)+1$ since $a_i \leq a_t$ for all $r+1 \leq i \leq t$. Thus, $\alpha_{s}=0$.
 
Next we consider
\begin{eqnarray*}
\alpha_k=a_{r+1}-r-2+\#\{1 \leq i \leq r: a_i>a_{r+1} \}.
\end{eqnarray*}
 
 Let $1 \leq s \leq r$ be the largest index such that $a_s<a_{r+1}$. Then $\alpha_k=a_{r+1}-s-2$. As $a_{r+1}$ is larger than $a_2,\dots,a_s$, which are all distinct elements in $\{2,3,\dots,n+1\}$, we deduce that $a_{r+1} \geq s+2$. Hence $\alpha_k \geq 0$. Since $\alpha_s \geq \alpha_k$ we then conclude that $\alpha_k =0$.

 \end{enumerate}
  The result follows.
\end{proof}

\begin{lemma} \label{lem:O11BKT}
Let $\alpha \in \BKT(k,2n+1)$ be such that $\alpha_1=2n+1-k$, and let $w \in \Wodd \cap W^P$ be the Weyl group element corresponding to $\alpha$. Then $\alpha^{\OY(1)}$ corresponds to $w \cdot \OY(1)$.
\end{lemma}

\begin{proof}
Let $w=(1<a_2 < \dots <a_r < \bar a_k < \dots < \bar a_{r+1}) \in \Wodd \cap W^P$.  Then by Lemma \ref{lem:comp}
\begin{align*}
    w \cdot \OY(1) &=(w(2)<\cdots<\bar j_\Y<\cdots<w(k)), 
\end{align*}
where 
\begin{align*}
    j_\Y &=\min \{2, \dots, n+1\} \setminus \{a_2,\dots,a_k\},
\end{align*}
Recall that $\bar j_\Y$ is possibly smaller than $w(2),w(3)$ or larger than $w(k)$. Denote by $s$ the integer such that $(w \cdot O_Y(1))(s)=\bar{j}_Y$.

By Lemma \ref{lem:partBKTa} we deduce that $\alpha_{s-1} \geq 0$ and $\alpha_s=-1$. By definition of $w \cdot \OY(1)$,
\[
    \#\{i<j+1 \mid w(i)+w(j+1)<2n+3 \} = \#\{i<j \mid (w\cdot O_Y(1))(i)+(w \cdot O_Y(1))(j)<2n+3 \} 
\] 
for $1 \leq j \leq s-1$. It follows that $\alpha^{O_Y(1)}_j=\alpha_{j+1} \geq 0$ for $1 \leq j \leq s-1$. By Lemma \ref{lem:partBKTb}, we have that $\alpha_j=0$ for $s \leq j \leq k$. The result follows.
\end{proof}

\bibliographystyle{alphabetic}
\bibliography{bibliography}

\end{document}